\documentclass[a4paper,12pt]{article}
\usepackage{a4wide}
\usepackage{amsmath}
\usepackage{amsfonts}
\usepackage{amssymb}
\usepackage{amsthm}
\usepackage{amscd}
\usepackage{mathrsfs}
\usepackage{verbatim}
\usepackage[all]{xy}
\usepackage[]{graphicx}
\usepackage{color}
\usepackage[bookmarks=true]{hyperref}
\usepackage{enumerate}

\numberwithin{equation}{section}

\DeclareMathOperator{\inc}{in}

\newcommand{\ext}{E}
\newcommand{\loday}{\Lambda}
\newcommand{\Loday}{L_*}
\newcommand{\SLoday}{\mathcal{L}}
\newcommand{\morava}[1]{A\langle{#1}\rangle}
\newcommand{\B}{B}

\newcommand{\polIdeal}{\mathfrak{p}}

\DeclareMathOperator{\sk}{sk}

\newcommand{\HF}{{H\bb{F}_p}}
\newcommand{\HZ}{{H\bb{Z}}}

\newcommand{\simp}{\Delta}

\newcommand{\sphereproj}{g}
\newcommand{\attach}{f}
\newcommand{\torusIso}{\alpha}

\newcommand{\za}{a}
\newcommand{\zb}{b}

\newcommand{\zi}{i}

\newcommand{\zk}{k}

\newcommand{\zu}{u}
\newcommand{\zv}{v}

\newcommand{\ord}[1]{{[#1]}}
\newcommand{\uord}[1]{{\bf #1}}

\newcommand{\basis}{\mathcal{G}}
\newcommand{\obasis}{\overline{\basis}}

\newcommand{\sBar}{\varrho}
\newcommand{\sBok}{\sigma}
\newcommand{\sOper}{\sigma}
\newcommand{\sOperSm}{\widehat{\sigma}}
\newcommand{\sCircleInc}{\omega}
\newcommand{\sSmashInc}{\widehat{\omega}}

\newcommand{\oxi}{\overline{\xi}}
\newcommand{\otau}{\overline{\tau}}

\newcommand{\n}{{\bf n}}

\newcommand{\mult}{\phi}
\newcommand{\unit}{\eta}
\newcommand{\comult}{\psi}
\newcommand{\counit}{\epsilon}
\newcommand{\redcomult}{\widetilde{\comult}}

\newcommand{\twist}{\tau}

\newcommand{\pinch}{\comult}
\newcommand{\fold}{\nabla}

\newcommand{\type}[1]{type $#1$}
\newcommand{\divByP}[1]{\frac{#1}{p}}

\newcommand{\nat}{\bb{N}} 
\newcommand{\natpos}{\bb{N}_+} 

\DeclareMathOperator{\catCommRings}{CRings}

\newcommand{\catSpanSet}[1]{V(#1)}

\newcommand{\bb}{\mathbb}

\newcommand{\catTop}{Top}

\newcommand{\catFinSet}{\mathcal{I}}

\newcommand{\colim}{\operatornamewithlimits{colim}}

\DeclareMathOperator{\id}{id}
\DeclareMathOperator{\pr}{pr}
\DeclareMathOperator{\tor}{Tor}
\DeclareMathOperator{\point}{\{pt\}}

\DeclareMathOperator{\im}{im}

\newcounter{enumi_saved}

\newtheorem{thm}[equation]{Theorem}
\newtheorem{lemma}[equation]{Lemma}
\newtheorem{prop}[equation]{Proposition}
\newtheorem{corr}[equation]{Corollary}
\newtheorem{defn}[equation]{Definition}
\theoremstyle{definition} 
\theoremstyle{definition} \newtheorem{ex}[equation]{Example}
\theoremstyle{definition} \newtheorem{remark}[equation]{Remark}

\begin{document}

\title{Detecting Periodic Elements in Higher Topological Hochschild Homology}
\author{Torleif Veen, \url{torleif.veen@gmail.com}}

\maketitle

\begin{abstract}
Given a commutative ring spectrum $R$ let $\Lambda_XR$ be the Loday functor
constructed by Brun, Carlson and Dundas. Given a prime $p\geq 5$ 
we calculate
$\pi_*(\Lambda_{S^n}H\mathbb{F}_p)$ and $\pi_*(\Lambda_{T^n}H\mathbb{F}_p)$ for $n\leq p$, and use these results to deduce 
that $v_{n-1}$
in the $n-1$-th connective Morava $K$-theory of
$(\Lambda_{T^{n}}H\mathbb{F}_p)^{hT^{n}}$ is non-zero and detected in the homotopy fixed point spectral sequence by an explicit element, which class we name the Rognes class.

To facilitate these calculations we introduce Multifold Hopf algebras. 
Each axis circle in $T^n$ gives rise to a Hopf algebra structure on $\pi_*(\Lambda_{T^n}H\mathbb{F}_p)$, 
and the way these Hopf Algebra structures interact is encoded with a Multifold Hopf algebra structure. This structure puts several restrictions on 
the possible algrebra structures on $\pi_*(\Lambda_{T^n}H\mathbb{F}_p)$ and is a vital tool in the calculations above. 

\end{abstract}

\section{Introduction}

Topological Hochschild
homology of an orthogonal commutative ring spectrum $R$ can be defined as the tensor
$S^1\otimes R$ in the category of orthogonal commutative ring spectra, see \cite{McClureetAl97} and \cite{MandellMay02}. 
Several people have put a lot of effort into computing the homotopy groups of
topological Hochschild homology of various ring spectra. In
\cite{Bokstedt86} he calculates the homotopy groups of $THH$ of the Eilenberg Mac Lane spectra  $\HF$ and $H\bb{Z}$, in  
\cite{McClureStaffeldt93} they calculate the mod $p$ homotopy groups of $THH$ of the Adams
summand $\ell$, in \cite{Ausoni05} he calculates the mod $v_1$ homotopy groups of 
$THH$ of connective complex $K$-theory  and in \cite{AngeltveitHillTyler} the integral
homotopy
groups of $THH(\ell)$ and the $2$-local homotopy groups of $THH(ko)$ are caulcated. 

Let $X$ be a space and write $\loday_{X}R$ for the Loday functor, defined in  Definition \ref{def:LodayFunctor}, first defined for $\Gamma$-spaces in \cite{BrunCarlssonDundas10} and then 
defined for Orthogonal spectra in \cite{Stolz11} and \cite{BrunDundasStolz16}. 
If $G$ is a compact group, and $X$ is a $G$-space, then $\loday_{X}R$ is a $G$-spectrum which is $G$-equivariant equivalent to $R\otimes X$ the categorical tensor, 
when using the $\bb{S}$-model structure from \cite{BrunDundasStolz16}. We will be interested in the case when both $X$ and $G$ are toruses.
 
We write $\Loday(X)$ for the graded ring $\pi_*(\loday_{X}\HF)$.  
Iterated topological Hochschild homology of $\HF$ is then isomorphic
 to $\Loday(T^n)$, where $T^n$ is the $n$-fold pointed torus.

We calculate $\Loday(T^n)$ using the Bar spectral sequences associated with the cofibration given by attaching the top cell in $T^n$.
A first step is thus to calculate $\Loday(S^n)$. 

Let $p$ be an odd prime, $B_0=\bb{F}_p$, $B_1 = P(\mu)$ the polynomial algebra 
over $\bb{F}_p$ on a generator of degree $2$, and when $n\geq 2$ we recursively define 
$B_n = \tor^{B_{n-1}}(\bb{F}_p,\bb{F}_p)$, where the grading in $B_n$ is given by the total grading in $\tor$. 

In Theorem \ref{thm:smash_over_circles} we prove:
\begin{thm}
 When $n\leq 2p$, there is an $\bb{F}_p$-Hopf algebra isomorphism 
  $$\Loday(S^n)\cong B_n.$$
\end{thm}

Let $\uord{n}$ denote the set $\{1,\ldots ,n\}$ of natural numbers. 

In Theorem \ref{thm:smash_over_torus} we prove:
\begin{thm}
 Given $1\leq n\leq p$ when $p\geq 5$ and $1\leq n\leq 2$ when $p=3$,
 there is a graded $\bb{F}_p$-algebra isomorphism
 $$\Loday(T^n)\cong \bigotimes_{U\subseteq \uord{n}}B_{\vert U \vert}.$$
\end{thm}

The fold and pinch maps on each circle factor in $T^n$ produces $n$ different $\Loday(T^{n-1})$-Hopf algebra structures on $\Loday(T^{n})$.

Let $V(\uord{n})$ be the category with objects subsets of  $2^\uord{n}$ and morphisms
from $U$ to $V$ given by subsets of $U\cap V$, where composition is intersection. 
Define the functor $\SLoday: V(\uord{n}) \rightarrow \catCommRings$, where $\catCommRings$ is the category of commutative rings,
by mapping $U\subset \uord{n}$ to $\SLoday(U) = \Loday(T^{U})$, and sending a 
morphism $W\subset U\cap V$ to $\Loday(T^U\rightarrow T^W\rightarrow T^V)$, where the first map is projection and the second is inclusion. 

The functor $\SLoday$ has the structure of a multifold Hopf algebra, as introduced in section \ref{sec:multFoldHopf}. 
This follows from the fact that the all maps from $T^n$ to $(S^1\vee S^1)^{\times n}$ given by pinching each circle in $T^n$ once, are homotopic.
Simultaneously primitive elements is a generalization of primitive elements to multifold Hopf algebras and are elements which are primitive in all the Hopf algebras structures. 

The calculation of $\Loday(T^n)$ is a double induction proof on the dimension of $T^n$ and the degree of $\Loday(T^n)$. 
Similarly to Hopf algebras, simultaneously primitive elements limits the possible non-zero differentials in the Bar spectral
sequence calculating $\Loday(T^n)$, and using this it is shown to collapse at the $E^2$-page, giving the $\bb{F}_p$-module structure. 
Furthermore, the possible algebra structures on $\Loday(T^n)$ are limited by the simultaneously primitive elements helping us identifying the $\bb{F}_p$-algebra structure. 

The redshift conjecture predicts that under favorable circumstances algebraic K-theory 
increases ``telescopic complexity''. See the introduction to \cite{AusoniRognes02}, \cite{AusoniRognes08} and \cite{BaasDundasRognes04}, 
and also \cite{Rognes14} for a wider perspective.  The simplest sequence of examples 
that should display this phenomenon for all complexities is the following: 
for a given prime $p$ the iterated algebraic K-theory $K^{(n)}(\bb{F}_p)$ should 
have telescopic complexity $n-1$.   In particular, if $k(n-1)$ is the $(n-1)$-st 
connective Morava K-theory with coefficient ring $k(n-1)_*=\bb{F}_p[v_{n-1}]$, the 
redshift conjecture predicts that the element $v_{n-1}\in k(n-1)_*K^{(n)}(\bb{F}_p)$ 
is a nonzero divisor.

Most of the evidence for redshift stems from trace methods: according to 
\cite{BokstedtHsiangMadsen93} the trace $K(A)\to THH(A)=\loday_{S^1}(A)$ factors through the 
inclusion of fixed points under the action of subgroups of $S^1$. Since these 
fixed points provide a very close approximation to $K(A)$, \cite{DundasGoodwillieMcCarthy13}
it is reasonable to hope that chromatic behavior for $K^	{(n)}(\bb{F}_p)$ is 
reflected in similar behavior for fixed points of the iterates 
$THH^{(n)}(\bb{F}_p)=\loday_{T^n}H\bb{F}_p$. See the introduction to \cite{CarlssonDouglasDundas11} for more details.

Theorem~\ref{thm:periodic_elm} below is an indication that this is indeed so. 
While we are not able to establish that $v_{n-1}$ is a nonzero divisor, 
we are able to show it is nonzero in an unprecedented range.

\begin{thm} \label{thm:intro_2}
Let $p\geq 5$ and $1\leq n \leq  p$ or $p=3$ and $1\leq n \leq 2$. 
The unit 
 $$\xymatrix{ k(n-1)_* 
 \ar[r]^-{v_{n-1}} &
 k(n-1)_*(\loday_{T^\uord{n}}\HF)^{hT^n}
 }$$
 maps $v_{n-1}$ to a non-trivial class.
\end{thm}
 We call this class \emph{the Rognes class}.
 
 The proof shows a specific element in 
 the homotopy fixed point spectral sequence, denoted the Rognes element, is not hit by any differential, is a cycle and is the image of $v_{n-1}$.
 The only differential that might hit it is the $d^2$-differential, which is induced by the various circle actions on the factors in $T^n$. 
 That this is not possible is a consequence of there being $n$ circle factors, but only $n-1$ odd dimensional generators $\otau_{i}$ in $k(n-1)_*$ of degree less than the degree of $v_{n-1}$. 

 The calculation of $\Loday(T^n)$ should be possible to generalize
to a calculation of the mod $p$ homotopy groups $V(0)_*(\loday_{T^{n}}\HZ)$
and possibly to the mod $v_1$ homotopy groups $V(1)_*(\loday_{T^n}\ell)$ 
in some range which depends on $p$.

\section*{Acknowledgments}
This article is based on parts of my PhD thesis and I would like to thank my 
supervisor Prof. Bj\o{}rn Ian Dundas for all his helpful discussions, and insight, and 
Prof. John Rognes who did the calculations for low $n$ and shared the results with me. 
I would also like to thank the referee and Prof. Haynes Miller for all the helpful suggestions and corrections, and 
in particular for outlining a restructuring of the sections and outlining a much better proof of Theorem \ref{thm:smash_over_torus}.

\section*{Organization}
Section \ref{sec:Loday} recalls some results about spectra and the Loday functor and in section \ref{sec:calcSphereHF} we explicitly calculate
$\Loday(S^n)$
for $n\leq 2p$. Section \ref{sec:multFoldHopf} introduces
multifold Hopf algebras, and in section \ref{sec:possCoalg} we prove that the
structure of a multifold Hopf algebra puts restriction on the
possible coalgebra structures that can appear in $\Loday(T^n)$. In section
\ref{sec:calcTorus}, we calculate $\Loday(T^n)$ for
$n\leq p$
when $p\geq 5$ and $n\leq 2$ when $p=3$.

In section \ref{sec:perElm} we show that there is an element in the second column of
the homotopy fixed points spectral sequence calculating $k(n-1)_*((\loday_{T^{n}}\HF)^{hT^n})$ that is a cycle and not a boundary,
and represents $v_{n-1}$.

The rest of sections several technical results which have been moved out of the main sections to improve the flow of the arguments.

\section{The Loday Functor} \label{sec:Loday}
We will work in the category of orthogonal spectra. See \cite{HillHopkinsRavenel2016} for 
details. In \cite{MandellMaySchwedeShipley01} they prove that 
the category of orthogonal commutative ring spectra is enriched over
topological spaces, and is tensored and cotensored.

We recall the definition of the Loday functor given in Definition 4.3.9 in \cite{BrunDundasStolz16}. The Loday functor was originally stated for
$\Gamma$-spaces in \cite{BrunCarlssonDundas10}, but restated in 
simpler terms for Orthogonal Ring spectra in Martin Stoltz thesis \cite{Stolz11}, which \cite{BrunDundasStolz16} is heavily based on. 
We will use the $\bb{S}$-model structure on Orthogonal Spectra from \cite{BrunDundasStolz16} as this gives us nice equivariant properties as explained below Definition \ref{def:LodayFunctor} 

\begin{defn} \label{def:LodayFunctor}
 Given a space $X$ and a commutative ring spectrum
$R$ we define 
$$\loday_{X}R = R\otimes X.$$ 
where $R\otimes X$ is the categorical tensor in Orthogonal spectra. 
\end{defn}
 If $G$ is a compact Lie group and $X$ is a $G$-space, then $G$ acts on $R\otimes X$ through the action on $X$. 
 When using the $\bb{S}$-model structure this is the action we are interested in, the one used to calculate among other things TC, topological cyclic homology in the case of $G=S^1$.

 When $X$ is a simplicial space, Lemma 4.3.10 in \cite{BrunDundasStolz16} gives us a natural isomorphism 
$\loday_{|X|}R \cong |\loday_XR|,$
where the realization on the right is in orthogonal spectra. 

\begin{prop} \label{prop:lodayProperties}
The Loday functor has the following properties:
\begin{enumerate}
\item  A weak equivalence $X\rightarrow  Y$ of simplicial sets 
induces a weak equivalence $\loday_XR\rightarrow \loday_YR$.
\item  Given a simplicial set $Y$, there is a natural equivalence
$\loday_X(\loday_YR)\simeq \loday_{X \times Y}R$.
\item \label{prop_part:loday_equiv_coprod} Given a cofibration $L\rightarrow X$ and
a map $L\rightarrow K$ between
simplicial sets there is an equivalence
  $\loday_{X\coprod_LK}R\simeq \loday_XR\wedge_{ \loday_LR}\loday_KR.$
\end{enumerate}
\end{prop}
\begin{proof}
Statement 1 and 2 follows from the definition of the Loday functor, and statement 3 follows from  
the fact that tensor commutes with colimits.
\end{proof}

The following two Hopf algebra structures will be essential for our calculations. 

\begin{prop} \label{prop:lodayS^nIsHopf}
 Let $n\geq 1$, $R$ be a commutative ring spectrum, and assume that
$\pi_*(\loday_{S^n}R)$
is flat as a $\pi_*(R)$-module. 
Then $\pi_*(\loday_{S^n}R)$ is  an $\pi_*(R)$-Hopf algebra with unit and counit
induced by choosing a base point in $S^n$ and collapsing $S^n$ to a point,
respectively. The multiplication and coproduct are induced by the fold map
$\fold: S^n\vee S^n\rightarrow S^n$ and the pinch map $\pinch: S^n\rightarrow
S^n\vee S^n$ given by collapsing a chosen equator through the basepoint, respectively, and the conjugation map is induced by the
reflection map $-\id:S^n\rightarrow S^n$. 
\end{prop}
\begin{proof}
By part \ref{prop_part:loday_equiv_coprod} in Proposition \ref{prop:lodayProperties}, $\loday_{S^n\vee S^n}R\simeq
\loday_{S^n}R\wedge_R \loday_{S^n}R$ and since $\pi_*(\loday_{S^n}R)$
is flat as a $\pi_*(R)$-module, $\pi_*(\loday_{S^n}R\wedge_R \loday_{S^n}R)\cong
\pi_*(\loday_{S^n}R)\wedge_{\pi_*(R)}\pi_*( \loday_{S^n}R)$ by Proposition \ref{prop:barSS}. 
That the various diagrams in the definition of a $\pi_*(R)$-Hopf algebra commutes now
follows from commutativity of the corresponding diagrams on the level of
simplicial sets. 
\end{proof}

Given a map of spaces $f:X\rightarrow Y$ we will, when there is no room for
confusion, write  $f$ for both the induced maps $\loday_f\HF\colon
\loday_X\HF\rightarrow \loday_Y\HF$ and $\Loday(f):\Loday(X)\rightarrow
\Loday(Y)$.

\begin{prop} \label{prop:LodayIsHopfAlg}
Let $U$ be a finite set and let $T^U$ be the $U$-fold torus.  For each $u\in U$, if $\Loday(T^U)$ is flat as an
$\Loday(T^{U\setminus u})$-module, then $\left(\Loday(T^U),\Loday(T^{U\setminus u})\right)$ is a commutative Hopf Algebra
where:
\begin{enumerate}
\item Multiplication is induced by the fold map $T^U\amalg_{T^{U\setminus
u}} T^U\cong T^{U\setminus u}\times (S^1\vee S^1)\rightarrow T^U$.
 \item Coproduct is induced by the pinch map $S^1\rightarrow S^1\vee
S^1$ on the $u$-th circle 
in $T^U$.
\item The unit map is induced by the inclusion $U\setminus u\rightarrow U$.
\item The counit map is induced by collapsing the $u$-th circle in $T^U$ to its basepoint.
\end{enumerate}
\end{prop}
\begin{proof}
Since $\loday_{T^U}\HF\cong \loday_{S^1}\loday_{T^{U\setminus u}}\HF$ this
follows from Proposition~\ref{prop:lodayS^nIsHopf}.
\end{proof}

We now set the notation for two maps which are used throughout the article.
\begin{defn} \label{def:circleInc}
Let $X$ be a simplicial set, $R$ be a commutative ring spectrum and let $\bigvee_{x\in X}R$ be the $|X|$-fold wedge sum of $R$, indexed by the elements in $X$.
 Each element $x\in X$ induces a 
 map $\loday_{x}R\rightarrow \loday_{X}R$, and these maps assemble to a natural map
 $$\sCircleInc_X:X_+\wedge R \cong \bigvee_{x\in X}R\cong \bigvee_{x\in X}\loday_{x}R \rightarrow \loday_{X}R.$$
 
Let $Y$ be a simplicial set. Composing $\sCircleInc_X:X_+\wedge \loday_{Y}R\rightarrow \loday_{X\times Y}R $ with the map 
induced by the quotient map $X\times Y\rightarrow X\times Y/(X\vee Y)\cong X\wedge Y$,
 yields a natural map 
$$\sSmashInc_{X}:X_+\wedge \loday_{Y}R\rightarrow \loday_{X\wedge Y}R.$$
\end{defn}

 The map $\sCircleInc_X$ was first constructed in Section 5 of \cite{McClureetAl97}.

\begin{defn} \label{def:circleInLodayMap}
 Composing the maps $\sCircleInc_{S^1}$ and $\sSmashInc_{S^1}$ with a chosen
stable splitting $S^1_+\simeq S^1\vee S^0$, 
induce maps in homotopy
\begin{align*}\pi_*(S^1\wedge R)
\cong
H_*(S^1)\otimes \pi_*(R)
&\rightarrow \pi_*(\loday_{S^1}R) \\
\pi_*(S^1 \wedge \loday_{Y}R)
\cong
H_*(S^1)\otimes \pi_*(\loday_YR) 
&\rightarrow
\pi_*(\loday_{S^1\wedge
Y}R).\end{align*}
We define maps 
\begin{align*}\sOper:\pi_*(R)
&\rightarrow \pi_*(\loday_{S^1}R) \\
\sOper: \pi_*(\loday_YR) 
&\rightarrow
\pi_*(\loday_{S^1\wedge
Y}R).\end{align*} by mapping $z\in \pi_*(R)$ and $y\in \pi_*(\loday_{Y}R)$ to the image of
$[S^1]\otimes z$ and $[S^1]\otimes y$ under $\sCircleInc_{S^1}$ and $\sSmashInc_{S^1}$, 
where $[S^1]$ is a chosen generator of $\widetilde{H}_1(S^1)$.
\end{defn}

The following statement is well known, and is proven in Proposition 5.10 in
\cite{AngeltveitRognes05} for homology, but the same proof works for homotopy. 

\begin{prop} \label{prop:sOperDeriv}
 Let $R$ be a commutative ring spectrum. Then
$\sOper:\pi_*(R)\rightarrow \pi_*(\loday_{S^1}R)$ is a graded
derivation of degree $1$, i.e., 
 $$\sOper(xy) = \sOper(x)y + (-1)^{|x|}x\sOper(y)$$
 for $x,y\in \pi_*R$.
 
Similarly, the composite
 $\sOper:\pi_*(\loday_{S^1}R)\rightarrow \pi_*(\loday_{S^1\times
S^1}R)\rightarrow \pi_*(\loday_{S^1}R)$ where the last map is induced by the
multiplication in $S^1$ is also a derivation. 
\end{prop}

\begin{prop} \label{prop:sOperPrim}
Let $R$ be a commutative ring spectrum, and assume that $\pi_*(\loday_{S^1}R)$
is flat as a $\pi_*(R)$-module. Given $z$ in $\pi_*(R)$,
then 
$\sigma(z)$ is primitive in the $\pi_*(R)$-Hopf algebra
$\pi_*(\loday_{S^1}R)$.
\end{prop}
\begin{proof}
The diagram 
\begin{equation} \label{eq:s_derivation}
\xymatrix{ 
S^1_+\wedge R \ar[r]^-{\sCircleInc} \ar[d]^{\pinch_+\wedge \id} &
\loday_{S^1}R \ar[d]^{\loday_{\pinch}R} \\
(S^1\vee S^1)_+\wedge R \ar[r]^-{\sCircleInc} &
\loday_{S^1\vee S^1}R}\end{equation}
commutes. Hence, $\comult(\sOper(z))=\sOper(z)\otimes 1 + 1\otimes \sOper(z)$.
\end{proof}

\section{Calculating the Homotopy Groups of $\loday_{S^n}\HF$}
\label{sec:calcSphereHF}
In this section we will calculate $\Loday(S^n)$, when $n\leq
2p$ and $p$ is odd. 
First we describe an $\bb{F}_p$-Hopf algebra $\B_n$, and then we show that 
$\Loday(S^n)\cong \B_n$. 

The calculations of $B_n$ were first done in the calculations of the 
Eilenberg MacLane spaces for $\bb{F}$ in \cite{Cartan54}. 

\begin{defn}
 Given the 
letters $\mu$, $\sBar$, $\sBar^k$ and $\varphi^k$ for $k\geq
0$. Define an \emph{admissible word} to be a word such that
\begin{enumerate}
 \item It ends with the letter $\mu$.
 \item If $\mu$ is preceded by a letter, it must be $\sBar$.
  \item If $\sBar$ is preceded by a letter, it must be $\sBar^k$.
  \item If $\sBar^k$ or $\varphi^k$ is preceded by a letter, it must be
  $\sBar$ or $\varphi^l$ for some $l\geq 0$.
\end{enumerate}

We define  a \emph{monic word} to be an admissible word that begins with 
one of the letters $\sBar,
\sBar^0,
\varphi^0$ or $\mu$. 

We define the degree of $\mu$ to be  $2$, and recursively define the
degree of
an admissible word by the rules 
\begin{align*}
 |\sBar x|&= 1+|x| \\
 |\sBar^kx| &= p^k(1+|x|)\\
 |\varphi^kx| &=p^k(2+p|x|).
\end{align*}
\end{defn}

An example of an
admissible word of length $6$ is $\sBar\varphi^m\varphi^l\sBar^k\sBar\mu$. 

\begin{lemma} \label{lemma:AdmissWord}
The following statements hold:
\begin{enumerate}
\item \label{lemma_part:endAdmWord} An admissible word of length at least $3$
always ends with the letter combination $\sBar^k\sBar\mu$
\item \label{lemma_part:sBarAdmWord} There are at most $\frac{n-1}2$ occurrences
of the letter $\sBar$ in an admissible word of even degree of length $n$.  
\item \label{lemma_part:degAdmWord} Every admissible word of length $n$ has
degree at least $n+1$
\item \label{lemma_part:oddAdmWord} All admissible words of odd degree begin
with the letter $\sBar$.
\item \label{lemma_part:modpAdmWord}
Given $0\leq k < p$.
A monic word of degree $2k$ modulo $2p$ is either 
equal to $(\sBar^0\sBar)^{k-1}\mu$, or starts with
the letter combination
$(\sBar^0\sBar)^{k-1}\varphi^0$ or $(\sBar^0\sBar)^{k}$. 
A monic word of degree $2k+1$ modulo $2p$ is either 
equal to $\sBar(\sBar^0\sBar)^{k-1}\mu$, or starts with
the letter combination
$\sBar(\sBar^0\sBar)^{k-1}\varphi^0$ or $\sBar(\sBar^0\sBar)^{k}$. 
\end{enumerate}
\end{lemma}
\begin{proof}
All but the last statement is obvious. 
 The last statement follows from the observation that the degree of a word
starting with $\varphi^l$ or $\sBar^l\sBar $
is $0$
modulo $2p$, when $l\geq1$, and the degree of a  word starting with $\varphi^0$
is $2$ modulo $2p$.
\end{proof}

Let $R$ be a commutative ring, let $x$ and $y$ be of even and odd degree,
respectively. We let $P_R(x)$ be the polynomial ring over $R$ and let $E_R(y)$ be the
exterior algebra over $R$. When $R$ is clear from the setup we often leave it
out of the notation and write $P_p(x)=P(x)/(x^p)$ for the truncated polynomial
ring. Furthermore, we let  $\Gamma(x)$ be the  divided power algebra over
$R$, which as an $R$-module is generated by the elements $\gamma_i(x)$ in degrees
$i|x|$ for $i\geq 0$, with $R$-algebra structure given by
$\gamma_i(x)\gamma_j(x) = \binom{i+j}{j}\gamma_{i+j}(x)$, and $R$-coalgebra
structure given by $\comult(\gamma_{k}(x)) = \sum_{i+j=k}\gamma_i(x)\otimes
\gamma_j(x)$. 

\begin{defn} \label{def:B_n}
We define 
 $\B_{1}$ to be the polynomial $\bb{F}_p$-Hopf algebra $P(\mu)$, with
$|\mu|=2$. 
 Given $n\geq 2$, we define the $\bb{F}_p$-Hopf algebra $\B_{n}$ to be equal to
the tensor product of exterior algebras on all monic words of length
$n$ of odd degrees, and divided power algebras on all
 monic words of length $n$ of even degrees. 
For the divided power algebra structure on $\B_{n}$ we will write
$ \sBar^kx= \gamma_{p^k}(\sBar^0x)$ and $\varphi^kx=\gamma_{p^k}(\varphi^0x)$ 
where $x$ is an admissible word of length $n-1$ and of odd and even degree, respectively.
\end{defn}

For example, the monic words of length $4$ are $\sBar\sBar^k\sBar\mu$ and
$\varphi^0\sBar^k\sBar\mu$ for $k\geq 0$. Hence,  $B_4 = \bigotimes_{k\geq
0}(E(\sBar\sBar^k\sBar\mu)\otimes \Gamma(\varphi\sBar^k\sBar\mu))$.

\begin{prop} \label{prop:B_n}
When $n\geq 2$ there is an isomorphism of $\bb{F}_p$-Hopf algebras
$$B_n\cong \tor^{B_{n-1}}(\bb{F}_p,\bb{F}_p).$$
The map 
$\sigma:\B_{n-1}\rightarrow \B_n$ is determined by $\sigma(x)= \sBar x$ and $\sigma(x) = \sBar^0 x$, when $x$ is an admissible word of even and odd degree, respectively.
\end{prop}
\begin{proof}
This is classical. For a proof see Proposition 7.24 in \cite{McCleary01}.
\end{proof}

Before we calculate $L(S^n)$, we state a technical lemma which is needed in the
proof. Given a graded module $A$, we will write $A_i$ for the part in degree
$i$. 

\begin{lemma} \label{lemma:B_n_primitive}
Let $P(\B_n)$ be the submodule of primitive elements in $\B_n$.
If $2\leq n \leq 2p$ , then $P(B_n)_{2pi-1}=
P(B_n)_{2pi}=0$ for all $i\geq 2$.
\end{lemma}
\begin{proof}
 In a divided power algebra $\Gamma(x)$, the only primitive elements are non-zero scalar multiples of
 $\gamma_1(x)$, so by the graded version of Proposition~3.12 in \cite{MilnorMoore65} and
\ref{prop:B_n}, the
primitive elements in $B_n$ are linear combinations of monic words
 of length $n$.

We will show that the shortest monic word in degree $0$ modulo $2p$ and of
degree greater than $2p$, has length $2p+2$. 

By part~\ref{lemma_part:modpAdmWord} of Lemma~\ref{lemma:AdmissWord},  a monic
word of degree $0$ modulo $2p$ must either be equal to
$(\sBar^0\sBar)^{p-1}\mu$, or start with the letter combination
$(\sBar^0\sBar)^{p-1}\varphi^0$ or $(\sBar^0\sBar)^{p}$. 

The word $(\sBar^0\sBar)^{p-1}\mu$ has degree $2p$, so the shortest 
 monic word in degree $0$ modulo $2p$ of degree greater than $2p$, is thus
$(\sBar^0\sBar)^{p-1}\varphi^0\sBar^k\sBar\mu_0$, for $k\geq 1$, and
it has length $2p+2$.

By a similar argument, we get that the shortest monic word in degree $-1$ modulo
$2p$ of degree greater than $2p$, is 
$\sBar(\sBar^0\sBar)^{p-2}\varphi^0\sBar^k\sBar\mu_0$, for $k\geq 1$, and it
has length $2p+1$. 
\end{proof}

Applying the Loday functor $\Loday(-)$ to the cofiber sequence  
$$\xymatrix{	
S^{n-1} \ar[r] &
D^n \ar[r] &
S^n} $$
gives rise to a bar spectral sequence
$$E^2(S^n)=\tor^{\Loday(S^{n-1})}(\bb{F}_p,\bb{F}_p)\Rightarrow
\Loday(S^n),$$
by Proposition~\ref{prop:lodayProperties} and Proposition~\ref{prop:barSS}.
The spectral sequence is indexed such that the differentials are of the form $d^r:E^r_{s,t}\rightarrow
E^r_{s-r,t+r-1}$. The differentials are only given up to multiplication with a unit. 

For an abelian group $G$ and $m>1$, the bar construction of the Eilenberg-Maclane space $K(G,m)$ is equivalent to $K(G,m+1)$, 
and the spectral sequence $E^2(S^n)$ is analogous to the Eilenberg Moore spectral sequence calculating $K(G,m+1)$ from  $K(G,m)$

The pinch map $\pinch$ induces vertical maps of cofiber sequences
$$\xymatrix{
S^{n-1} \ar[r] \ar[d]^{\pinch}&
D^n \ar[r] \ar[d] &
S^n \ar[d]^{\pinch} \\ 
S^{n-1}\vee S^{n-1} \ar[r] &
D^n\vee D^n \ar[r] &
S^n\vee S^n,}
$$
and this in combination with the reflection map on $S^n$, gives a map of simplicial spectra
$$B(\HF,\loday_{S^{n-1}}\HF,\HF)\rightarrow B(\HF,\loday_{S^{n-1}}\HF,\HF)\wedge_{\HF}
B(\HF,\loday_{S^{n-1}}\HF,\HF)$$ that endows this spectral sequence with a 
$\bb{F}_p$-Hopf algebra structure as explained in Proposition \ref{prop:barSSisCoalg}.
Flatness is no problem, since $\bb{F}_p$ is a field. 

\begin{thm} \label{thm:smash_over_circles}
 When $n\leq 2p$, there are no differentials in the spectral sequence
$E^*(S^n)$, and there is an $\bb{F}_p$-Hopf algebra isomorphism 
  $$\Loday(S^n)\cong B_n.$$
\end{thm}

\begin{proof}
The proof is by induction on $n$. B{\"{o}}kstedt 
calculated in \cite{Bokstedt86} that 
$\Loday(S^1)\cong P(\mu)=B_1$. 

Assume we have proved the theorem for $n-1$.
The bar spectral sequence then becomes
$$E^2(S^n)=\tor^{B_{n-1}}(\bb{F}_p,\bb{F}_p)\cong B_n\Rightarrow
\Loday(S^n).$$

By Proposition~\ref{prop:shortestDiff}, the shortest differential in lowest
total degree goes from an indecomposable element to a primitive element.
We have $E^2(S^n)_{0,*}\cong \bb{F}_p$, so the indecomposable elements in
$B_{n}$
that can support differentials, are generated
by $\sBar^kw$ and
$\varphi^kw$, with $k\geq1$, where $w$ is some admissible word of length $n-1$.
By part~\ref{lemma_part:degAdmWord} in Lemma~\ref{lemma:AdmissWord} these
elements are all in degrees greater than or equal to $4p$, and equal to $0$
modulo $2p$ since $k\geq 1$. Thus if $z$ is an 
indecomposable element, $d^r(z)$ is in
degree $-1$ modulo $2p$, greater than or equal to $4p-1$. 
 By~\ref{lemma:B_n_primitive} there are no
primitive elements in these
degrees when $n\leq 2p$, so there are no differentials in the spectral sequence.
Hence, $E^2(S^n)= E^{\infty}(S^n)$.

To solve the multiplicative extensions we must determine $(\sBar^kw)^p$ and
$(\varphi^kw)^p$ for all $k\geq 0$, and $w$ an admissible word of length $n-1$. 

Assume $z$ is
one of the generators $\sBar^kw$ or $\varphi^kw$ of lowest total
degree
with $z^p\not=0$. 
Then, since the $p$-th powers of primitives are primitive, 
\begin{align*}
 \comult(z^p) &=\comult(z)^p =( 1\otimes z+ z\otimes 1 + \sum z'\otimes z'')^p
  = 1\otimes z^p + z^p\otimes 1+ \sum (z')^p\otimes (z'')^p\\ 
  &= 1\otimes z^p+z^p\otimes 1,
\end{align*}
so $z^p$ must be a  primitive element in degree $0$ modulo $2p$. 
By Proposition~\ref{lemma:B_n_primitive}, this is impossible when $n\leq
2p$ and $|z^p|\geq 4p$, so there are no multiplicative extensions.

When $n\geq 2$ the pinch map $ \pinch:S^n\rightarrow S^n\vee S^n$ is homotopy
cocommutative, i.e. the following
diagram commutes
$$\xymatrix@R=0.5pc{
 & S^n\vee S^n \ar[dd]^-{\twist} \\
S^n\ar[ru]^-{\pinch } \ar[rd]_-{\pinch} \\
& \;S^n\vee S^n,}$$
where $\twist$ interchanges the two factors. 
Cocommutativity is shown by suspending a homotopy between the identity and
antipodal map on $S^1$, picking one of the endpoints of the suspension as the
basepoint in $S^n$, and identifying the suspension of two antipodal points on
$S^1$ to a point,
to define $\pinch$. 

From this is it follows that
$\Loday(S^n)$ is cocommutative as an $\bb{F}_p$-coalgebra when $n\geq2$. Since
$E^2 (S^n)$ is a tensor product of exterior algebras and divided power
algebras, 
Proposition
\ref{prop:coextension_SS} says that there are no coproduct
coextensions. 
Thus $\Loday(S^n)\cong E^\infty (S^n)=E^2(S^n)\cong B_n$ as an $\bb{F}_p$-Hopf
algebra. 
\end{proof}

\section{Multifold Hopf Algebras} \label{sec:multFoldHopf}
The homotopy groups $\Loday(T^n)$ will have several Hopf
algebra structures coming from the various circles in $T^n$. These structures will be
interlinked, and in this section we set up an algebraic framework for this
interlinked structure. 
Our main goal is to state Proposition
\ref{prop:PossIteratedHopfStruc} which is a crucial ingredient in the
calculation of the multiplicative structure of $\Loday(T^n)$.
Throughout this section we will prove statements about the multifold Hopf 
algebra structure of $\Loday(T^n)$ as an illustration of the definitions and propositions.

This exposition of Multifold Hopf algebra leaves several unanswered question. In
particular, it would be
interesting to have a good description of the module of elements that are
primitive in all the Hopf
algebra structures simultaneously. In that regard a generalization of the
very special case in Lemma
\ref{lemma:dimOfn-primitive} would be welcomed. 

Let $\catCommRings$ be the category of commutative rings.
In this section we will assume that all our Hopf algebras are connected and
commutative. 

First we construct a category of Hopf algebras, and show that it has all small
colimits. Objects in this category are ordinary Hopf algebras, but we use the
morphisms to define a multifold Hopf algebra. 

\begin{defn}
The category of  Hopf algebras has objects pairs of
commutative rings $(A,R)$ where $A$ is given the structure of a
commutative connected $R$-Hopf algebra. 
A morphism from $(A,R)$ to $(B,S)$ consists of two maps $f\colon A\rightarrow B$
and
$g\colon R\rightarrow S$ of commutative rings, such that
$f$ is a map of
 $R$-algebras, and $f\otimes_g \id:A\otimes_RS\rightarrow B\otimes_SS\cong B$ is a map of $S$-coalgebras. 
\end{defn}

\begin{prop} \label{prop:pushoutHopfAlg}
The category of Hopf algebras has all small colimits, and the colimit
$\colim_J(A_j,R_j)$ is equal to the pair $(\colim_JA_j,\colim_JR_j)$, of 
colimits in the category of commutative rings. 
\end{prop}
\begin{proof}
The proof is left to the reader.

\end{proof}

Our multifold Hopf algebras will be functors from the following categories. 
Let $S$ be a finite set and define $\catSpanSet{S}$ to be the category 
with objects subsets of $S$ and morphisms
from $U$ to $V$ given by subsets of $U\cap V$, where composition is intersection.
Given an element $s\in S$ we will often write $S\setminus
s$ for $S\setminus \{s\}$ to make the formulas more readable.

Equivalently, $\catSpanSet{S}$ is isomorphic to the category of spans in
$\ord{2}^S$, where $\ord{2}^S$ is the
category with objects subsets of $S$ and morphisms inclusions of sets. 
There is an inclusion $\ord{2}^S\rightarrow \catSpanSet{S}$ given by sending a
morphism $U\subseteq V$ to the morphism $U$ from $U$ to $V$.   

\begin{ex}
 Let $S=\{u,v\}$. Then, the non-identity morphisms in the category $\catSpanSet{S}$ are given in the diagram 
$$
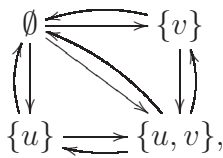
\xymatrix{
\emptyset \ar[r] \ar[d] \ar[rd]&
\{v\} \ar@/_0.5pc/[l] \ar[d] \\
\{u\} \ar@/^0.5pc/[u] \ar[r] &
\{u,v\}, \ar@/_0.5pc/[u] \ar@/^0.5pc/[l], \ar@/_0.5pc/[ul]
}$$
and the image of the non-identity morphisms under the inclusion $\ord{2}^S\rightarrow \catSpanSet{S}$ are
the straight arrows.
\end{ex}

The next definition is only a preliminary step towards the final definition of an
$S$-fold Hopf algebra in \ref{def:S-foldHopfAlgebra}.

\begin{defn} \label{def:preS-foldHopfAlgebra}
 Let $S$ be a finite set viewed as a discrete category, and let $X\subset V(S)\times S$ denote the full subcategory of pairs $(V,v)$ with $v\in V$.
A pre $S$-fold Hopf algebra $A$ is a functor $A:\catSpanSet{S}  \rightarrow 
\catCommRings$, such
that:

For every $v\in V\subseteq S$, the pair $(A(V),A({V\setminus v}))$ is equipped with the
structure of a Hopf algebra with unit and counit induced by the spans
$V\setminus v\leftarrow V\setminus v \rightarrow V$ and $V\leftarrow V\setminus v \rightarrow V\setminus v$, respectively,
such that with this structure the composite
$$\xymatrix{X \ar[r]^-{F} & \catSpanSet{S}\times \catSpanSet{S}\ar[r]^-{A\times A} &
\catCommRings\times \catCommRings,}$$
where $F$ is the functor $F(V,v)=(V,V\setminus v)$, becomes a functor to the category of Hopf algebras.

 We let $\comult_V^v$, $\mult_V^v$, $\unit_V^v$ and
$\counit_V^v$ denote the coproduct, product, unit and counit in the Hopf algebra
$(A(V),A({V\setminus v}))$, respectively.
\end{defn}

For a functor $A:C\rightarrow D$ where the objects of $C$ are finite sets, we will for $V\in C$ write $A_V$ for $A(V)$.

\begin{defn}
 A map from a pre $S$-fold Hopf algebra $A$ to a pre $S$-fold Hopf algebra
$B$ is a natural transformation from $A$ to $B$ such that for every 
$v\in V\subseteq S$ the
induced map from $(A_V,A_{V\setminus v})$ to $(B_V, B_{V\setminus v})$ is a map
of Hopf algebras.  
\end{defn}

Fix a basepoint on the circle $S^1$. 
Let $\catFinSet$ be the category with objects finite sets of natural numbers,
and morphisms
inclusions. 

We define the functor $T:\catFinSet\rightarrow \catTop$ 
by $T(\emptyset) = \point$, and when $U\not= \emptyset$, $T(U)=T^U$, the
$U$-fold torus. On morphisms it takes an inclusion $U\subseteq V$
to the inclusion
$\inc_U^V\colon T^U\rightarrow T^V$, where we use the basepoint in the
factors not in $U$. Furthermore there is the projection map 
$$\pr^V_U\colon T^V\rightarrow T^U.$$

Give the circle $S^1$ the minimal $CW$-structure with one $0$-cell and one $1$-cell, and give the $U$-fold
torus $T^U$ the product $CW$-structure. We write $T_{k}^U$ for the $k$-skeleton of $T^U$.
If $U$ has cardinality $k$, the quotient map 
$$\sphereproj^U:T^U\rightarrow T^U/{T^U_{k-1}}$$ maps to
the $U$-sphere $S^U$.

\begin{defn}
Let $W$ be an object in $\catFinSet$.
 Define the functor $\SLoday:\catSpanSet{W} \rightarrow \catCommRings$
 on an object $V\in \catSpanSet{W}$ by $\SLoday(V) = \Loday(T^{V})$ and on a  map $U:V\rightarrow X$
by $\SLoday(U)=\inc_X^U\circ \pr_U^V$.  
\end{defn}

\begin{prop} \label{prop:TorusPreSFoldHopf}
Let $W$ be an object in $\catFinSet$.
\begin{enumerate}
\item
The functor $\SLoday$ is a pre $W$-fold Hopf algebra, when for all $v\in U\subseteq W$, the pair $\big(\Loday(T^U),\Loday(T^{U\setminus
v})\big)$ are equipped with the
Hopf algebra structure in Proposition~\ref{prop:LodayIsHopfAlg}. 

\item
The map $\sphereproj^W:T^W\rightarrow S^W$ induces a
map of Hopf
algebras 
  $$\big(\Loday(T^W),\Loday(T^{W\setminus j})\big)\rightarrow
\big(\Loday(S^{|W|}),\bb{F}_p\big).$$
\end{enumerate}
\end{prop}
\begin{proof}
 Given $U\subseteq V\subseteq W$ and $v\in U$ we get two homomorphisms of Hopf
algebras $\big(\Loday(T^U),\Loday(T^{U\setminus
v})\big)\rightarrow \big(\Loday(T^V),\Loday(T^{V\setminus v})\big)\rightarrow
\big(\Loday(T^U),\Loday(T^{U\setminus v})\big)$ induced by the inclusion
$U\setminus v
\rightarrow V\setminus v$. 
Hence, $\SLoday$ is a pre $W$-fold Hopf algebra. 

Given subsets $U\subseteq V \subseteq W$, the ring $\SLoday^U_V$, as defined in
Definition~\ref{def:cubeProduct}, is isomorphic to $\Loday(T^{V\setminus
U}\times (S^1\vee S^1)^U)$, since $\SLoday$ commutes with colimits, and the
colimit of the composite 
$$\xymatrix{
T(U)\ar[r]^-{-\cup-} &
\ord{2}^U \ar[rr]^-{-\cup (V\setminus U)} &&
\ord{2}^S \ar[r]^-{T^-} &
\catTop.}$$
is $T^{V\setminus
U}\times (S^1\vee S^1)^U$ where the definitions are as in Definition
\ref{def:cubeProduct}.

\end{proof}

Let $T(S)$ be the full subcategory of $\ord{2}^S\times \ord{2}^S$
with objects pairs $(U,V)$ with $U\cap V=\emptyset$.

\begin{defn} \label{def:cubeProduct}
Let $S$ be a finite set and $A$ a functor $A:\ord{2}^S\rightarrow
\catCommRings$.
 Given finite sets $U\subseteq V \subseteq S$
 we define the functor $F_{A,V}^U$
to be the composite
$$\xymatrix{
T(U)\ar[r]^-{-\cup-} &
\ord{2}^U \ar[rr]^-{-\cup (V\setminus U)} &&
\ord{2}^S \ar[r]^-{A} &
\catCommRings,}$$
and define $A_V^U$ to be the colimit of the functor $F_{A,V}^U$.
\end{defn}

Using the inclusion $\ord{2}^S\rightarrow
\catSpanSet{S}$ this construction applies to any pre $S$-fold Hopf algebra $A$. 

Note that $A^{\emptyset}_{V} = A_V$.

\begin{ex} \label{ex:cubeColimit}
Let $U=\{u,v\}\subseteq V $.
The source category $T(U)$ of $F_{A,V}^U$ is the diagram on the left, and
the image of $F_{A,V}^U$ in commutative rings is the diagram on the right:
$$
\xymatrix{
\{u,v\},\emptyset  &
\{u\},\emptyset \ar[r] \ar[l] &
\{u\},\{v\} \\
\{v\},\emptyset \ar[u]\ar[d]  &
\emptyset,\emptyset \ar[r] \ar[d] \ar[u] \ar[l] &
\emptyset,\{v\} \ar[u] \ar[d] \\
\{v\},\{u\}  &
\emptyset,\{u\} \ar[r] \ar[l] &
\emptyset,\{u,v\}}
\qquad \qquad
\xymatrix{
A_V  &
A_{V\setminus v} \ar[r] \ar[l] &
A_{V} \\
A_{V\setminus u} \ar[u]\ar[d]  &
A_{V\setminus \{u,v\}} \ar[r] \ar[d] \ar[u] \ar[l] &
A_{V\setminus u} \ar[u] \ar[d] \\
A_{V}  &
A_{V\setminus v} \ar[r] \ar[l] &
\;A_{V}.}$$
\end{ex}

\begin{ex}
 For $U\subset V\subset W$ we have 
 $\SLoday^U_V \cong \Loday(T^{V\setminus U}\times (S^1\vee
S^1)^{\times U})$
\end{ex}

\begin{remark}
We will now describe a helpful way to think about the rings $A_V^U$. 
The power set, $P(U)$ of $U$, can be thought of as a discrete category,
with objects the subsets of $U$. 
There is a functor $G$ from $P(U)$ to $T(U)$ given by mapping 
$W\subseteq U$ to the pair $(U\setminus W, W)$. The composite $F_{A,V}^U\circ G$ is the
constant functor $A_V$, so this induces a surjective map on colimits from 
$A_V^{\otimes P( U)}$ to $A_V^U$. 

An element in $A_V^U$ can thus be represented
by an element in $A_V^{\otimes P( U)}$. 
By thinking of the objects in $P(U)$ as the corners of the unit cube in $\bf{R}^U$, 
we can think of the monomials in $A_V^{\otimes P( U)}$ as $U$-indexed cubes with a monomial from $A_V$ in each corner. 

We will write the image of such a
representative in $A_V^{\otimes P( U)}$ as a
sum of $U$-indexed cubes as well.
\end{remark}

\begin{ex}
 Let $U=\{\zu,\zv\}\subseteq V$. 
An element of $A_V^U$ is represented by a sum of cubes 
$$\begin{bmatrix}
   x_{\emptyset} & x_{\{\zv\}} \\ x_{\{\zu\}} & x_{\{\zu,\zv\}}
  \end{bmatrix},$$
where the displayed cube is the notation for $\bigotimes_{I\in P(U)}x_I\in A_V^{\otimes P(U)}$. The four entries in the cube 
correspond to the four corners in
the right diagram in Example \ref{ex:cubeColimit}, and the subscripts are given
by the second set in the four corners in the left diagram. Multiplication is
done
component wise, and we have the following identifications
\begin{align*}
 \begin{bmatrix}
   a_{ {\zu}}x_{\emptyset} & x_{\{\zv\}} \\ x_{\{\zu\}} & x_{\{\zu,\zv\}}
  \end{bmatrix}
  &=
  \begin{bmatrix}
   x_{\emptyset} &  a_{ {\zu}}x_{\{\zv\}} \\ x_{\{\zu\}} & x_{\{\zu,\zv\}}
  \end{bmatrix}
  &
   \begin{bmatrix}
   x_{\emptyset} & x_{\{\zv\}} \\ a_{ {\zu}} x_{\{\zu\}} & x_{\{\zu,\zv\}}
  \end{bmatrix}
  &=
  \begin{bmatrix}
   x_{\emptyset} &  x_{\{\zv\}} \\ x_{\{\zu\}} & a_{ {\zu}}x_{\{\zu,\zv\}}
  \end{bmatrix}
 \\
 \begin{bmatrix}
   a_{ {\zv}}x_{\emptyset} & x_{\{\zv\}} \\ x_{\{\zu\}} & x_{\{\zu,\zv\}}
  \end{bmatrix}
  &=
  \begin{bmatrix}
   x_{\emptyset} &  x_{\{\zv\}} \\a_{ {\zv}} x_{\{\zu\}} & x_{\{\zu,\zv\}}
  \end{bmatrix}
  &
   \begin{bmatrix}
    x_{\emptyset} & a_{{\zv}}x_{\{\zv\}} \\  x_{\{\zu\}} & x_{\{\zu,\zv\}}
  \end{bmatrix}
  &=
  \begin{bmatrix}
   x_{\emptyset} & x_{\{\zv\}} \\ x_{\{\zu\}} &   a_{ {\zv}}x_{\{\zu,\zv\}}
  \end{bmatrix}
\end{align*}
when $a_{\zu}$ is an element in $A_{V\setminus \{\zu\}}\subseteq A_V$, and
$a_{ {\zv}}$ is an element in $A_{V\setminus \{\zv\}}\subseteq A_V$. Observe
that
if $a$ is an element in $A_{V\setminus \{\zu,\zv\}}$ we can move the element
between all four corners of the cube.
\end{ex}

\begin{defn} \label{lemma:cubeMap}
Observe that the colimits of the columns in the right diagram in Example
\ref{ex:cubeColimit} are 
$$\xymatrix{A_V^{\{u\}}  &
A_{V\setminus \{v\}}^{\{u\}} \ar[r] \ar[l] &
A_{V}^{\{u\}}}.$$
Given a map of diagrams 
$$\xymatrix{
A_V \ar[d]^{\comult_V}&
A_{V\setminus v} \ar[r] \ar[l] \ar[d]^{\comult_{V\setminus v}} &
A_{V} \ar[d]^{\comult_V} \\
A_V^{\{u\}}  &
A_{V\setminus v}^{\{u\}} \ar[r] \ar[l] &
\;A_{V}^{\{u\}},}$$
we will write the map on the colimits of the horizontal direction as
$$\begin{bmatrix} \comult_V &\comult_V\end{bmatrix}\colon
A_V^{\{v\}}\rightarrow A_V^{\{u,v\}}.$$
\end{defn}

\begin{lemma} \label{lemma:pushoutOfCubes}
 For $U\subseteq V$ and $v\in V\setminus U$, the universal property of colimits
induces an isomorphism
$$A^U_{V}\otimes_{A^U_{V\setminus v}}A^U_{V}\cong A_{V}^{U\cup v}$$
of commutative rings. 
\end{lemma}
\begin{proof}
 Both sides are the colimit of the functor $F^{U\cup v}_{A,V}$. On the
left hand side the colimit is evaluated in two steps, evaluating the $v$-th
direction in the diagram $T(U\cup v)$ last. 

More explicitly, the middle term
$A^U_{V\setminus v}$ is the colimit of the functor $F^{U\cup v}_{A,V}$
precomposed with the inclusion $T(U)\rightarrow T(U\cup v)$. The two outer terms
$A^U_V$ are the colimit of the functor $F^{U\cup v}_{A,V}$ precomposed with the
two maps $T(U)\rightarrow T(U\cup v)$, given by adding $v$ to the first and
second set in $T(U)$, respectively.
\end{proof}

Given a pre $S$-fold Hopf algebra $A$, there are some related multifold
Hopf algebras. Thinking of a pre $S$-fold Hopf algebra as a sum of $S$-cubes, 
with corners indexed by the subset of $S$, 
the first part of the next proposition says that every face
is a pre multifold Hopf algebra in a natural way. 

\begin{prop} \label{prop:moreFoldHopfAlg}
Let $A$ be a pre $S$-fold Hopf algebra. If $U$ and $W$ are subsets of $S$, the
composite 
$$\xymatrix{\catSpanSet{W} \ar[r]^-{-\cup U} &\catSpanSet{S} \ar[r]^-A
&\catCommRings}$$ 
is a pre $W$-fold Hopf algebra.
If $U$ is a subset of $S$ the functor 
$$A^U:\catSpanSet{S\setminus U}\rightarrow \catCommRings$$
given by $A^U(V) = A^U_{V\cup U}$ is a pre $S\setminus U$-fold Hopf
algebra. 
\end{prop}
\begin{proof}
The first case is clear by definition. 
In the second case, for every $U\subseteq V\subseteq S$ and $v\in
V\setminus U$, we need to give a Hopf algebra structure to the pair  
$\big(A^{U}_{V}, A^{U}_{V\setminus v}\big)$ satisfying the definition of a pre
$S\setminus U$-fold Hopf
algebra. 

We claim 
there is a pushout diagram
$$\xymatrix{
\big(A^{U}_{V\setminus u}, A^{U}_{(V\setminus u)\setminus
v}\big) \ar[r] \ar[d]
&
\big(A^{U}_V, A^{U}_{V\setminus v}\big)  \ar[d] \\
\big(A^{U}_V, A^{U}_{V\setminus v}\big) \ar[r] &
\;\big(A^{U\cup u}_V, A^{U\cup u}_{V\setminus  v}\big)}$$
of Hopf algebras. 

The identification of the pushout  follows
from Lemma
\ref{lemma:pushoutOfCubes}. The case $U=\emptyset$ follows from
the definition of a pre $S$-fold Hopf algebra. The other statements follow by 
induction on the cardinal of $U$ and Proposition \ref{prop:pushoutHopfAlg}. 

The universal property of pushouts guarantees that these Hopf algebras combine
to a functor satisfying 
the definition of a pre $S\setminus U$-fold Hopf algebra.
\end{proof}

Composing the various coproducts in a pre $S$-fold Hopf algebra gives rise to
several homomorphisms that we now introduce. An $S$-fold Hopf algebra is a pre
$S$-fold Hopf algebra where these various homomorphisms agree. 

\begin{defn} \label{def:multifoldComult}
Let $A$ be a pre $S$-fold Hopf algebra. Given a pair of sets 
$U\subseteq V \subseteq S$ with
$v\in V\setminus U$ we define  
$$\comult_V^{U,v}\colon A_V^U \rightarrow
A_V^U\otimes_{A_{V\setminus v}^U}A_V^U\cong A_V^{U\cup
v}$$
to be the composition of the coproduct in the Hopf algebra
$\big(A^{U}_{V}, A^{U}_{V\setminus v}\big)$ with the isomorphism from
Lemma \ref{lemma:pushoutOfCubes}

Given a sequence of distinct elements
$u_1,u_2,\ldots ,u_k\in V\subseteq S$, we
define 
$$\comult_V^{u_1,\ldots,u_k}:A_V\rightarrow A_V^{\{u_1,\ldots,
u_k\}}$$ by the recursive formula 
$$\comult_V^{u_{i},\ldots,u_k} = \comult_V^{\{u_{i+1},\ldots,u_{k}\},u_i}
 \circ \comult_V^{u_{i+1},\ldots,u_k}.$$
Similarly, we define 
$$\redcomult_V^{u_1,\ldots,u_k}\colon A_V\rightarrow A_V^{\{u_1,\ldots,
u_k\}}$$ using the reduced
coproducts $\redcomult_V^{\{u_{i+1},\ldots,u_{k}\},u_i}$ which are defined by the maps $\comult_V^{\{u_{i+1},\ldots,u_{k}\},u_i}-1\otimes \id - \id \otimes 1$.
\end{defn}

\begin{defn}
 Let $S$ be a finite set, and let $A$ a pre $S$-fold Hopf algebra.
 An \emph{$S$-fold primitive element} in $A$, is an element which is 
 primitive in $(A_S, A_{S\setminus s})$ for each $s \in S$. When $S$ is clear from context, it will also be denoted a \emph{simultaneous primitive element}
\end{defn}

\begin{defn} \label{def:S-foldHopfAlgebra}
 Let $S$ be a finite set.
An $S$-fold Hopf algebra $A$ is a pre $S$-fold Hopf algebra $A$ 
with the additional requirement that for every sequence $u_1,u_2,\ldots,u_k$ of
distinct elements in $V\subseteq S$, and all permutations $\alpha$ of $k$, 
$$\comult_V^{u_{\alpha(1)},\ldots , u_{\alpha(k)}}=\comult_V^{u_1,\ldots
,u_k}:A_V\rightarrow A_V^{\{u_1,\ldots ,u_k\}}.$$
We denote this map 
$$\comult_V^U\colon A_V\rightarrow A_V^U,$$
where $U=\{u_1,\ldots,u_k\}$.
\end{defn}

\begin{defn}
 A map from an $S$-fold Hopf algebra $A$ to an $S$-fold Hopf algebra
$B$ is a map of pre $S$-fold Hopf algebras. 
\end{defn}

\begin{prop} \label{prop:TorusSFoldHopf} 
Let $W$ be an object in $\catFinSet$.
The functor $\SLoday:\catSpanSet{W} \rightarrow \catCommRings$ is a $W$-fold Hopf algebra
\end{prop}
\begin{proof}
We proved in Proposition \ref{prop:TorusPreSFoldHopf} that $\SLoday$ is a pre $W$-fold Hopf algebra. 
That $\SLoday$ is also a $W$-fold Hopf algebra, follows from the geometric
origin of the coproducts $\comult_V^i$ for $i\in V\subseteq W$. 
Given a sequence $u_1,u_2,\ldots , u_k$ of distinct elements in $V\subseteq
W$, let $U=\{u_1,u_2,\ldots , u_k\}$. The map $$\comult_V^{u_1,u_2,\ldots,
u_k}:\SLoday(V)\cong \Loday(T^V)\rightarrow \SLoday^U_V\cong \Loday(T^{V\setminus
U}\times (S^1\vee S^1)^U),$$ defined in Definition~\ref{def:multifoldComult}, is
induced by the pinch map on every circle in $T^U\subseteq T^V$. Hence, it is
independent of the order of the elements $u_i$ in $\comult_V^{u_1,u_2,\ldots,
u_k}$.
\end{proof}

It shouldn't come as a surprise that when the composition of the coproducts
agree, the composition of the reduced coproducts agree. More precisely:

\begin{prop}
 Let $A$ be an $S$-fold Hopf algebra,  and 
let $u_1,u_2,\ldots,u_k$ be a sequence of distinct elements in $V\subseteq S$. 
 Then 
$$\redcomult_V^{u_{\alpha(1)},\ldots , u_{\alpha(k)}}=\redcomult_V^{u_1,\ldots
,u_k}.$$
for all permutations $\alpha$ of $k$. 
\end{prop}
\begin{proof}
By Proposition  \ref{prop:moreFoldHopfAlg}, it
 suffices to check the claim for transpositions, since $A^{\{u_i,\ldots,
u_k\}}$ is an $S\setminus \{u_i,\ldots, u_k\}$-fold Hopf algebra for every
$i\leq k$. That it holds for transpositions is easily checked by calculating the two sides, using the cube notation. 
\end{proof}

We end this section by constructing some special $S$-fold Hopf algebras. 

\begin{defn} \label{def:satSubCat}
 Let $S$ be a finite set. We define a subcategory $\Delta$ of $\catSpanSet{S}$
to be \emph{saturated} if it has the property that if $W\in \Delta$ then
$\catSpanSet{W}\subseteq \Delta$.
\end{defn}

\begin{defn} \label{def:partialHopfAlg}
 Let $S$ be a finite set and let $\Delta$ be a saturated subcategory of
$\catSpanSet{S}$. Define a \emph{partial $S$-fold Hopf algebra} $A\colon\Delta
\rightarrow
\catCommRings$, to be a functor $A$ which for every $W\in \Delta$, is a
$W$-fold Hopf algebra when restricted to $W$. 

\end{defn}

\begin{defn} \label{def:extensionHopfAlg2}
 Let $S$ be a finite set, let $\Delta$ be a saturated subcategory of
$\catSpanSet{S}$, let $\overline{\Delta}$ be the subcategory $\bigcup_{W\in \Delta}\ord{2}^W\subseteq
\Delta$, and let $A\colon\Delta
\rightarrow
\catCommRings$ be a partial $S$-fold Hopf algebra.

The functor  
$$\overline{A}:\catSpanSet{S} \rightarrow \catCommRings$$
defined by $\overline{A}(W)=\colim_{U\subseteq W, U\in \overline{\Delta}}A(U)$
has
the structure of an $S$-fold Hopf algebra and we denote it \emph{the extension
of $A$ to $S$}.

\end{defn}

Since the
category of Hopf algebras has all small colimits, and these are given as pair
of colimits in commutative rings, it is clear that the extension of $A$ is an
$S$-fold Hopf algebra. All the properties of an $S$-fold Hopf algebra follows
from functoriality of the colimit.

\begin{defn} \label{def:restrictionHopfAlg}
Let $S$ be a finite set and let $\Delta$ be a saturated subcategory of
$\catSpanSet{S}$
Given an $S$-fold Hopf algebra $A$, we define 
\emph{the restriction of $A$ to $\Delta$} to be the $S$-fold Hopf algebra which
is
the extension of the functor $$A|_{\Delta}:\Delta \rightarrow \catCommRings.$$
\end{defn}

\begin{ex}
 Let $W$ be a finite set, let $k>0$ be an integer, and let $\Delta$ be the saturated subcategory of
$\catSpanSet{W}$ containing all sets of cardinality at most $k$. 
Let $\SLoday_{\Delta}$ be the restriction of $\SLoday$ to $\Delta$. Then for $U\subseteq W$
$$\SLoday_{\Delta}(U) \cong \Loday(T^U_k)$$
where $T^U_k$ is the $k$-skeleton of $T^U$. 
\end{ex}

\begin{defn}
Let $S$ be a finite set, and let $m$ be a positive even integer. Let
$\Delta\subseteq \catSpanSet{S}$ be the full subcategory containing all sets
with at most one element. Let $A:\Delta\rightarrow \catCommRings$ be the
functor given by $A(\emptyset)=R$ and $A(\{s\})=P_R(\mu_s)$, with
$|\mu_s|=m$. 

We define $P_R(\mu_{-})$, \emph{the polynomial $S$-fold Hopf algebra over $R$
in degree
$m$}, to be the extension of the functor $A$ to all of $S$.
\end{defn}

\begin{ex}
For the set $S=\{s_1,\ldots,s_k\}$ there is an $R$-algebra isomorphism $
P_R(\mu_{S})\cong P_R(\mu_{s_1},\ldots , \mu_{s_k})$, and for each $s\in S$ the
element $\mu_{s}$ is primitive in the Hopf algebra $(P_R(\mu_{S}),
P_R(\mu_{S\setminus s}))$. 

For $U\subseteq V \subseteq S$ there is an $R$-algebra isomorphism 
$$(P_R(\mu_{S}))^U_V \cong P_R(\mu_{V})\otimes_{P_R(\mu_{V\setminus U})}P_R(\mu_{V}).$$

Since $P_R(\mu_{S})$ is generated as an $R$ algebra by the set $\{\mu_{s_1},\ldots , \mu_{s_k}\}$, and $\comult_S^{s_j}(\mu_{s_i}) = \mu_{s_i}$ when $i\not=j$, 
the iterated coproduct 
$$\comult_V^U:P_R(\mu_{V})\rightarrow (P_R(\mu_{S}))^U_V \cong P_R(\mu_{V})\otimes_{P_R(\mu_{V\setminus U})}P_R(\mu_{V})$$
is determined by $\comult_V^U(\mu_{u_i}) = \mu_{u_i}\otimes 1 + 1 \otimes \mu_{u_i}$ for $u_i\in U$ and $\comult_V^U(\mu_{u_i}) = \mu_{u_i}\otimes 1 = 1 \otimes \mu_{u_i}$ for $u_i\in V\setminus U$. 
\end{ex}

\begin{ex}
When $m=2$, and $\Delta$ is the subcategory of
$\catSpanSet{W}$ containing all sets with at most one element, the functor 
$\SLoday_{\Delta}$ is naturally isomorphic to  $P_{\bf{F}_p}(\mu_{-})$.
\end{ex}

\section{Coproduct in a Multifold Hopf Algebra}
\label{sec:possCoalg}
In this section we will state a proposition that we need when we 
calculate the multiplicative structure of $\Loday(T^n)$.

Given an integer $n$ divisible by $p$, we write $\frac{n}{p}$ for the image of
$\frac{n}{p}$ under the ring map $\bb{Z}\rightarrow \bb{F}_p$. 
In the polynomial $\bb{F}_p$-Hopf algebra $P_{\bb{F}_p}(\mu)$, we write
$\divByP{\redcomult(\mu^{p^i})}$ for the image of
$\divByP{\redcomult(\mu^{p^i})}$ under the ring map 
$P_{\bb{Z}}(\mu_n)\rightarrow P_{\bb{F}_p}(\mu_n)$ given by mapping $\mu_n$
to $\mu_n$. This is well defined since $\binom{p^i}{k}$ is
divisible by $p$ for
all $i$ and $k$ with $0<k<p^i$.

\begin{lemma} \label{lemma:HopfAlgebraRelations}
 Let $M$ be an $\bb{F}_p$-module, and let $n$ be a natural
number greater than $2$. Let
$\{r_{k,n-k}\}_{0\leq k \leq n}$ be a subset of $M$ which satisfy 
the relations
$\binom{a+b}{b}r_{a+b,c}=\binom{b+c}{b}r_{a,b+c}$ for all $a+b+c=n$ and
$0<a,c<n$.
Then the following relations hold:
\begin{enumerate}
 \item \label{lemma_part:one_prime} If $n=p^{m+1}$ for some $m\geq 0$, then 
$$r_{k,n-k}=\divByP{\binom{n}{k}}r_{p^m,(p-1)p^m}$$
for all $0<k<n$.
\item \label{lemma_part:n_sum_two_primes} If $n=p^{m_1}+p^{m_2}$ with
$m_1<m_2$ and
$k\not=p^{m_1},p^{m_2}$, then 
$r_{k,n-k}=0$.
\item \label{lemma_part:n_sum_many_primes} If $n\not= p^{m+1}, p^{m_1}+p^{m_2}$
with $m_1<m_2$, then   
$$r_{k,n-k} = \binom{n}{k}n_m^{-1}r_{p^m,n-p^m}$$
for all $0<k<n$, where $n=n_0+n_1p^1+\ldots +n_mp^m$ with $0\leq n_i<p$ and
$n_m\not=0$
is the $p$-adic representation of $n$.
\end{enumerate}

\end{lemma}
The only case which is not covered by the lemma is $n=p^{m_1}+p^{m_2}$
with $m_1\not=m_2$, when the relations in the lemma doesn't give any
relation between
$r_{p^{m_1},p^{m_2}}$ and $r_{p^{m_2},p^{m_1}}$.

\begin{proof}
Given a set $\{r_{k,n-k}\}_{0<k<n}$ of elements in an abelian group, let $\sim$
be the equivalence relation generated by
$\binom{a+b}{b}r_{a+b,c}\sim\binom{b+c}{b}r_{a,b+c}$. Let
$\bb{F}_p\{r_{1,n-1},\ldots ,r_{n-1,1}\}$ be the free
$\bb{F}_p$-module on the set $\{r_{1,n-1},\ldots ,r_{n-1,1}\}$.
Since $M$ is an $\bb{F}_p$-module, there is a
homomorphism
$$\bb{F}_p\{r_{1,n-1},\ldots ,r_{n-1,1}\}/{\sim}\rightarrow M$$
defined by mapping $r_{k,n-k}$ to $r_{k,n-k}$.
Hence it suffices to prove the lemma for the module
$M=\bb{F}_p\{r_{1,n-1},\ldots
,r_{n-1,1}\}/{\sim}$.

Let 
$$k=k_0+k_1p^1+\ldots +k_jp^j$$ 
with $0\leq k_i<p$ and $k_j\not=0$, be
the $p$-adic representations of $k$. Similarly, let 
$$n=n_0+n_1p^1+\ldots +n_mp^m$$
with $0\leq n_i<p$ and
$n_m\not=0$ be the $p$-adic representation of $n$, except that when $n$ is a
power of $p$ we express it as $n=p^{m+1}$. 

The proof consists of two parts. First we prove that unless both $k$ and $n-k$
are powers of $p$, there is a sequence of equations expressing $r_{k,n-k}$
as a multiple of $r_{p^m,n-p^m}$.
The second part is to identify the factor in this equation in terms of
$\binom{n}{k}$. 

We will now use the relations $\binom{a+b}{b}r_{a+b,c}=\binom{b+c}{b}r_{a,b+c}$
to express $r_{k,n-k}$ as a multiple of $r_{p^m,n-p^m}$. 
Lucas' Theorem says that $\binom{n}{k}=\prod_i\binom{n_i}{k_i} \mod p$ 
so $\binom{k}{k-p^j}=\binom{k}{p^j}=k_j$,
giving us
 the equation 
$$r_{k,n-k}=\frac{\binom{n-p^j}{k-p^j}}{\binom{k}{k-p^j}}r_{p^j,n-p^j}.$$
If $j=m$ we are done. Otherwise, if $n>p^j+p^m$, the $m$-th coefficient in
the $p$-adic expansion of $n-p^j$
is at least $1$. Hence $\binom{n-p^j}{p^m}\not=0$, and we have two equations
\begin{align*}r_{p^j,n-p^j} &=
\frac{\binom{p^j+p^m}{p^m}}{\binom{n-p^j}{p^m}}r_{p^j+p^m,n-p^j-p^m} &
r_{p^j+p^m,n-p^j-p^m}&=
\frac{\binom{n-p^m}{p^j}}{\binom{p^j+p^m}{p^j}}r_{p^m,n-p^m}.
\end{align*}

If $n<p^j+p^m$ there is an $i<j$ such that $n_i\not=0$ and the $i$-th
coefficient in the $p$-adic expansion of $n-p^j$ is $n_i$. Hence
$\binom{n-p^j}{p^i}=n_i$ and $\binom{n-p^i}{p^m}=n_m$ and the 
four equations below move these
powers of $p$ back and forth
\begin{align*}r_{p^j,n-p^j} &=
\frac{\binom{p^j+p^i}{p^i}}{\binom{n-p^j}{p^i}}r_{p^j+p^i,n-p^j-p^i} &
r_{p^j+p^i,n-p^j-p^i}&=
\frac{\binom{n-p^i}{p^j}}{\binom{p^j+p^i}{p^j}}r_{p^i,n-p^i} \\
r_{p^i,n-p^i} &=
\frac{\binom{p^i+p^m}{p^m}}{\binom{n-p^i}{p^m}}r_{p^i+p^m,n-p^m-p^i} &
r_{p^i+p^m,n-p^i-p^m}&=
\frac{\binom{n-p^m}{p^i}}{\binom{p^i+p^m}{p^i}}r_{p^m,n-p^m}.\end{align*}
 
Combining three or five equations, respectively,  we get when
$(k,n)\not=(p^j,p^j+p^m)$ with $j<m$, the equation
$$r_{k,n-k}=ur_{p^m,n-p^m}$$
where $u$ is some element in $\bb{F}_p$.

To determine the element $u$ we will take a detour through $\bb{Z}_{(p)}$, the integers
localized at $p$.
In the $\bb{Q}$-module $\bb{Q}\{r_{1,n-1},\ldots ,r_{n-1,1}\}/{\sim}$ we let
$r_{1,n-1}=nr$. 
The formula
$\binom{k}{1}r_{k,n-k}=\binom{n-k+1}{1}r_{k-1,n-k+1}$ and induction, give
the equality $$r_{k,n-k} = \frac{n-k+1}{k}r_{k-1,n-k+1} =
\frac{n-k+1}{k}\binom{n}{k-1}r = \binom{n}{k}r$$
in $\bb{Q}\{r_{1,n-1},\ldots ,r_{n-1,1}\}/{\sim}$.

Thus, if
$n=p^{m+1}$, then 
$r_{k,n-k}=\frac{\binom{p^{m+1}}{k}}{\binom{p^{m+1}}{p^m}}r_{p^m,(p-1)p^m}=
\frac{\frac{\binom{p^{m+1}}{k}}{p}}{\binom{p^{m+1}-1}{p^m-1}}r_{p^m,(p-1)p^m},$
and when $n$ is not a power of $p$,  
$r_{k,n-k}=\frac{\binom{n}{k}}{\binom{n}{p^m}}r_{p^m,n-p^m} =
\binom{n}{k}n_m^{-1}r_{p^m,n-p^m}.$
By Lucas' Theorem, $\binom{p^{m+1}}{k}$ is divisible by
$p$ for every $k$, but neither
$\binom{p^{m+1}-1}{p^m-1}$ nor $n_m$ are divisible by $p$. Hence these
relations exists in the submodule 
$\bb{Z}_{(p)}\{r_{1,n-1},\ldots ,r_{n-1,1}\}/{\sim}\subseteq
\bb{Q}\{r_{1,n-1},\ldots
,r_{n-1,1}\}/{\sim}$. 

By the universal property of localization we get a 
map $$f:\bb{Z}_{(p)}\{r_{1,n-1},\ldots ,r_{n-1,1}\}/{\sim} \rightarrow 
  \bb{F}_p\{r_{1,n-1},\ldots ,r_{n-1,1}\}/{\sim}$$
by mapping $r_{k,n-k}$ to $r_{k,n-k}$. 

By Lucas' Theorem, $\binom{p^{m+1}-1}{p^m-1} = 1 \mod p$
and $\binom{n}{n_mp^m} = 1 \mod p$. So when $n$ is not a power of $p$, 
$f\left(\frac{\binom{n}{k}}{\binom{n}{n_mp^m}}\right) = \binom{n}{k}=u$
proving part \ref{lemma_part:n_sum_many_primes}. In particular when $k\not =
p^j,p^m$ the binomial coefficients $\binom{p^j+p^m}{k}$
are equal to $0$, proving part \ref{lemma_part:n_sum_two_primes} of the lemma.

When $n=p^{m+1}$, then $ 
f\left(\frac{\frac{\binom{p^{m+1}}{k}}p}{\binom{p^{m+1}-1}{p^m-1}}\right) =
\divByP{\binom{p^{m+1}}{k}} = u$ proving part \ref{lemma_part:one_prime}.
\end{proof}

Given a graded module $M$, let $M_{\leq n}$ denote the module
$\bigoplus_{i\leq n}M_n$, and similarly for other inequalities $<,>$ and
$\geq$. 

\begin{defn} \label{def:homoOfDegq}
Let $R$ be a commutative ring. Let $A$ and $B$ be graded $R$-algebras. 
\begin{enumerate}
 \item Let $A$ and $B$ be graded $R$-algebras. 
 An $R$-algebra homomorphism from $A$ to $B$ in degrees less than or equal to 
$q$ is an $R$-module homomorphism $f:A\rightarrow B$ which induces an
$R$-algebra homomorphism on the quotients $A/A_{>q}\rightarrow B/B_{>q}$. 
\item Let $A$ and $B$ be graded $R$-coalgebras. An $R$-coalgebra homomorphism from $A$ to $B$ in degrees less than or equal to 
$q$ is an $R$-module homomorphism $f:A\rightarrow B$ which induces an
$R$-coalgebra homomorphism $A_{\leq q}\rightarrow B_{\leq q}$.
\item Let $A$ and $B$ be graded $R$-Hopf algebras. 
An $R$-Hopf algebra homomorphism from $A$ to $B$ in degrees less than or equal to 
$q$ is an $R$-module homomorphism $f:A\rightarrow B$ which is both an $R$-algebra and $R$-coalgebra homomorphism in degrees less than or equal to $q$, and which preserves the antipode map in degrees less than or equal to $q$. 
\end{enumerate}

\end{defn}

Let $\nat$ denote the natural numbers including $0$, let $\natpos$ denote the strictly positive natural numbers and 
let $\bb{P}$ denote the set of prime powers $\{p^0,p^1,p^2\ldots\}\subseteq \nat$.

\begin{prop} \label{prop:PossibleHopfStructures} 
Let $R$ be an
$\bb{F}_p$-algebra and let $A$ be an $R$-Hopf algebra such that:
\begin{enumerate}
 \item \label{prop_ass:subHopfAlg} There is a sub $R$-Hopf Algebra
$P_R(\mu)\subseteq A$.
 \item \label{prop_ass:splitAlg} There is a chosen retraction $\pr:A \rightarrow P_R(\mu)$ 
 of the $\bb{F}_p$ module inclusion $P_R(\mu)\subset A$ that is a homomorphism of $R$-algebras in degrees less than or equal to $q$.
\item \label{prop_ass:possHopfDiag} In degree less than or equal to $q-1$ this
is a splitting as an $R$-Hopf algebra, i.e.,  
the following diagram commutes
$$\xymatrix{
A \ar[rr]^-{\pr} \ar[d]^{\psi_A} &&
P_R(\mu) \ar[d]^{\psi_{P_R(\mu)}} \\
A\otimes A \ar[rr]^-{\pr\otimes \pr} &&
\;P_R(\mu)\otimes P_R(\mu)}$$
in degree less than or equal to $q-1$. 

\end{enumerate}

Let $x$ be an element in degree $q$ in $\ker(\pr)$. Then there exist elements
$r_n\in R$ for $n\in \natpos$, and $t_{(n_1<n_2)}\in R$ for pairs $(n_1<n_2)\in
\bb{P}\times \bb{P}$, such that
the 
coproduct satisfy
$$ (\pr_{P_R(\mu)}\otimes
\pr_{P_R(\mu)})\circ \comult(x) = \sum_{n \in \natpos}r_n\redcomult(\mu^n) +
\sum_{n\in \bb{P}}r_n\divByP{\redcomult(\mu^n)} + \sum_{(n_1<n_2)\in
\bb{P}\times \bb{P}}t_{(n_1<n_2)}\mu^{n_1}\otimes \mu^{n_2}.$$
\end{prop}

Recall that $\redcomult(\mu^n)=\sum_{k=1}^{n-1} \binom{n}{k}\mu^k\otimes
\mu^{n-k}$ and that when $n$
is
a power of $p$,  $\binom{n}{k}$ is divisible by $p$ for all $0<k<n$. Hence,
$\divByP{\redcomult(\mu^n)}$ is well defined. 

Observe, that since $\redcomult(\mu^{p^i})=0$ for all $i\geq 1$, the first sum
is independent
of the values of $r_{p^i}$.
An example where this proposition applies is the dual Steenrod algebra $A_*$ 
with $P(\xi_1)\subseteq A_*$. Then 
$\redcomult(\xi_2) = \xi_1^p\otimes \xi_1 = \redcomult(\xi_1^{p+1})-\xi_1\otimes \xi_1^p$ so 
$t_{1<p}=-1$ and $r_{p+1} = 1$.

\begin{proof}
 In general $(\pr_{P_R(\mu)}\otimes
\pr_{P_R(\mu)})\circ \comult(x) = \sum_{n \in\nat}
\sum_{a+b=n}r_{a,b}\mu^{a}\otimes \mu^b$, for some $r_{a,b}\in R$.
Since $(\epsilon\otimes \id)\comult = (\id\otimes \epsilon )\comult=\id$ and
$x\in \ker(\pr)$, we
must have that 
$r_{0,n}=r_{n,0}=0$ for all $n$.

By assumption \ref{prop_ass:possHopfDiag} of the
proposition, we have
$$(\pr\otimes
\pr\otimes
\pr)(\redcomult\otimes \id)\redcomult(x) =
\sum_{n\in\nat}\sum_{d+c=n}r_{d,c}\sum_{a+b=d}
\binom{d}{b}\mu^{a}\otimes\mu^b\otimes \mu^c,$$
and
$$(\pr\otimes \pr\otimes\pr)
(\id\otimes \redcomult)\redcomult(x) =
\sum_{n\in\nat} \sum_{a+d=n}r_{a,d}\sum_{b+c=d}
\binom{d}{b}\mu^{a}\otimes \mu^b\otimes \mu^c.
$$

From coassociativity of $\redcomult$ we know that the coefficients
in front of $\mu^a\otimes \mu^b\otimes \mu^c$ in the
two expressions above must be equal. Hence there are relations
$\binom{a+b}{b}r_{a+b,c}=\binom{b+c}{b}r_{a,b+c}$, for all $a,c\geq 1$ and
$b\geq 0$. 

Given such relations, if
$n=p^{m+1}$, then by Lemma \ref{lemma:HopfAlgebraRelations}  
$r_{k,n-k}=\divByP{\binom{n}{k}}r_{p^m,(p-1)p^m}$, and we let $r_n =
r_{p^m,(p-1)p^m}$. 
If $n=p^{m_1}+p^{m_2}$ with $m_1< m_2$, then $r_{k,n-k}=0$
when $k\not=p^{m_1},p^{m_2}$. We let $r_n=r_{p^{m_2},p^{m_1}}$ and
$t_{(p^{m_1}<p^{m_2})}=
r_{p^{m_1},p^{m_2}}-r_{p^{m_2},p^{m_1}}$.
Otherwise, let
$n=n_0+n_1p^1+\ldots +n_mp^m$, with $0\leq n_i<p$ and $n_m\not=0$,
be the $p$-adic representation of $n$. Then
$r_{k,n-k}=\binom{n}{k}n_m^{-1}r_{p^m,n-p^m}$, and we let
$r_n = n_m^{-1}r_{p^m,n-p^m}$. 
\end{proof}

\begin{defn} \label{def:SplitSHopf}
 Let $B\stackrel{f}{\rightarrow} A$ be $S$-fold Hopf
algebras with $R=A_\emptyset\cong B_\emptyset$. We say that $A$ is a $B$ split $S$-Hopf algebra in degree $q\geq 0$ if  
\begin{enumerate}
\item \label{def_ass:splitSFoldHopfAlg} There is a splitting of
$S$-fold Hopf algebras $B\stackrel{f}{\rightarrow}
\widetilde{A}\stackrel{\pr}{\rightarrow} B$, where $\widetilde{A}$ is
the restriction of $A$, as in Definition \ref{def:restrictionHopfAlg}, to the full
subcategory of $\catSpanSet{S}$ not containing $S$.
\item \label{def_ass:extendSplitAlg} In degree less than or equal to $q$, the
map $\pr$ can be extended to
$A_S$, i.e., in degree less than or equal to $q$,  there is an $R$-algebra
homomorphism 
$\pr:A_S\rightarrow B_S$ (see Definition
\ref{def:homoOfDegq}) such that the following diagram commutes
$$\xymatrix{\widetilde{A}_S\ar[r] \ar[d]^{\pr_S} & A_S \ar[dl]^{\pr} \\
\;B_S,}$$
in degree less than or equal to $q$. 
\item \label{def_ass:mapHopfAlg} 
For all $s\in S$,
the map $\pr\colon\big(A_S,A_{S\setminus
s}\big)\rightarrow \big(B_S,B_{S\setminus s}\big)$ is a map of
Hopf algebras  in degree less than or equal to $q-1$. 
\end{enumerate}

\end{defn}

The next proposition is similar to the previous one, but involves
$S$-fold Hopf algebras. Although they are similar, the next proposition doesn't 
specialize to the previous one when $S$ contains exactly one
element, 
since in part \ref{def_ass:splitSFoldHopfAlg} in Definition \ref{def:SplitSHopf}
, $\widetilde{A}=R$ giving
an
impossible splitting $P_R(\mu)\rightarrow R\rightarrow P_R(\mu)$.
Given a finite set $U=\{u_1,\ldots,u_k\}$ we write
$P_R(\mu_U)$ for the polynomial ring $P_R(\mu_{u_1},\ldots, \mu_{u_{k}})$, and 
given an element $m\in\nat^V$ where $U\subseteq V$, we let
$\mu_U^{m}$ in $P_R(\mu_U)$ denote the product $\mu_{u_1}^{m_{u_1}}\cdots
\mu_{u_{k}}^{m_{u_{k}}}$.

\begin{prop} \label{prop:PossIteratedHopfStruc}
Let $A$ be a $P_R(\mu_{-})$ split $S$-Hopf algebra in degree $q$.

Let $x$ be an element in $\bigcap_{s\in S} \ker(\counit_S^s: A_S\rightarrow
A_{S\setminus s})\subseteq A_S$ of degree $q$. If $x\in \ker(\pr)$ and $s\in S$,
then there exist elements $r_b \in R$ for $b\in \natpos^{\times S}$ such that 
for every $s\in S$,
\begin{multline} \label{eq:PossIteratedHopfStruc}
\begin{bmatrix}
 \pr & \pr
\end{bmatrix}
 \circ
\comult^s_S(x) = 
\sum_{b \in
\natpos^{\times S}}
r_{b}\mu_{S\setminus s}^b \redcomult^s_S(\mu_s^{b_s}) +
\sum_{b\in \bb{P}^{\times S}}
r_{b,s}\mu_{S\setminus s}^b\divByP{\redcomult^s_S(\mu_s^{b_s})} \\
+ \sum_{b\in \bb{P}^{S\setminus s}}\sum_{c_1<c_2\in
\bb{P}\times \bb{P}}
t_{b,c_1<c_2,s}\mu_{S\setminus s}^\zb 
\begin{bmatrix}\mu_s^{c_1} & \mu_s^{c_2}\end{bmatrix},
\end{multline}
where $b_{s}$ is the $s$-th component of $b$, and 
$r_{b,s}$ and $t_{b,c_1<c_2,s}$ are
elements in $R$. 
\end{prop}

An important observation is that in the first sum, the coefficients $r_b$ are
independent of the element $s$. The $\bb{P}^{\times S}$ part in
the first sum is zero since $\redcomult^s(\mu_s^{p^i})=0$ for all $i\geq 0$. 
The map $\begin{bmatrix}
 \pr & \pr
\end{bmatrix}$ was given in Definition \ref{lemma:cubeMap}.

\begin{proof}
In this proof we will compare $\redcomult_S^{i,k}(x)$ with $\redcomult_S^{k,i}(x)$
for all pair of elements $i\not= k$ in $S$, where the definition of
$\redcomult^{\zk,\zi}_S$ is found in Definition
\ref{def:S-foldHopfAlgebra}.

For every element $\zi\in S$ the ring $A_S$ is an $A_{S\setminus i}$-Hopf
algebra and $A_{S\setminus i}$ is an $\bb{F}_p$-algebra since  $R=A_\emptyset =
\bb{F}_p$. By part \ref{def_ass:splitSFoldHopfAlg} in Definition \ref{def:SplitSHopf}, the unit
$\unit_S^i:A_{S\setminus i}\rightarrow A_S$ induces an inclusion
$P_{A_{S\setminus i}}(\mu_i)\cong P_R(\mu_{S})\otimes_{P_R(\mu_{S\setminus
i})}A_{S\setminus i}\rightarrow
\widetilde{A}_S\rightarrow A_S$ so assumption \ref{prop_ass:subHopfAlg} 
in Proposition \ref{prop:PossibleHopfStructures} is satisfied for the Hopf
algebra $(A_S,A_{S\setminus i})$. 
The splitting in assumption \ref{prop_ass:splitAlg} in Proposition
\ref{prop:PossibleHopfStructures} comes from the homomorphism $A_S\rightarrow
P_R(\mu_S)\otimes_{P_R(\mu_{S\setminus i})}A_{S\setminus i}\cong
P_{A_{S\setminus i}}(\mu_i)$ induced by $\counit_S^i$ and the splitting in
part \ref{def_ass:splitSFoldHopfAlg} in Definition \ref{def:SplitSHopf}. 
From part \ref{def_ass:mapHopfAlg}  in Definition \ref{def:SplitSHopf} this splitting induces a map of Hopf
algebras $$\big(A_S,A_{S\setminus i}\big)\rightarrow
\big(P_R(\mu_S)\otimes_{P_R(\mu_{S\setminus i})}A_{S\setminus
i},P_R(\mu_{S\setminus i})\otimes_{P_R(\mu_{S\setminus i})}A_{S\setminus
i}\big)\cong \big(P_{A_{S\setminus
i}}(\mu_i),A_{S\setminus i}\big),$$
satisfying assumption \ref{prop_ass:possHopfDiag} in Proposition
\ref{prop:PossibleHopfStructures}.

By Proposition \ref{prop:PossibleHopfStructures}, there
exist elements $r_{\zb,\zi}$ and $t_{\zb,c_1<c_2,\zi}$ in
$R$ such that
\begin{multline} \label{eq:proofPossIteratedHopfStruc}
\begin{bmatrix}\pr\\ \pr\end{bmatrix}
\circ \comult^\zi_S(x) =
\sum_{\zb\in\nat^S}r_{\zb,\zi}
\mu_{S\setminus \zi}^\zb \redcomult(\mu_\zi^{\zb_\zi})+ \sum_{\zb\in
\bb{N}^{S\setminus \zi}|\zb_\zi\in\bb{P}}r_{\zb,\zi}
\mu_{S\setminus \zi}^\zb \divByP{\redcomult(\mu_\zi^{\zb_\zi})} 
\\
+ \sum_{b\in\nat^{S}}\sum_{c_1<c_2\in \bb{P}\times \bb{P}}
t_{\zb,c_1<c_2,\zi}\mu_{S\setminus \zi}^\zb 
\begin{bmatrix}\mu_\zi^{c_1} \\ \mu_\zi^{c_2}\end{bmatrix}.
\end{multline}

Observe that if $b_i=1$, we can choose $r_{b,i}$ arbitrary. 

We will now show that if $b_i\geq 2$ and $b_k=0$ for some $k\not=i$, then
$r_{b,i}=0$. 
The counits $\counit_S^k$ and $\counit_{S\setminus i}^k$ induce a
map of Hopf algebras $(A_S,A_{S\setminus i})\rightarrow (A_{S\setminus
k},A_{S\setminus\{i,k\}})$.  Since $x$ is in
$\bigcap_{s\in S} \ker(\counit_S^s\colon A_S\rightarrow
A_{S\setminus s})$, we have $\comult^i_{S\setminus k}\circ \counit_S^k(x)=0$.
If $r_{b,i}\not=0$, then $\counit^k_S\otimes \counit^k_S(\comult^i_S(x))\not=0$
so the commutative diagram 
$$\xymatrix{
A_S \ar[r]^-{\comult^i_S} \ar[d]^-{\counit_S^k} &
A_S\otimes_{A_{S\setminus i}}A_S \ar[d]^-{\counit_S^k\otimes \counit_S^k} \\
A_{S\setminus k} \ar[r]^-{\comult_{S\setminus k}^i} &
A_{S\setminus k}\otimes_{A_{S\setminus \{i,k\}}}A_{S\setminus k}}$$
gives a contradiction. Thus $r_{b,i}=0$.

From part \ref{def_ass:mapHopfAlg}  in Definition \ref{def:SplitSHopf}, we get a
commutative diagram
$$\xymatrix{
\ker(\counit_S^i) \ar[r]^-{\redcomult_S^i} \ar[d]^-{\redcomult_S^i} &
 A_S^{\{i\}} \ar[r]^-{\redcomult_S^{\{i\},k}} &
 A_S^{\{i,k\}} \ar[d]^-{\pr} \\
 A_S^{\{i\}} \ar[r]^-{\pr} &
 P_R(\mu_S)^{\{i\}} \ar[r]^-{\redcomult_S^{\{i\},k}} &
P_R(\mu_S)^{\{i,k\}},}$$
in degree less than or equal to $q$, 
where the composition of the two morphisms on the top is the definition of
$\redcomult_S^{k,i}$. The diagram commutes in degree less than or equal to $q$, 
and not just $q-1$, since we use the reduced coproduct.

From this diagram we have the formula
\begin{align*}
\begin{bmatrix}
 \pr & \pr \\ \pr &\pr
\end{bmatrix} 
\circ \redcomult^{\zk,\zi}_S(x)
&= 
 \sum_{\zb\in \natpos^S}\sum_{ 0< \za_\zi <
\zb_\zi}\sum_{0< \za_\zk< \zb_\zk}
r_{\zb,\zi}\binom{\zb_\zk}{\za_\zk}\binom{\zb_\zi}{\za_\zi}
\mu_{S\setminus\{\zi,\zk\}}^\zb
\begin{bmatrix}\mu_{\zi}^{\za_\zi} \mu_\zk^{\za_\zk}&
\mu_\zk^{\zb_\zk-\za_\zk} \\ \mu_\zi^{\zb_\zi-\za_\zi} & 1
\end{bmatrix} \\
&+
\sum_{\zb\in \natpos^S|\zb_\zi\in \bb{P}}
\sum_{0< \za_\zi < \zb_\zi}
\sum_{0< \za_\zk< \zb_\zk}
r_{\zb,\zi}\binom{\zb_\zk}{\za_\zk}\divByP{\binom{\zb_\zi}{\za_\zi}}
\mu_{S\setminus\{\zi,\zk\}}^\zb
\begin{bmatrix}\mu_{\zi}^{\za_\zi}\mu_\zk^{\za_\zk} &
 \mu_\zk^{\zb_\zk-\za_\zk} \\ 
 \mu_\zi^{\zb_\zi-\za_\zi}  & 1
\end{bmatrix} \\
&+
\sum_{\zb\in \natpos^{S \setminus \zi}}
\sum_{c_1<c_2\in \bb{P}\times \bb{P}}
\sum_{0< \za_\zk< \zb_\zk}
t_{\zb,c_1<c_2,\zi}\binom{\zb_\zk}{\za_\zk}
\mu_{S\setminus\{\zi,\zk\}}^\zb
\begin{bmatrix}\mu_{\zi}^{c_1}\mu_\zk^{\za_\zk} &
 \mu_\zk^{\zb_\zk-\za_\zk} \\ 
 \mu_\zi^{c_2} & 1 
 \end{bmatrix}.
\end{align*}

The three lines correspond  to the three
summands in equation \ref{eq:proofPossIteratedHopfStruc}. 

Since $A$ is an $S$-fold Hopf algebra,
$\redcomult^{\zk,\zi}_S=\redcomult^{\zi,\zk}_S$ so 
$$\begin{bmatrix}
 \pr & \pr \\ \pr &\pr
\end{bmatrix} \circ\redcomult^{\zk,\zi}_S(x)=
\begin{bmatrix}
 \pr & \pr \\ \pr &\pr
\end{bmatrix} \circ\redcomult^{\zi,\zk}_S(x).$$
In this equation we will now compare the coefficient in front of
$\mu_{S\setminus\{\zi,\zk\}}^\zb
\begin{bmatrix}\mu_{\zi}^{\za_\zi}\mu_\zk^{\za_\zk} &
 \mu_\zk^{\zb_\zk-\za_\zk} \\\mu_\zi^{\zb_\zi-\za_\zi}  & 1
\end{bmatrix}$ for $b\in\natpos^S$, with $0<a_j<b_j$.

We will say that an integer
$b_i\geq 2$ is;
\begin{itemize}
 \item \type{1} if $b_i$ is equal to a
power of the prime $p$,
\item \type{2} if $b_i$ is equal to a sum of two distinct
powers of $p$,
\item and \type{3} otherwise. 
\end{itemize} 
There are six possible cases since $b_i$ and $b_k$ may be interchanged. 

\emph{Case 1, both $b_i$ and $b_k$ are \type{3}:} 

We get the equation
$$\binom{\zb_\zk}{\za_\zk}\binom{\zb_\zi}{\za_\zi} r_{\zb,\zi}=
 \binom{\zb_\zi}{\za_\zi}\binom{\zb_\zk}{\za_\zk} r_{\zb,\zk}.$$
Since neither $b_i$ nor $b_k$ are of \type{1}, there exists integers
$0<a_i<b_i$ and $0<a_k<b_k$ such that $\binom{\zb_\zi}{\za_\zi}\not=0$ and
$\binom{\zb_\zk}{\za_\zk}\not=0$. Thus $ r_{\zb,\zi}=r_{\zb,\zk}$.

\emph{Case 2,  $b_i$ is \type{2} and $b_k$ is \type{3}:}

Let $b_i=p^j+p^l$ with $j<l$. 
When $a_i=p^j$ we get the equation
$$ \binom{\zb_\zk}{\za_\zk}\binom{\zb_\zi}{p^j} r_{\zb,\zi} +
\binom{\zb_\zk}{\za_\zk}t_{\zb,p^j<p^l,\zi} =
 \binom{\zb_\zi}{p^j}\binom{\zb_\zk}{\za_\zk} r_{\zb,\zk}, $$
and when $a_i=p^l$ we get the equation
$$\binom{\zb_\zk}{\za_\zk}\binom{\zb_\zi}{p^l} r_{\zb,\zi} =
 \binom{\zb_\zi}{p^l}\binom{\zb_\zk}{\za_\zk} r_{\zb,\zk}.$$
 By Lucas' Theorem $\binom{\zb_\zi}{p^j}=
\binom{\zb_\zi}{p^l}=1$. 
Since $\zb_\zk$ is not of \type{1}, there exists an $a_k$ such that
$\binom{\zb_\zk}{\za_\zk}\not=0$. The last equation thus gives 
$r_{\zb,\zi}=r_{\zb,\zk}$, and the second equation becomes
$r_{\zb,\zi}+t_{\zb,p^j<p^l,\zi}=r_{\zb,\zk}$, so $t_{\zb,p^j<p^l,\zi}$ must be
equal to $0$. 

The rest are proven similarly, and we just state the results. 

\emph{Case 3,  $b_i$ is \type{1} and $b_k$ is \type{3}:}
 In this case $r_{\zb,\zi}=0$.

\emph{Case 4, both $b_i$ and $b_k$ are \type{2}:}
In this case
$r_{\zb,\zi}=r_{\zb,\zk}$.

\emph{Case 5, $b_i$ is \type{2} and $b_k$ is \type{1}:}
In this case $r_{b,k}=0$, and $r_{b,i}$ is in-determined.

\emph{Case 6, both $b_i$ and $b_k$ are \type{1}:}
In this case both $r_{b,i}$ and $r_{b,k}$ are in-determined. 

From these six cases we will now deduce equation \ref{eq:PossIteratedHopfStruc}
in the proposition. 

Consider an $S$-tuple $b\in\natpos^S$. These fall in five classes:
\begin{enumerate}
 \item All $b_i$'s are equal to $1$.
 \item  All $b_i$-s are of \type{1}, or equal to $1$. 
 \item Exactly one $b_i$ is of \type{2}, and the rest are of \type{1} or equal to $1$.
 \item At least two $b_i$-s are of \type{2}, and the rest are of \type{1} or equal to $1$.
\item At least one $b_i$ is of \type{3}.
\end{enumerate}

We will now consider these cases one by one. 

\begin{enumerate}
 \item We can choose $r_{b,i}$ arbitrary since they don't affect the sum,
  so we let $r_b=0$.
  
\item 
 This correspond to the middle sum
in  equation \ref{eq:PossIteratedHopfStruc}.

\item  Assume $b_i$ is of \type{2}.
By case 5, for all $b_k$ of \type{1} $r_{b,k}=0$, but nothing
can be said about $r_{b,i}$ nor $t_{b,p^j<p^l,i}$. We choose $r_b=r_{b,i}$,
and this correspond to one summand in the first sum and one summand in the last
sum in equation
\ref{eq:PossIteratedHopfStruc}.

\item  Assume $b_i=p^{j_i}<p^{l_i}$ and $b_k=p^{j_k}<p^{l_k}$
are of \type{2}.
Then using case 4 twice, we get that 
$t_{b,p^{j}<p^{l},k}=t_{b,p^{j}<p^{l},k}=0$ and $r_{b,k}=r_{b,i}$. If $b_j$ is
of \type{1}, case 5 shows that $r_{b,j}=0$. We choose $r_b=r_{b,i}$, and this
correspond to the first sum in equation \ref{eq:PossIteratedHopfStruc}.

\item  Case 2
shows that for all $b_k=p^j+p^l$ of \type{2}, 
$t_{b,p^j<p^l,k}=0$ and $r_{b,k}=r_{b,i}$. 
From case 3, $r_{b,k}=0$ for all $k$
with
$b_k$ of
\type{1}. Finally, case 1 says that $r_{b,k}=r_{b,i}$ for all $b_k$ of
\type{3}, so we let $r_{b}=r_{b,i}$. This also correspond to the first sum in 
equation \ref{eq:PossIteratedHopfStruc}.

\end{enumerate}

\end{proof}

\section{Calculating the Homotopy Groups of {$\loday_{T^n}\HF$}}
\label{sec:calcTorus}
In this section we will calculate the homotopy groups 
$\Loday(T^n)$ for $n\leq p$. 
We will use the bar spectral sequence, and the multifold Hopf algebra structure of 
$\Loday(T^n)$ to make the calculation. 

The proof of the Theorem \ref{thm:smash_over_torus} calculating $\Loday(T^n)$ , is very long and spread over several lemmas,
 so we will now give a sketch of how the proof is structured.
 
The outermost layer is a double induction argument on the 
dimension $n$ of the Tori and the degrees of the elements in $\Loday(T^n)$, and 
this induction arguments uses several properties of $\Loday(T^n)$ 
all of which must be proven in the induction step and is thus included in 
Theorem \ref{thm:smash_over_torus}. 

The main calculation in the induction step, is done using the bar spectral sequence $E^*(T^\uord{n})$. We use the B\"okstedt spectral sequence, 
to identify the $E^2$-page of the bar spectral sequence and to show that  
all $d^2$-differentials are zero. 

The collection $\Loday(T^n)$ can be endowed with a multifold Hopf algebra structure, and the degrees, modulo $2p$, 
of the simultaneously primitive elements in 
 $\Loday(T^\uord{n})$ is calculated. Using this we can show that there are no other
non-zero differentials, and hence calculate $E^\infty(T^\uord{n})$.

From $E^\infty(T^\uord{n})$ we can choose a set of $\bb{F}_p$-algebra generators for
$\Loday(T^\uord{n})$, and in a couple of steps we perturb these sets of generators such that they have nicer and nicer properties, and using these
nicer properties we are able to prove all the remaining statements in Theorem \ref{thm:smash_over_torus}, and thus finish the induction step. 
In particular, we need the multifold Hopf algebra structure to get hold of the multiplicative
structure in $\Loday(T^\uord{n})$.

We will now construct a family of bar spectral sequences that will be the
backbone in our calculation of $\Loday(T^\uord{n})$. 

The attaching maps in the
$CW$-structures yield cofiber sequences
$$  \xymatrix{
  S^{n-1} \ar[rr]^-{\attach^\n} &&
  T^\uord{n}_{n-1} \ar[r]  & 
  \, T^\uord{n}.
  }$$
giving an equivalence of commutative $\HF$-algebra spectra
$$B(\loday_{D^n}\HF,\loday_{S^{n-1}}\HF,\loday_{T^\uord{n}_{n-1}}\HF)\simeq
\loday_{T^\uord{n}}\HF.$$
  By Proposition~\ref{prop:barSS}, there is an $\bb{F}_p$-algebra bar
spectral sequence
\begin{align*}
 E^2(T^\uord{n})=\tor^{\Loday(S^{n-1}) }
(\Loday(T^{\uord{n}}_{n-1}), \bb{F}_p) &\Rightarrow
\Loday(T^\uord{n}).
\end{align*}

The spectral sequence is indexed such that the differentials are of the form $d^r:E^r_{s,t}\rightarrow
E^r_{s-r,t+r-1}$. The differentials are only given up to multiplication with a unit. 

For each $i\in \uord{n}$, the pinch of the $i$-th circle in $T^\uord{n}$
induces a map  of cofiber sequences 
 $$\xymatrix{
  S^{n-1} \ar[r]^{\attach^\uord{n}} \ar[d]  &
  T^\uord{n}_{n-1} \ar[d] \ar[r] &
  T^\uord{n}  \ar[d] \\
  S^{n-1}\vee S^{n-1} \ar[r] &
  T^\uord{n}_{n-1}\amalg_{T^{\uord{n}\setminus i}}T^\uord{n}_{n-1}\ar[r] &
  T^{\uord{n}}\amalg_{T^{\uord{n}\setminus i}}T^\uord{n}, }$$
  inducing a map of simplicial spectra  
\begin{multline*}B(\HF,\loday_{S^{n-1}}\HF,\loday_{T^\uord{n}_{n-1}}
\HF) \\ \rightarrow B(\HF,\loday_{S^{n-1}}\HF\wedge_{\HF}\loday_{S^{n-1}}\HF,
\loday_{T^\uord{n}_{n-1}}\HF
\wedge_{\loday_{T^{\uord{n}\setminus i}}\HF}\loday_{T^\uord{n}_{n-1}}\HF ) \\ \simeq
B(\HF,\loday_{S^{n-1}}\HF,\loday_{T^\uord{n}_{n-1}}\HF)\wedge_{\loday_{T^{\uord{n}\setminus i}}\HF}
B(\HF,\loday_{S^{n-1}}\HF,\loday_{T^\uord{n}_{n-1}}\HF).\end{multline*}
  
Hence  by Proposition \ref{prop:barSSisCoalg}, if
$E^r(T^\uord{n})$ is flat as an $\Loday(T^{\uord{n}\setminus i})$-module then
$E^*(T^\uord{n})$ is a spectral sequence of $\Loday(T^{\uord{n}\setminus i})$-Hopf
algebras, and if $\Loday(T^\uord{n})$ is flat as an $\Loday(T^{\uord{n}\setminus i})$-module, then $\Loday(T^\uord{n})$  is an $\Loday(T^{\uord{n}\setminus i})$-Hopf algebra
and the spectral sequence converges to $\Loday(T^\uord{n})$  as an $\Loday(T^{\uord{n}\setminus i})$-Hopf algebra.

\begin{defn} \label{def:labelB_U}
Given a finite ordered set $S=\{s_1<\ldots <s_n\}$ we define an $S$-labeled 
admissible word to be an admissible word of length $n$,
where the first letter is labeled with $s_n$, the second with $s_{n-1}$, and so
forth. 
We define $B_S$ to be the $\bb{F}_p$-Hopf algebra that is a tensor product of
exterior algebras
 on all $S$-labeled  admissible monic words of odd degree and divided power
algebras on all $S$-labeled admissible monic words of even degree. We let
$B_{\emptyset}=\bb{F}_p$ be generated by the empty word in degree zero. 

The operator $\sigma_{s_n}:B_{S\setminus x_n}\rightarrow B_S$ is determined by  $\sigma_{s_n}(x)= \sBar_{s_n} x$ and $\sigma(x) = \sBar_{s_n}^0 x$, when $x$ is an $S\setminus s_n$-labeled	 admissible word of even and odd degree, respectively.
\end{defn}
Forgetting the labels on the letters induces an $\bb{F}_p$-Hopf
algebra isomorphism between $B_S$ and
$B_n$. 
An example of an $S$-labeled word of length $3$
is $\sBar_{s_3}^k\sBar_{s_2}\mu_{s_1}$.

Given a finite subcategory $\simp\subseteq \catFinSet$, we define 
$$T^\simp = \colim_{U\in \simp}T^U.$$ 

\begin{thm} \label{thm:smash_over_torus}
Given $1\leq k\leq p$ when $p\geq 5$ and $1\leq k\leq 2$ when $p=3$, let $\simp$
be a finite subcategory of $\catFinSet$ of
dimension at
most $k$ and let $V$ be a non-empty set in $\catFinSet$ of
cardinality at most $k$. 
\begin{enumerate}
 \item \label{thm_part:factors_Fp} The map $L(f^\uord{k}):\Loday(
S^{k-1})\rightarrow \Loday(T^\uord{k}_{k-1})$
factors through $\bb{F}_p$. 
\item \label{thm_part:SScollapses} When $k\geq 2$, the spectral sequence
$E^*(T^\uord{k})$
collapses on the $E^2$-term.
\item \label{thm_part:algIso}There is a natural isomorphism 
$$\torusIso_V:\Loday(T^V)\rightarrow \bigotimes_{U\subseteq V } B_U,$$ where $B_U$ is
described in Definition~\ref{def:labelB_U}, of functors from $\catSpanSet{\uord{k}}$ to $\bb{F}_p$-algebras 
inducing an isomorphism of $\bb{F}_p$-algebras
$$\Loday(T^\simp)\cong \colim_{U\in \simp}
\Loday(T^U) \cong \bigotimes_{U\in \simp} B_U.$$
In the sequel, we will identify $\Loday(T^V)$ and $\bigotimes_{U\subseteq V } B_U$ by means of $\torusIso_V$.

\item \label{thm_part:wedge_projection} 
For every $v\in V$, the projection maps $$ \pr:\Loday(T^V)\stackrel{\alpha}{\cong}
\bigotimes_{U\subseteq V} B_U\rightarrow \bigotimes_{i\in
V}B_{\{i\}}.$$  induces maps of Hopf algebras 
$$(\Loday(T^{V}),\Loday(T^{V\setminus v}))\stackrel{\pr}{\rightarrow}(P(\mu_V),P(\mu_{V\setminus v})).$$

\item \label{thm_part:sigmaElements} Assume $|V|\geq 2$ and let $v$ be the
greatest integer in $V$.

The operator
$$\sOper:\Loday(T^{V\setminus v})\rightarrow \Loday(T^V)$$ is determined by the
fact that $\sOper$ is a derivation and that 
$\sOper(z)= \sBar_v z$ and $\sOper(z) = \sBar_v^0z$ when $\emptyset \not=
U\subseteq V\setminus v$ and $z$ is an
$U$-labeled
admissible word in $ B_{U}\subseteq \Loday(T^{V\setminus
v})$ of even and odd
degree, respectively.

In particular, for any $z\in \Loday(T^{V\setminus v})$, $\sOper(z)$ is in the
kernel of $$ \pr:\Loday(T^V)\stackrel{\alpha}{\cong}
\bigotimes_{U\subseteq V} B_U\rightarrow \bigotimes_{i\in
V}B_{\{i\}}.$$

\item \label{thm_part:sphere_projection} There is
a commutative diagram 
$$\xymatrix{\Loday(T^V) \ar[r]^-\torusIso \ar[d]^-{\sphereproj^V} &
  \bigotimes_{U\subseteq V} B_U \ar[d]^{\pr} \\
 \Loday(S^V) \ar[r]^-\cong &
 B_V,}$$
where the bottom isomorphism is the one from
Theorem~\ref{thm:smash_over_circles}, together with the canonical isomorphism
$B_{|V|}\cong B_V$ given by labeling the words in $B_{|V|}$.

\end{enumerate}
\end{thm}

The range $k\leq p$ comes from all the lemmas in Section
\ref{sec:primElm} concerning the degrees of primitive elements. It is
possible that this range could be improved by getting better control of the
degrees of the primitive elements.

When $k=p=3$ part~\ref{thm_part:factors_Fp} and~\ref{thm_part:SScollapses} of
the theorem still holds, but we are not able to determine the multiplicative
structure of $\Loday(T^\uord{3})\cong E^\infty(T^\uord{3})\cong
\Loday(T^\uord{3}_2)\otimes B_3$. This is because the degrees of
$\gamma_{p^{k+1}}(\sBar^0\sBar\mu)\in \Gamma(\sBar^0\sBar\mu) =B_3$ equals the
degrees of $\mu_1^{p^k+p^{k+1}}\mu_2^{p^k}\mu_3^{p^k}$. Thus, we can't use
Proposition~\ref{prop:PossIteratedHopfStruc} to show that
$(\gamma_{p^{k+1}}(\sBar^0\sBar\mu))^p$ is a simultaneously primitive element in the $\uord{3}$-fold Hopf algebra $\Loday(T^\uord{3})$. 

The idea to look at the simultaneously primitive elements to show that the
spectral sequence collapses on the $E^2$-term originated from a note by John
Rognes, where he showed that the spectral sequence $E^*(T^\uord{3})$ collapses on the $E^2$-term.

\begin{remark}
It should be possible to prove a similar result for 
 $V(0)_*(\loday_{T^\uord{n}}\HZ)$. The
difference would be the degrees of the elements in the rings, and thus the
degrees of the simultaneously primitive elements. The arguments in Section
\ref{sec:calcSphereHF} would thus have to be adjusted for these new elements,
and possibly you would want to work modulo $2p^2$ instead of modulo $2p$.
\end{remark}

We need this corollary to
identify the $E^2$-term $E^2(T^\uord{n})$ and show that there are no
$d^2$-differentials. 

\begin{corr} \label{corr:noD2diffTorus} 
Given $n$, assume Theorem~\ref{thm:smash_over_torus} holds when $1\leq
k\leq n-1$.
Given $m\geq0$, if $\attach^\uord{n}\colon \Loday(S^{n-1})\rightarrow
\Loday(T^\uord{n}_{n-1})$ factors through $\bb{F}_p$ in degrees less than or
equal to $2pm-1$ and the spectral
sequence $E^*(T^\uord{n})$ collapses in total
degrees
less than or equal to $2pm-1$ (that is $E^2(T^\uord{n})=E^\infty(T^\uord{n})$ in
these degrees), then:
\begin{enumerate}
 \item The map $\attach^\uord{n}\colon \Loday(S^{n-1})\rightarrow
\Loday(T^\uord{n}_{n-1})$ factors through $\bb{F}_p$ in degrees less than or
equal to $2p(m+1)-2$.
\item The spectral sequence $E^*(T^\uord{n})$ collapses in total
degrees less than or equal to $2p(m+1)-2$.
\end{enumerate}

\end{corr}
\begin{proof}
 From Lemma \ref{lemma:bokstedttorus} we know that as an $\bb{F}_p$-module 
$$H_*(\loday_{T^n}\HF)_{\leq 2p(m+1)-2} \cong (A_*\otimes \bigotimes_{U\subseteq
\uord{n}} B'_U))_{\leq 2p(m+1)-2}.$$
Since $\loday_{T^\uord{n}}\HF$ is a generalized Eilenberg Mac~Lane spectrum, the
Hurewicz homomorphism induces an isomorphism between the 
$\bb{F}_p$-modules  $A_*\otimes
\Loday(T^\uord{n})\cong A_*\otimes E^\infty(T^\uord{n})$ and 
$H_*(\loday_{T^\uord{n}}\HF)$.

The $E^1$-terms of the bar spectral sequence $E^1(T^\uord{n})$ 
is the two-sided bar complex
$$E^1_{s,*}(T^\uord{n})=B_s(\Loday(T^{\uord{n}}_{n-1}),B_{n-1},\bb{F}_p)\cong
\Loday(T^{\uord{n}}_{n-1})\otimes\B_{n-1}^{\otimes s}\otimes
\bb{F}_p$$ where the $\B_{n-1}^{\otimes s}$ module structure on $\Loday(T^{\uord{n}}_{n-1})$ is induced by $\attach^\uord{n}$, 
and the
differential
$d^1:E^1_{s,t}(T^\uord{n})\rightarrow E^1_{s-1,t}(T^\uord{n})$ 
 is given by
$$d^1(a\otimes b_1\otimes\cdots \otimes b_{s+1}) =
a\attach^\uord{n}(b_1)\otimes b_2\otimes \cdots \otimes b_{s+1} +
\sum(-1)^i
a\otimes b_1\otimes \cdots \otimes b_ib_{i+1}\otimes\cdots \otimes b_{s+1}.$$

If $\attach^\uord{n}$ factors through $\bb{F}_p$ in degrees less than $l$,
then 
$$E^2(T^\uord{n})=
\tor^{\Loday(S^{n-1})}(\Loday(T^\uord{n}_{n-1}),\bb{F}_p)
\cong\Loday(T^\uord{n}_{n-1})
\otimes\tor^{\Loday(S^{n-1})} (\bb{F}_p,\bb{F}_p) =\Loday(T^\uord{n}_{n-1 }
)\otimes B_\uord{n}.$$
in bidegrees $(s,t)$ with $t<l$.  Furthermore, 
$$E^2_{0,l}(T^\uord{n})\cong \Loday(T^\uord{n}_{n-1})/\im(\attach^{\uord{n}})$$

If $\attach^\uord{n}$ doesn't factor through
$\bb{F}_p$ in degrees $l\leq 2p(m+1)-2$ then the dimension of $A_*\otimes
E^2(T^\uord{n})$ in total degree $l$ is smaller than the dimension of
$H_*(\loday_{T^n}\HF)$ in degree $l$, giving us a contradiction. 

Thus $\attach^\uord{n}$ factors through $\bb{F}_p$ in degrees less than or
equal to $2p(m+1)-2$.

By a similar argument if there are any non-zero $d^r$-differentials in
$E^r(T^\uord{n})$ starting in total degrees less than or equal to $2p(m+1)-1$,
the dimension of $A_*\otimes
E^r(T^\uord{n})$ in the degree of the image of this differential will be
smaller than the dimension of $H_*(\loday_{T^\uord{n}}\HF)$ in this degree.

Thus the spectral sequence $E^*(T^\uord{n})$ collapses in total
degrees less than or equal to $2p(m+1)-2$.
\end{proof}

\begin{proof}[Proof of Theorem~\ref{thm:smash_over_torus}]
 The proof is by induction. Given $n$, with $1\leq n\leq p$ when $p\geq 5$ and
$1\leq n\leq 2$ when $p=3$, assume the theorem holds
for all $k$ with $1\leq k< n$. In this proof we denote this \emph{the torus-induction hypothesis}.
The only place in the proof where there is a difference between $p=3$ and
$p\geq 5$ is when we invoke Corollary~\ref{corr:degOfComult} in the proof of
part~\ref{thm_part:algIso}.

When $n=2$, the theorem holds since
$\Loday(T^U_1)\cong P(\mu_U)$.

\emph{Proof of part~\ref{thm_part:factors_Fp}~and~\ref{thm_part:SScollapses} :}
We prove it by induction on the degrees of elements in part
\ref{thm_part:factors_Fp} and total degrees in part~\ref{thm_part:SScollapses}.
Given $m$, assume that part~\ref{thm_part:factors_Fp} and
\ref{thm_part:SScollapses}
 holds in degrees less than or equal to $2pm-1$. This is trivially true
when $m=0$. 

By Corollary~\ref{corr:noD2diffTorus}, part~\ref{thm_part:factors_Fp} and
\ref{thm_part:SScollapses} holds in
(total) degrees less than or equal to $2p(m+1)-2$.
We must thus show that they hold in degree $2p(m+1)-1$. 

The attaching map
$\attach^\uord{n}:\Loday(S^{n-1})\rightarrow \Loday(T^\uord{n}_{n-1})$ is
determined by what it does on the set of algebra generators in
$\Loday(S^{n-1})$ given by the monic words of length $n-1$, and by Lemma~\ref{lemma:indecWord} there are no
such element in degrees $-1$ modulo $2p$, and hence $\attach^\uord{n}$
factors through $\bb{F}_p$ in degrees less than or equal to $2p(m+1)-1$. So,
in vertical degrees less than or equal to $2p(m+1)-1$ the K\"unneth isomorphism yields an
$\Loday(T^\uord{n}_{n-1})$-module isomorphism 
$$E^2(T^\uord{n})=\tor^{\Loday(S^{n-1})}(\Loday(T^\uord{n}_{n-1}),\bb{F}_p)\cong
\Loday(T^\uord{n}_{n-1})\otimes \tor^{\Loday(S^{n-1})}(\bb{F}_p,\bb{F}_p)\cong
\Loday(T^\uord{n}_{n-1})\otimes B_n.$$

It remains to show that there are no $d^r$-differentials in $E^r(T^\uord{n})$
starting in total degrees $2p(m+1)$.  For every $i$ in $\n$, $E^2(T^{\uord{n}})$
is an
$\Loday(T^{\uord{n}\setminus i})$-Hopf algebra
spectral
sequence, since $E^2(T^{\uord{n}})$ is flat over $\Loday(T^{\uord{n}\setminus i})$. 
The Hopf algebra structure on $E^2(T^\uord{n})$ is the
tensor product of the $\Loday(T^{\uord{n}\setminus i})$-Hopf algebra structures on $
\Loday(T^\uord{n}_{n-1})$ and the $\bb{F}_p$ Hopf algebra structure on $B_n$. Thus, by Proposition
\ref{prop:shortestDiff}, a shortest non-zero
differential in lowest total degree, must go
 to a primitive element in the $\Loday(T^{\uord{n}\setminus i})$-Hopf algebra structure.
Hence, if a shortest non-zero differential
starts in total degree $2p(m+1)$, there must be elements in degree $2p(m+1)-1$
that are primitive in the $\Loday(T^{\uord{n}\setminus i})$-Hopf algebra structure for
all $i\in \uord{n}$.

The 
$\Loday(T^{\uord{n}\setminus i})$-primitive elements in $\Loday(T^\uord{n}_{n-1})\otimes
B_n$ are by
the graded version of Proposition
3.12 in \cite{MilnorMoore65} linear combinations of primitive
elements in $\Loday(T^\uord{n}_{n-i})$ and $B_n$. By the graded version of Proposition
3.12 in \cite{MilnorMoore65} the module of
$\Loday(T^{\uord{n}\setminus i})$-primitive elements in
$B_n$ is
$\Loday(T^{\uord{n}\setminus i})\{X_n\}$, where $X_n$ is the set of monic words
in $B_n$. The intersection $\bigcap_{i\in
\uord{n}}\Loday(T^{\uord{n}\setminus i})\{X_n\}$ is equal to $\bb{F}_p\{X_n\}$ since
$\bigcap_{i\in
\uord{n}}\Loday(T^{\uord{n}\setminus i})=\bb{F}_p$. 
Thus, the module of elements in $B_n\subseteq E^2(T^\uord{n})$ that are
primitive in  the $\Loday(T^{\uord{n}\setminus i})$-Hopf algebra structure for every
$i\in \uord{n}$
 is $\bb{F}_p\{X_n\}\subseteq B_n$, 
which is isomorphic to the module of $\bb{F}_p$-primitive elements in $B_n$,
under the projection map $E^2(T^\uord{n})\rightarrow B_n$. By
Proposition
\ref{lemma:B_n_primitive} there are no $\bb{F}_p$-primitive elements in
$B_n$ in degrees $-1$ modulo $2p$ when $n\leq 2p$. Hence, there are no
differentials starting in total degree $2p(m+1)$ that have target in filtration
1 or higher.

It remains to show that there are no differentials starting in total degree
$2p(m+1)$ that have non-zero target in filtration $0$. This is only possible if
there are $\uord{n}$-fold primitive elements in
$\Loday(T^\uord{n}_{n-1})$ in the target of the differential. 
If $z$ is an indecomposable element in $B_n$ in degree $2p(m+1)$, 
 Corollary~\ref{corr:dimOfnFoldPrim} says there are
no $\n$-fold primitive elements in $\Loday(T^\uord{n}_{n-1})$ in degree
$2p(m+1)-1$ when $n\leq p$. 

Hence, there are no differentials in
$E^*(T^{\uord{n} })$ when $n\leq p$, so $E^*(T^\uord{n})$ collapses on the
$E^2$-term. Since $E^2(T^\uord{n})\cong E^\infty(T^\uord{n})$, $E^2(T^\uord{n})$ is flat as an $\Loday(T^{\uord{n}\setminus i})$-module, and $\Loday(T^\uord{n})$ is flat as an $\Loday(T^{\uord{n}\setminus i})$-module, 
the spectral sequence converges to $\Loday(T^\uord{n})$ as an $\Loday(T^{\uord{n}\setminus i})$-Hopf algebra.

\emph{Proof of part~\ref{thm_part:algIso} and 
\ref{thm_part:wedge_projection}: }
 We will only show the theorem for the set $V=\uord{n}$.
  
\begin{enumerate}
 \item Let $\basis_{1}^{\text{odd}}$ and $\basis_{1}^{\text{even}}$ and
  be the sets of all admissible
words of length
$n$ starting with $\sBar$ or $\sBar^0$, respectively, and let $\basis_{1}=\basis_{1}^{\text{odd}}\cup \basis_{1}^{\text{even}}$. 
\item Let $\basis_{2}$ be the set of all
admissible
words of length $n$ that starts with $\varphi^i$ or $\sBar^{i+1}$ for $i\geq
0$.
\end{enumerate}
The set $\basis_{2}$ only contains elements in degrees $0$ modulo $2p$, and the sets $\basis_{1}$ and
$\basis_{2}$ generate $B_n$ as an $\bb{F}_p$-algebra.

We can also think of  $\basis_{1}^{\text{odd}}$ and
$\basis_{1}^{\text{even}}$
as sets of elements 
 in $ E^2_{1,*}(T^\uord{n})$ of odd and even degree, respectively,
and $\basis_{2}$ as a set of elements in 
$E^2_{s,*}(T^\uord{n})$ with $s\geq
2$, and together they generate $E^2(T^\uord{n})$ as an
$\Loday(T^\uord{n}_{n-1})$-algebra. 

We will define sets $\obasis_{1}$ and
$\obasis_{2}$ of elements in $\Loday(T^\uord{n})$, with bijections $\obasis_{1}\cong \basis_1$ and 
$\obasis_2\cong \basis_2$, which  generate 
$\Loday(T^\uord{n})$ as an $\Loday(T^\uord{n}_{n-1})$-algebra. 

The $\bb{F}_p$-isomorphism $\torusIso$ is then the composite:
$$\xymatrix{
 \torusIso:\Loday(T^\uord{n})\ar[r]^-{\alpha'} \ar[r] & 
 \Loday(T^\uord{n}_{n-1})\bigotimes B_n \ar[r]^-{\id\otimes\alpha''} \ar[r]&
 \Loday(T^\uord{n}_{n-1})\bigotimes B_\uord{n}
 }$$
 where $\alpha'$ is the $\bb{F}_p$-isomorphism induced by the bijections $\obasis_{i}\cong\basis_i$ and $\alpha''$ is the 
$\bb{F}_p$-algebra isomorphism
$ B_n\cong B_\uord{n}$
given by labeling the words in $B_n$.

Let $\obasis_{1}$ be
the set of elements
$\sOper(z_{\uord{n-1}})$ in $\Loday(T^\uord{n})$, where $z_{\uord{n-1}}$ runs
over all
admissible words in $B_{\uord{n-1}}\subseteq
\Loday(T^\uord{n-1})$.  Define the bijection $\basis_1\cong \obasis_1$ by mapping 
$\sBar(z)$ or $\sBar^0(z)$, where $z$ is an admissible word in $B_{n-1}$ of odd or even degree, respectively, to $\sOper(z_{\uord{n-1}})$, where $z_{\uord{n-1}}$ is the labeled version of $z$. 

We will use induction on the degrees of elements in $\Loday(T^{\uord{n}})$ to 
define the bijection $\obasis_2\cong\basis_2$
and to prove part~\ref{thm_part:algIso} and 
\ref{thm_part:wedge_projection}.

Let \emph{the degree induction hypothesis} be the following hypothesis:
\begin{enumerate}
\item \label{induction1} In degrees less than $2pl$ we have lifted all elements in $\basis_2$ to elements in $\obasis_2$ and in degrees less than $2pl$ this lift induces an isomorphism
$\alpha$ satisfying part~\ref{thm_part:algIso} of Theorem \ref{thm:smash_over_torus}. 
 \item  \label{induction2}In degrees less than $2pl$, we have defined a homomorphism $\pr$ satisfying part 
\ref{thm_part:wedge_projection}  of Theorem \ref{thm:smash_over_torus}. 
\end{enumerate}
See Definition \ref{def:homoOfDegq} for the definition of an homomorphism in degrees less than $2pl$.
When $l=0$ there is nothing to prove. Assume we have proved it for $m$ when $m = l$.

First we proof hypothesis \ref{induction1} of the degree induction hypothesis for $m+1$. 
For each $x\in \basis_{2}$, let 
\begin{equation} \label{eq:widehat_in_proof}
\widehat{x}=\sum_{U\subseteq\uord{n}}(-1)^{n-|U|}\inc^\uord{n}_U\pr^\uord{n}_U(x)
\end{equation}
and let
\begin{equation} \label{eq:hopf_claim}
\overline{x}=\widehat{x}-\sum_{b\in
\natpos^\uord{n}}r_{b,x}\mu_\uord{n}^b 
\end{equation}
where the elements $r_{b,x}$ are given in Equation \ref{eq:def_rbx}.

Define $\obasis_2$ to be the set $\{\overline{x}|x \in \basis_2, |x|\leq2pm\}$.

To prove part~\ref{thm_part:algIso} of Theorem \ref{thm:smash_over_torus} for $m+1$, we will first show that 
$\torusIso$ is well defined in degree $2pm$ on decomposable elements. By the degree induction hypothesis, it suffices to prove that  
if $y\in \obasis_{1}\cup \obasis_{2}$ is a non-zero element of degree $2m$, then $y^p=0$. 

By Proposition~\ref{prop:TorusSFoldHopf}, $\Loday(T^{-})$ is an $\uord{n}$-fold Hopf
algebra. 
By the degree induction hypothesis and the definition of $\torusIso$,  the only monomials in $\Loday(T^\uord{n})$ 
in degrees less than $2m$ that are non-zero when raised to the
power of $p$ are the monomials in the subring
$P(\mu_\uord{n})\subseteq \Loday(T^{\uord{n}})$.
Thus, by Frobenius and part~\ref{thm_part:wedge_projection} of Theorem \ref{thm:smash_over_torus}
 we have, 
\begin{align*}
 \comult_{\uord{n}}^i(y^p) &=\comult_\uord{n}^i(y)^p =( 1\otimes y+ y\otimes 1
+ \sum y'\otimes y'')^p
  = 1\otimes y^p + y^p\otimes 1,
\end{align*}
in all $\Loday(T^{\uord{n}\setminus i})$-Hopf algebra structures. 
Hence, $y^p$ is primitive in the $\Loday(T^{\uord{n}\setminus i})$-Hopf algebra
structure for every $i\in \uord{n}$. 

Let $y^p$ be represented by $y_s\in E_{s,*}^\infty(T^\uord{n})$ modulo lower filtration. Then, since $y^p$ is
primitive in the $\Loday(T^{\uord{n}\setminus i})$-Hopf algebra
structure, $y_s$ must be primitive in 
the  $\Loday(T^{\uord{n}\setminus i})$-Hopf algebra $E^\infty(T^\uord{n})$ for every
$i\in \uord{n}$. If $y_s$ is not primitive, $\comult_{\uord{n}}^i(y^p)$ would not be equal to
$y_s\otimes 1 + 1\otimes y_s$ in filtration $s$.

The 
$\Loday(T^{\uord{n}\setminus i})$-primitive elements in $\Loday(T^\uord{n}_{n-1})\otimes
B_n$ are by the graded version of Proposition
3.12 in \cite{MilnorMoore65} linear combinations of primitive
elements in $\Loday(T^\uord{n}_{n-i})$ and $B_n$. By the graded version of Proposition
3.12 in \cite{MilnorMoore65} the module of
$\Loday(T^{\uord{n}\setminus i})$-primitive elements in
$B_n$ is
$\Loday(T^{\uord{n}\setminus i})\{x_{j}\}$, where $x_{j}$ runs over the monic words
in $B_n$. The intersection $\bigcap_{i\in
\uord{n}}\Loday(T^{\uord{n}\setminus i})\{x_j\}$ is equal to $\bb{F}_p\{x_j\}$ since
$\bigcap_{i\in
\uord{n}}\Loday(T^{\uord{n}\setminus i})=\bb{F}_p$. 
Thus, the module of elements in $B_n\subseteq E^2(T^\uord{n})$ that are
primitive in  the $\Loday(T^{\uord{n}\setminus i})$-Hopf algebra structure for every
$i\in \uord{n}$
 is $\bb{F}_p\{x_j\}\subseteq B_n$, 
which is isomorphic to the module of $\bb{F}_p$-primitive elements in $B_n$,
under the projection map $E^2(T^\uord{n})\rightarrow B_n$

The degree of $y^p$ is at least four and is zero modulo $2p$,
so by Lemma~\ref{lemma:B_n_primitive}
there are no $\bb{F}_p$-primitive elements in $B_n$ in the degree of $y^p$.

Hence, $y^p$ must be equal to an $\uord{n}$-fold primitive element in
$\Loday(T^\uord{n}_{n-1})$. 
By Corollary~\ref{corr:dimOfnFoldPrim}, the degree of 
 $y^p$ is not equal to the degree of any $\uord{n}$-fold primitive element in
$\Loday(T^\uord{n}_{n-1})$ when $n\leq p$.
Thus $y^p=0$.

We will now extend $\torusIso$ to degree $2pm$ on all elements. First we show that $\Loday(T^{-})$ is a $P_{\bb{F}_p}(\mu_{-})$ split  $S$-fold Hopf algebra in degree $q$. 
 By the torus induction hypothesis and part~\ref{thm_part:wedge_projection} in Theorem~\ref{thm:smash_over_torus},
 there is a splitting of
$\uord{n}$-fold Hopf algebras $P_{\bb{F}_p}(\mu_{-})\rightarrow
\Loday(T^{-}_{n-1})\rightarrow P_{\bb{F}_p}(\mu_{-})$. Since $\Loday(T^{-}_{n-1})$ is
the restriction, see Definition~\ref{def:restrictionHopfAlg}, of $\Loday(T^{-})$  to
the full subcategory of
$V(\uord{n})$ not containing $\uord{n}$, part
\ref{def_ass:splitSFoldHopfAlg} in Definition~\ref{def:SplitSHopf}
is thus satisfied for the $\uord{n}$-fold Hopf algebra $\Loday(T^{-})$ for $q = 2pm$. By the degree induction hypothesis and 
part~\ref{thm_part:wedge_projection} in Theorem~\ref{thm:smash_over_torus}, this map can be extended to a map 
$$\pr:\Loday(T^{\uord{n}})\rightarrow P_{\bb{F}_p}(\mu_{-})$$
in degrees less than $2pm$ such that assumption \ref{def_ass:mapHopfAlg}  in Definition~\ref{def:SplitSHopf} is satisfied. 
We can further extend this map to an $\bb{F}_p$-algebra map 
$$\pr':\Loday(T^{\uord{n}})\rightarrow P_{\bb{F}_p}(\mu_{-})$$
in degrees less than or equal to $2pm$, by mapping $\widehat{x}$, defined in Equation \ref{eq:widehat_in_proof}, for $x\in\basis_2$ in degree $2pm$ to zero.
This is well defined since $\widehat{x}$ is indecomposable. The map $\pr'$ satisfies part \ref{def_ass:extendSplitAlg} in Definition~\ref{def:SplitSHopf}.

From Equation \ref{eq:widehat_in_proof}, $\widehat{x}\in\bigcap_{i\in
\uord{n}}\ker(\counit_{\uord{n}}^i:\Loday(T^\uord{n})\rightarrow
\Loday(T^{\uord{n}\setminus i}))$, and by the definition of $\pr'$ we have $\widehat{x}\in \ker(\pr')$

For $x\in\basis_{2}$, the degree of $x$ is equal to the degree of an 
admissible word in
$B_n\subsetneq E^2(T^\uord{n})$ of even degree. 
By Corollary~\ref{corr:degOfComult}, $\widehat{x}$ is thus not
in the same degree as any of the elements $\mu_1^{p^{j_1}}\mu_2^{p^{j_2}}\cdots 
\mu_n^{p^{j_n}}$ or $(\mu_1^{p^{j_1}}\mu_2^{p^{j_2}}\cdots
\mu_n^{p^{j_n}})\mu_s^{p^{j_{n+1}}}$, where $1\leq s \leq n$ and  $j_i\in
\bb{N}$ for all $1\leq i \leq n+1$.  

Hence,
Proposition~\ref{prop:PossIteratedHopfStruc}
gives us that 
\begin{equation} \label{eq:def_rbx}
(\pr'_{P(\mu_\uord{n})}\otimes_{P(\mu_{\uord{n}\setminus
i})}\pr'_{P(\mu_\uord{n})})\circ
\psi^i_\uord{n}(\widehat{x}) = 
\sum_{b \in \natpos^\uord{n}}
r_{b,x}\mu_{\uord{n}\setminus i}^b\redcomult^i(\mu_i^{b_i})
\end{equation}
for some $r_{b,x}$ in $\bb{F}_p$.

To finish the construction of $\torusIso$ in degrees less than $2p(m+1)$ we must show that it is well defined on elements in $\obasis_1$. 
Note that if $y\in\obasis_1$ is of odd degree then $y^2=0$, since the ring is graded
commutative. This shows that $\torusIso$ is an $\bb{F}_p$ isomorphism in degree less than $2p(m+1)$. 

To show that it is a natural isomorphism of functors from $\catSpanSet{\uord{k}}$ to $\bb{F}_p$-algebras we must show that 
$\pr^{\uord{n}}_V(y)=0$ for all $y\in \obasis_1\cup\obasis_2$. In Equation \ref{eq:hopf_claim} we only sum over sequences of positive
integers so $\pr^{\uord{n}}_V(\overline{x})=\pr^{\uord{n}}_V(\widehat{x})=0$
for all $V\subsetneq \uord{n}$ proving naturality for for the set $\obasis_{2}$. 
Naturality for the elements in $\obasis_1$ follows from Lemma \ref{lemma:pr_sigma_null}.

Now we proof hypothesis \ref{induction2} of the degree induction hypothesis for $m+1$. 
To show that $\pr_{P(\mu_\uord{n})}$ is an Hopf algebra morphism for all $i \in \uord{n}$ we must show that
\begin{equation} \label{eq:proof_pr_zero}
(\pr_{P(\mu_\uord{n})}\otimes_{P(\mu_{\uord{n}\setminus
i})}\pr_{P(\mu_\uord{n})})\circ
\psi^i_\uord{n}(x)=0 
\end{equation}
 for $x$ in $\obasis_1$ or $\obasis_2$. 
By Equation \ref{eq:hopf_claim} and \ref{eq:def_rbx} this holds for $x$ in $\obasis_2$.

Each element in $\obasis_1$ is of the form $\sOper(z)$ for some admissible word $z$ in $B_{\uord{n-1}}$.
By Proposition~\ref{prop:sOperDeriv}, $\sOper:\pi_*(
\loday_{T^\uord{n-1}}\HF)\rightarrow
\pi_*(\loday_{T^\uord{n}}\HF)$ is a derivation. From the commutative diagram \ref{eq:s_derivation}, 
if $\comult_\uord{n-1}^i(z)= 1\otimes z + z\otimes 1 + \sum z_i'\otimes z_i''$
for
$i\in \uord{n-1}$,
then 
$$\comult_\uord{n}^i(\sOper(z)) =  1\otimes 
\sOper(z) + \sOper(z)\otimes 1 + \sum
\sOper(z_i')\otimes z_i'' \pm z_i'\otimes \sOper(z_i'').$$

By part \ref{induction2} of the degree induction hypothesis and since $\sOper$ is a derivation,
$\sOper(z_i')$ and $\sOper(z_i'')$ are in the kernel of $\pr_{P(\mu_\uord{n})}$, and thus Equation \ref{eq:proof_pr_zero} is satisfied for $i<n$. 
Proposition~\ref{prop:sOperPrim} shows that $\sOper(z)$ is
primitive as an element in the $\Loday(T^{\uord{n-1}})$-Hopf algebra
$\Loday(T^\uord{n})$ proving it for $i=n$. 

\emph{Proof of part~\ref{thm_part:sigmaElements}:}

This follows from the definition of $\obasis_1$ and $\torusIso$.

\emph{Proof of part~\ref{thm_part:sphere_projection}:}
It suffices to prove it for the elements in $\obasis_1$ and $\obasis_2$ since they 
generate $\Loday(T^\uord{n})$ as an $\Loday(T^{\uord{n}}_{n-1})$-algebra. 

Let $X$ be the ideal generated by the non-units in $\Loday(T^\uord{n}_{n-1})$.
By Lemma \ref{lemma:pr_is_null}, $x$ and $\widehat{x}$ in Equation \ref{eq:widehat_in_proof} are equal modulo $X$. 
In Equation \ref{eq:hopf_claim} $\widehat{x}$ and $\overline{x}$ are equal modulo $X$. 
From this it follows that, modulo the ideal $X$, $\torusIso$ maps $\overline{x}$ to the labeled version of $x$ in $\B_{\uord{n}}$, 
proving part~\ref{thm_part:sphere_projection} for elements in $\obasis_2$.

By Proposition~\ref{prop:sigma_in_SS} and the commutative diagram 
$$\xymatrix{S^1_+\wedge \loday_{T^\uord{n-1}}\HF
\ar[r]^-{\sCircleInc} \ar[d]^-{S^1_+\wedge
\sphereproj^{\uord{n-1}}} &
\loday_{T^\uord{n}}\HF \ar[d]^-{\sphereproj^{\uord{n}}} \\
S^1_+\wedge \loday_{s^{n-1}}\HF \ar[r]^-{\sSmashInc} &
\loday_{S^n}\HF,}$$
the element $\sphereproj^\uord{n}(\sOper(z_{\uord{n-1}}))$ is  equal to
$\sBar\sphereproj^\uord{n-1}(z_{\uord{n-1}})= \sBar z$, where $z$ is the unlabeled version of $z_{\uord{n-1}}$. This proves
part~\ref{thm_part:sphere_projection} for elements in $\obasis_1$.

\end{proof}

\begin{lemma} \label{lemma:pr_is_null}
Given $x\in \Loday(T^\uord{n})$, let
\begin{equation*}
\widehat{x}=\sum_{U\subseteq\uord{n}}(-1)^{n-|U|}\inc^\uord{n}_U\pr^\uord{n}_U(x). 
\end{equation*}
For every $V\subsetneq \uord{n}$ 
$$\pr^\uord{n}_V(\widehat{x})=0$$
and $\widehat{x}=x$ in $\Loday(T^\uord{n})$ modulo the ideal
generated by the
non-units in $\Loday(T^\uord{n}_{n-1})$
\end{lemma}
\begin{proof}
For every $U\subsetneq \uord{n}$ we have
$\inc^\uord{n}_U\pr_U^\uord{n}(x)\in \Loday(T^\uord{n}_{n-1})$, so
$\widehat{x}$ is equal to $x$ in $\Loday(T^\uord{n})$ modulo the ideal
generated by the
non-units in $\Loday(T^\uord{n}_{n-1})$. 
The diagram
$$\xymatrix{
T^\uord{n} \ar[r]^-{\pr^\uord{n}_U} \ar[d]^-{\pr^\uord{n}_{U\cap V}} &
T^U \ar[r]^-{\inc^\uord{n}_U} &
T^\uord{n} \ar[d]^-{\pr^\uord{n}_{V}} \\
T^{U\cap V} \ar[rr]^-{\inc^V_{U\cap V}} & &
T^V}$$ commutes. Hence, if $V\subsetneq\uord{n}$, then 
\begin{align*}
 \pr^\uord{n}_V(\widehat{x}) &=\sum_{U\subseteq
\uord{n}}(-1)^{n-|U|}\pr^\uord{n}_V\inc^\uord{n}_U\pr^\uord{n}_U(x)
=
\sum_{S\subseteq V}\sum_{W\subseteq
\uord{n}\setminus V}(-1)^{n-|S|-|W|}\inc_S^V\pr^\uord{n}_{S} (x)
= 0,
\end{align*}
since $$\sum_{W\subseteq\uord{n}\setminus V}(-1)^{n-|S|-|W|}
= \sum_{i=0}^{n-|V|}(-1)^{n-|S|-i}\binom{n-|V|}{i} =0.$$
\end{proof}

\begin{lemma} \label{lemma:pr_sigma_null}
 Let  $x$ be an element in $\bigcap_{i\in
\uord{n-1}}\ker(\counit_{\uord{n-1}}^i:\Loday(T^\uord{n-1})\rightarrow
\Loday(T^{\uord{n-1}\setminus i}))$. Then 
$\sigma(x)$ is an element in $\bigcap_{i\in
\uord{n}}\ker(\counit_{\uord{n}}^i:\Loday(T^\uord{n})\rightarrow
\Loday(T^{\uord{n}\setminus i}))$
\end{lemma}
\begin{proof}
 Observe that  the diagrams
\begin{align*}\xymatrix{S^1_+\wedge \loday_{T^\uord{n-1}}\HF
\ar[r]^-{\sCircleInc} \ar[d]^-{S^1_+\wedge \pr^{\uord{n}-1}_{\uord{n-1}\setminus i}} &
\loday_{T^\uord{n}}\HF \ar[d]^-{\pr^\uord{n}_{\uord{n}\setminus  i}} \\
S^1_+\wedge \loday_{T^{\uord{n-1}\setminus i}}\HF \ar[r]^-{\sCircleInc} &
\loday_{T^{\uord{n}\setminus i}}\HF} & 
&&
\xymatrix{S^1_+\wedge \loday_{T^\uord{n-1}}\HF 
\ar[r]^-{\sCircleInc} \ar[d]^-{\pr_+\wedge \id} &
\loday_{T^\uord{n}}\HF \ar[d]^-{\pr^\uord{n}_{\uord{n-1}}} \\
S^0\wedge \loday_{T^\uord{n-1}}\HF \ar[r]^-{\cong}  &
\loday_{T^\uord{n-1}}\HF,}
\end{align*}
commute for all $i\in \uord{n-1}$.
From the left diagram, we conclude that
$\pr^\uord{n}_{\uord{n}\setminus i}(\sOper(x))$ is zero when $i\not=n$.
From the right diagram we conclude that $\pr^\uord{n}_{\uord{n-1}}(\sOper(z))$
is zero, since $H_1(S^0)=0$. 
\end{proof}

\section{Periodic Elements} \label{sec:perElm}

The connective  $m$-th Morava $K$-theory $k(m)$ is a ring spectrum with
coefficient ring
$k(m)_*=P_{\bb{F}_p}(v_m)$ where $|v_m|=2p^m-2$. The unit map of the ring
spectrum $\loday_{T^\uord{n}}\HF$ induces a
homomorphism $P_{\bb{F}_p}(v_{m})\rightarrow k(m)_*(\loday_{T^\uord{n}}\HF)$ and
we denote the image of $v_m$ with $v_m$.
In this section we show that the \emph{the Rognes element} $t_1\mu_1^{p^{n-1}}+t_2\mu_2^{p^{n-1}}+\ldots
+t_{n}\mu_{n}^{p^{n-1}}$, where $\mu_i\in L_2(T^\uord{n})$ is the image of the generator in $\Loday(S^1)$ under the inclusion of the $i$-th circle, in the homotopy fixed points spectral sequence
calculating $k(n-1)_*(F(E_2T^{n}_+,\loday_{T^\uord{n}}\HF)^{T^{n}})$ is not
hit by any differential, and that this implies that $v_{n-1}\in
k(n-1)_*((\loday_{T^\uord{n}}\HF)^{hT^{n}})$ is non-zero. 

Given a prime $p$, let $A_*$ be the dual Steenrod algebra, see \cite{Milnor58}
for details. When $p$ is odd $A_* = P(\oxi_1,\oxi_2,\ldots)\otimes
E(\otau_0,\otau_1,\ldots)$ where $|\oxi_i|=2p^i-2$ and $|\otau_i|=2p^i-1$, 
and when $p=2$,  $A_* = P(\oxi_1,\oxi_2,\ldots)$ where $|\oxi_i|=2^i-1$.

See \cite{JohnsonWilson75} for the following details about Morava $K$-theory. 
We have $H_*(k(n))
= \morava{n}_*$, where $\morava{n}_*\cong P(\oxi_1,\oxi_2,\ldots)\otimes
E(\otau_0,\ldots,\otau_{n-1},\otau_{n+1},\ldots )$ is isomorphic to the dual Steenrod algebra $A_*$,
without the generator $\otau_n$. 
Multiplication by
$v_n$ yields a cofiber sequence 
$$\Sigma^{2p^n-2}k(n)\rightarrow k(n)\rightarrow \HF$$
which in homology decomposes into short exact sequences 
$$0\rightarrow \morava{n}_* \rightarrow A_* \rightarrow
\Sigma^{2p^n-1}\morava{n}_*\rightarrow 0.$$
Since $\loday_{T^\uord{n}}\HF$ is an $\HF$-module spectrum, we have 
$k(m)_*(\loday_{T^\uord{n}}\HF)\cong k(m)_*(\HF)\otimes \Loday(T^\uord{n})$ as graded $k(m)_
*$-algebras.

Let $X$ be a $T^n$-spectrum. The homotopy fixed points spectrum of $X$ is defined as the mapping spectrum
$$X^{hT^n} = F(ET^n_+,X)^{T^n},$$ of $T^n$-equivariant based maps from $ET^n_+$ to $X$. 
Here $ET^n_+$ should be interpreted as the unreduced suspension spectrum of $ET^n$, the free contractible $T^n$-space. 

Now we want to construct the first two columns of the homotopy fixed points spectral sequence 
for the group $T^n$.
We use the setup in \cite{BrunerRognes05}.
Let the unit sphere $S(\bb{C}^\infty)$ be our model for $ES^1$ with the
$S^1$-action given by the coordinatewise action. The space $S(\bb{C}^\infty)$
is equipped with a free $S^1$-CW structure with one free $S^1$-cell in each
even degree. We use the product $S(\bb{C}^\infty)^n$ as a model for $ET^n$ with the product
$T^n$-CW structure. 

We now get a $T^n$-equivariant cofiber sequences 
$$E_{0}T^n\rightarrow E_{2}T^n\rightarrow T^n_+\wedge (\vee
S^{2})$$
where the wedge sum runs over all $2$-cells in $ET^n$. Here $T^n$ acts
trivially on the space $(\vee S^{2})$.

\begin{prop} \label{prop:homFixSS}
Let $X$ be a bounded below $T^n$-spectrum with finite homotopy groups in each
degree. Let $M$ be any homology theory. 
There is a strongly
convergent spectral sequence 
\begin{align*}
E^2_{s,t}\cong \mathbb{Z}\{1,t_1,\ldots, t_n\}\otimes M_t(X) &\Rightarrow
M_{s+t}(F(E_2T^n_+,X)^{T^n}).
\end{align*}

\end{prop}
\begin{proof}
 Proof is left to the reader. There are no convergence issue, since the spectral sequence is concentrated in two columns. 
\end{proof}

This spectral sequence is just a long exact sequence in disguise, but since the calculations are relevant for calculations in the homotopy fixed point spectral sequence, 
we have chosen to use the language of spectral sequences. 

When $X=\loday_{T^\uord{n}}\HF$ we will write $E^*(M,n)$ for the spectral sequence

$$E^2(M,n)=M_*(\loday_{T^\uord{n}}\HF)\{1,t_1,\ldots t_n\}\Rightarrow
M_*(F(E_2T^n_+,\loday_{T^\uord{n}}\HF)^{T^n})$$
when $M$ is $\HF$ or $k(m)$.

\begin{prop} \label{prop:d^2_is_enough}
Assume $x$ in $E^2_{-2,2p^{n-1}}(k(n-1),n)$ survives to $E^3_{-2,2p^{n-1}}(k(n-1),n)$.
If $d^2(\otau_{n-1})=h(x)$ in $E^2(\HF,n)$, where $h$ is the Hurewicz homomorphism, then $[x]=uv_{n-1}$, where $[x]\in k(n-1)_*(F(E_2T^n_+,\loday_{T^\uord{n}}\HF)^{T^n})$ is the class of $x$, for some unit $u$.
\end{prop}
\begin{proof}
Smashing the cofiber sequences
$$\xymatrix{\Sigma^qk(n-1)\ar[r]^-{v_{n-1}} &
k(n-1)\ar[r] & H\bb{F}_p \ar[r] &\Sigma^{q+1}k(n-1)}
$$
and
$$
\xymatrix{X_2\ar[r] &F(E_2T^n_+,\loday_{T^{\uord{n}}}H\bb{F}_p)^{T^n} \ar[r] &
X_0 \ar[r]^-{\partial} & \Sigma X_2,}$$
where $X_i = F(E_iT^n_+/E_{i-2}T^n_+,\loday_{T^{\uord{n}}}H\bb{F}_p)^{T^n}$, and taking homotopy groups, 
one obtains a diagram with exact rows and columns, where the connecting homomorphism $\partial$
induces the $d^2$ differential of the two-column spectral sequences. Since $x$ survives to $E^3_{-2,2p^{n-1}}(k(n-1),n)$, 
$[x]$ is non-zero in $k(n)_{2p^{n-1}-2}(F(E_2T^{n}_+,\loday_{T^\uord{n}}\HF)^{T^n})$.
Bu a diagram chase it is seen that
$[x]$ generates the kernel of
$$k(n)_{2p^{n-1}-2}(F(E_2T^{n}_+,\loday_{T^\uord{n}}\HF)^{T^n})\rightarrow
H_{2p^{n-1}-2}(F(E_2T^{n}_+,\loday_{T^\uord{n}}\HF)^{T^{n}}),$$ 
and that multiplication by $v_{n-1}$ is zero on $k(n-1)_*X_i$, from which the proposition follows.

\end{proof}

\begin{defn}
Let $M$ be a homology theory. 
The homomorphism 
$$(\sCircleInc_{T^\uord{n}})_*:M_*(T^\uord{n}_+\wedge
\loday_{T^\uord{n}}\HF)\rightarrow M_*(
\loday_{T^\uord{n}\times T^\uord{n}}\HF)$$ in Definition
\ref{def:circleInLodayMap},the fact that $T^n_+$ splits completely
 and the multiplication map in the group
$T^\uord{n}$, induces a homomorphism
$$H_*(T^\uord{n})\otimes M_*(\loday_{T^\uord{n}}\HF)\rightarrow
M_*(\loday_{T^\uord{n}}\HF).$$
Given $j\in \uord{n}$ and $x\in M_*(\loday_{T^\uord{n}}\HF)$ we write 
 $\sOper_j(x)$ for the image of $[S^1_j]\otimes x$ under this map, where
$[S^1_j]\in H_*(T^\uord{n})$ is the image of a fundamental class $[S^1]\in
H_*(S^1)$ under the
inclusion of the $j$-th circle.
\end{defn}

Here, $H_*$ is homology with $\bb{F}_p$ coefficients. 
Since $T^\uord{n}$ is a pointed space $\loday_{T^\uord{n}}\HF$ is an $\HF$-algebra and thus a generalized Eilenberg
Mac~Lane spectrum.
Let $p\geq 5$ and $1\leq n \leq p$ or $p=3$ and $1\leq n \leq 2$, and define $\polIdeal$ to be the kernel of the projection
homomorphism $$H_*(\loday_{T^\uord{n}}\HF)\cong A_*\otimes
\Loday(T^\uord{n})\cong A_*\otimes \bigotimes_{U\subseteq
\uord{n}}B_U\rightarrow \bigotimes_{i\in \uord{n}}B_{\{i\}},$$
where the isomorphism $\alpha:\Loday(T^\uord{n})\cong \bigotimes_{U\subseteq
\uord{n}}B_U$ is the one in Theorem \ref{thm:smash_over_torus}.

\begin{prop} \label{prop:sOperNotPol} Let $p\geq 5$ and $1\leq n \leq p$ or $p=3$ and $1\leq n \leq 2$.
 If $x$ is in the subring $P(\oxi_1,\oxi_2,\ldots )\otimes
\Loday(T^\uord{n})\subseteq A_*\otimes \Loday(T^\uord{n})\cong
H_*(\loday_{T^\uord{n}}\HF)$, then for every $j\in \uord{n}$ the element
$\sOper_j(x)$ is in $\polIdeal$.
\end{prop}
\begin{proof}
Since $\sOper_j$ is a derivation by Proposition \ref{prop:sOperDeriv}, it
suffices to check the claim for the set of $\bb{F}_p$-algebra
generators in $P(\oxi_1,\oxi_2,\ldots )\otimes
\Loday(T^\uord{n})\cong P(\oxi_1,\oxi_2,\ldots )\otimes \bigotimes_{U\subseteq \uord{n}}
B_U $ consisting of $\oxi_i$ for $i\geq 1$ together with all
$U$-labeled admissible words, where $U\subseteq \uord{n}$.

By Proposition 4.9 in \cite{AngeltveitRognes05}, the element $\sOper(\oxi_i)$ for
$i\geq1$ is represented by $\sBok \oxi_i$ in filtration $1$ in the  B\"okstedt
spectral sequence calculating $H_*(\loday_{S^1}\HF)$. By the calculation of B\"okstedt in \cite{Bokstedt86}, 
see Theorem 1 in \cite{Hunter96} for a published account,
the element $\sBok\oxi_i$ is a boundary in the B\"okstedt spectral sequence, and hence $\sOper(\oxi_i)\in A_*$. It must thus be equal
to zero  since it is the image of $[S^1]\otimes \oxi_i$, and $[S^1]$ is mapped
to zero on the left hand side in the commutative diagram
$$\xymatrix{
H_*(S^1_+)\otimes H_*(\HF) \ar[r]^-{\sCircleInc_*} \ar[d]^{\pr\otimes \id}  &
H_*(\loday_{S^1}\HF) \ar[d]^{\pr\otimes \id} \\
H_*(S^0)\otimes H_*(\HF) \ar[r]^-{\sCircleInc_*} &
H_*(\loday_{\point}\HF).}$$
Hence $\sOper_j(\oxi_i)=0 $ for all $i \geq 1$
and $1\leq j \leq n$.

We prove the proposition by induction on the degree $m$ of an element $x$ in
$\Loday(T^\uord{n})\cong \bigotimes_{U\subset \uord{n}}B_U$.  
When $m=0$, there is nothing to check since $\sOper_j$ is trivial on units. 

Assume the proposition holds for all elements in degrees less than $m$. 
If $m$ is even, the proposition holds because $\sOper_j(x)$  is then  of odd
degree, and $\bigotimes_{i\in \uord{n}}B_{\{i\}}$ is concentrated in even
degrees.
Assume $m$ is odd, and that $x$ is a $U$-labeled admissible word of degree $m$
for some $U\subseteq \uord{n}$.
By Proposition \ref{lemma:AdmissWord} part \ref{lemma_part:oddAdmWord}, $x$ is thus equal to $\sBar_k y$, where
$k$ is the largest element in $U$ and $y$ is a $U\setminus k$-labeled
admissible word of even degree. By part \ref{thm_part:sigmaElements}
of Theorem
\ref{thm:smash_over_torus}, $x$ is  equal to
$\sOper_k(y)$ where we think of $y$ as being an element in $B_{U\setminus
k}\subseteq \Loday(T^{U\setminus k})\subseteq \Loday(T^{\uord{k-1}})$.

If $j>k$, the element  $\sOper_j(x)=\sOper_j(\sOper_k(y))$ is in
$\polIdeal$ by part
\ref{thm_part:sigmaElements} of Theorem
\ref{thm:smash_over_torus}.

The element $\sOper_j(\sOper_k(y))$ is equal to the image of $[S^1_j]\cdot
[S^1_k]\otimes y$, where $[S^1_j]\cdot
[S^1_k]$ is the product in $H_*(T^\uord{n})$. 
When $k=j$,  $\sOper_j(\sOper_k(y))$ is thus zero since $[S_j^1]^2=0$. 

When $j<k$, we have $\sOper_j(\sOper_k(y)) = \pm\sOper_k(\sOper_j(y))$ since the
ring $H_*(T^\uord{n})$ is graded commutative. 
Now, $\sOper_j(y)$ is in $\Loday(T^\uord{k-1})$, so by part
\ref{thm_part:sigmaElements} in Theorem
\ref{thm:smash_over_torus}, the element $\sOper_k(\sOper_j(y))$ is in
$\polIdeal$. 
Hence, $\sOper_j(\sOper_k(y))$ is in $\polIdeal$.
\end{proof}

\begin{prop} \label{prop:no_diff_hit_mu}
The differential in $E^2(\HF,n)$ is given by
$$d^2(x) = t_1\sOper_1(x)+\ldots + t_{n} \sOper_n(x),$$ for $x\in
E^2_{0,*}(\HF,n)$. 

Thus, if  $p\geq 5$ and $1\leq n \leq p$ or $p=3$ and $1\leq n \leq 2$, and $x$ is in the subring $P(\oxi_1,\oxi_2,\ldots )\otimes
\Loday(T^\uord{n})\subseteq H_*(\loday_{T^\uord{n}}\HF)\cong E^2_{0,*}(\HF,n)$,
then
  $d^2(x)$ is in $\polIdeal\{t_1,\ldots, t_n\}$.
\end{prop}
\begin{proof}
For each $i \in \uord{n}$, inclusion of fixed points induces the projection homomorphism
from $E^2(\HF,n)$ to
$$E^2\langle i \rangle=H_*(\loday_{T^\uord{n}}\HF)\{1,t_i\}
\Rightarrow H_{*}(F(E_2S^1_+,\loday_{T^\uord{n}}\HF)^{S^1})$$
where $S^1$ acts on the $i$-th circle in $T^n$.
Now $E^2\langle i \rangle$ maps injectively to the Tate spectral sequence, so by Lemma 1.4.2 in 
\cite{Hesselholt96}, the $d^2$-differential in $E^2\langle i \rangle$ is induced by the operator
$\sOper_i$.

 Since $E^2_{2,*}(\HF,n)\cong H_*(\loday_{T^\uord{n}}\HF)\{t_1,t_2,\ldots,t_n\}$, and $E^2_{0,*}\langle i \rangle \cong E^2_{0,*}(\HF,n)$, 
 the formula for the differential in $E^2_{0,*}(\HF,n)$ is thus
$$d^2(x) = t_1\sOper_1(x)+\ldots + t_{n} \sOper_n(x),$$
and the second claim now follows by Proposition \ref{prop:sOperNotPol}.
\end{proof}

We will show that the Rognes element $t_1\mu_1^{p^{n-1}}+t_2\mu_2^{p^{n-1}}+\ldots
+t_{n}\mu_{n}^{p^{n-1}}$ in $E^2(\HF,n)$ is not hit by any differential in the homotopy fixed
points spectral sequence. The idea of the proof  is that by the previous
propositions only differentials on $\otau_i$ can hit an element in $P(\mu_1,\ldots
\mu_{n})\{t_1,\ldots t_{n}\}$. For dimension reasons this can only happen
when $i\leq n-2$, but since we have one  fewer variable $\otau_i$ than
$\mu_j$, these will not add up correctly. 

\begin{defn}
 Since the Rognes element $t_1\mu_1^{p^{n-1}}+t_2\mu_2^{p^{n-1}}+\ldots
+t_{n}\mu_{n}^{p^{n-1}}$ is a cycle, it represents an element in 
 $k(n-1)_*(F(E_2T^n_+,\loday_{T^\uord{n}}\HF)^{T^n})$ which we call \emph{the Rognes class}.
\end{defn}

\begin{prop} \label{prop:noRognesDiff}
Let $p\geq 5$ and $1\leq n \leq  p$ or $p=3$ and $1\leq n \leq 2$.
 The Rognes element $t_1\mu_1^{p^{n-1}}+t_2\mu_2^{p^{n-1}}+\ldots
+t_{n}\mu_{n}^{p^{n-1}}$ in  
 $$E^2(k(n-1),n)\cong k(n-1)_*(\loday_{T^\uord{n}}\HF)\{1,t_1,\ldots,
t_{n}\}\Rightarrow k(n-1)_*(F(E_2T^n_+,\loday_{T^\uord{n}}\HF)^{T^n})$$
 is not hit by any differential, and hence the Rognes class is a 
 non-zero element in $k(n-1)_*(F(E_2T^n_+,\loday_{T^\uord{n}}\HF)^{T^n})$.
\end{prop}

\begin{proof}
Since $k(n-1)_*(\loday_{T^\uord{n}}\HF)\subseteq H_*(\loday_{T^\uord{n}}\HF)$, the
differentials in $E^2(k(n-1),n)$ are determined by the differentials in
$E^2(\HF,n)$.
 By Theorem 5.2 in \cite{HesselholtMadsen97} and Proposition 4.9 in \cite{AngeltveitRognes05}, $\sOper_i(\otau_j) =
\mu_i^{p^j}$, so Proposition \ref{prop:no_diff_hit_mu} yields
$d^2(\otau_j)=\sum_{i=1}^{n}t_i\mu_i^{p^j}$.
 Assume $z$ is an element in $k(n-1)_*(\loday_{T^\uord{n}}\HF)$ with differential
$d^2(z) =
\sum_{i=1}^{n}t_i\mu_i^{p^{n-1}}$. It can be written as $$z=\otau_0z_0+\ldots+
\otau_{n-2}z_{n-2}+z',$$ where $z'$ is an element in $P(\oxi_1,\oxi_2,\ldots)\otimes
\Loday(T^\uord{n})$ and $z_i$, for $0\leq i \leq n-2$, are elements in $E(\otau_{i+1},\otau_{i+2},\ldots)\otimes P(\oxi_1,\oxi_2,\ldots)\otimes
\Loday(T^\uord{n})$. 
By Proposition
\ref{prop:no_diff_hit_mu}, $d^2(z')$ is in $\polIdeal\{t_1,\ldots, t_{n}\}$, and obviously the products $\otau_id^2(z_i)$ for 
$0\leq i \leq n-2$ is in $\polIdeal\{t_1,\ldots, t_{n}\}$
so we must have
\begin{equation} \label{eq:differential_hit_rognes}
 d^2(\otau_0)z_0+\ldots+
d^2(\otau_{n-2})z_{n-2} = \sum_{j=0}^{n-2}(t_1\mu_1^{p^j}+\ldots
 + t_{n}\mu_{n}^{p^j})z_j = \sum_{i=1}^{n}t_i\mu_i^{p^{n-1}} +y,
\end{equation}
for some  $y$
in $\polIdeal\{t_1,\ldots, t_{n}\}$.

Write  the elements $z_i$ in the monomial basis in $A_*\otimes
L(T^\uord{n})\cong A_*\otimes
\bigotimes_{U\subseteq\uord{n}}B_U$.
For equation \ref{eq:differential_hit_rognes} to hold, at least one of the
$z_i$-s must have a non-zero
coefficient in front of $\mu_1^{p^{n-1}-p^{i}}$. We let $k_1\geq 0$ be the
greatest integer $i$ such that this coefficient is non-zero. 

 Let $k_2<k_1$ be the
greatest integer where the coefficient in front of
$\mu_{1}^{p^{n-1}-p^{k_1}}\mu_{2}^{p^{k_1}-p^{k_2}}$ in $z_{k_2}$ is non-zero. 
Such an integer must exist, because the coefficient in front of 
$t_2\mu_{2}^{p^{k_1}}\mu_{1}^{p^{n-1}-p^{k_1}}$ on the left hand side in equation
\ref{eq:differential_hit_rognes} would otherwise be non-zero due to  the
contribution from $d^2(\otau_{k_1})z_{k_1}$. 

Continuing in this way we get that, since there are $n$ variables $t_i$,
there must be a sequence of integers $k_1>\ldots >k_{n}$ such that the
coefficient in front of the monomial
 $\mu_{1}^{p^{n-1}-p^{k_1}}\mu_{2}^{p^{k_1}-p^{k_2}}\cdots
\mu_{n}^{p^{k_n}-p^{k_{n}}}$ in $z_{k_{n}}$ is non-zero. But this is
impossible since there are only $n-1$ number of variables
$z_i$. 

We thus get a contradiction, so there is no element $z$ in
$k(n-1)_*(\loday_{T^\uord{n}}\HF)$ with differential $d^2(z) =
\sum_{i=1}^{n}t_i\mu_i^{p^{n-1}}$.
\end{proof}

\begin{thm} \label{thm:periodic_elm}
Let $p\geq 5$ and $1\leq n \leq  p$ or $p=3$ and $1\leq n \leq 2$. Then $v_{n-1}$ in $
k(n-1)_*((\loday_{T^\uord{n}}\HF)^{hT^{n}})$ 
is non-zero, and is detected by the Rognes class in the homotopy spectral sequence.
Equivalently, the homomorphism
 $$\xymatrix{ k(n-1)_*(\Sigma^{2p^{n-1}-2}F(E_2T^n_+,\loday_{T^\uord{n}}\HF)^{T^n})
\ar[r]^-{v_{n-1}} &
 k(n-1)_*(F(E_2T^n_+,\loday_{T^\uord{n}}\HF)^{T^n}) 
 }$$
 maps $1$ to a non-zero class.

\end{thm}

Whether higher powers of $v_{n-1}$ are non-zero is not known, and similar arguments to those in this article is probably not sufficient to resolve this question, as the number of potential differentials in the spectral sequences increases as the power increases.  
\begin{proof}
 The unit map $S^0\rightarrow (\loday_{T^\uord{n}}\HF)^{hT^{n}}$ and the
inclusion $E_2T^{n}\rightarrow ET^{n}$ induces the vertical
homomorphisms in the commutative diagram 
 $$\xymatrix{
  k(n-1)_*(\Sigma^{2p^{n-1}-2} S^0) \ar[r]^-{v_{n-1}}
\ar[d] &
 k(n-1)_*(S^0) \ar[d] \\ 
 k(n-1)_*(\Sigma^{2p^{n-1}-2}(\loday_{T^\uord{n}}\HF)^{hT^{n}}) \ar[r]^-{v_{n-1}}
\ar[d] &
 k(n-1)_*((\loday_{T^\uord{n}}\HF)^{hT^{n}}) \ar[d] \\
  k(n-1)_*(\Sigma^{2p^{n-1}-2}F(E_2T^n_+,\loday_{T^\uord{n}}\HF)^{T^n})
\ar[r]^-{v_{n-1}} &
 k(n-1)_*(F(E_2T^n_+,\loday_{T^\uord{n}}\HF)^{T^n}) 
 .}$$
 By Proposition \ref{prop:noRognesDiff} and \ref{prop:d^2_is_enough} the
homomorphism $v_{n-1}$ maps $1$ in the lower left hand corner to the non-zero
element represented by the cycle $t_1\mu_1^{p^{n-1}}+t_2\mu_2^{p^{n-1}}+\ldots
+t_{n}\mu_{n}^{p^{n-1}}$ in the lower right hand corner.
Hence, the image of $v_{n-1}$ must be non-zero in the middle group on the right hand
side of the diagram. 
\end{proof}

\section{Spectral Sequences} \label{sec:barSS}
In this section we review the well known bar spectral sequence which is the most important tool in our calculations. 

Let ${X_*}$ be a simplicial spectrum and define the simplicial abelian group
$\pi_t({X_*})$ to be $\pi_t(X_q)$ in degree $q$ with face and degeneracy
homomorphisms induced by the face and degeneracy maps in ${X_*}$. Write
$|{X_*}|$ for the realization of the simplicial spectrum ${X_*}$. 

The spectral sequence below is well known for spaces, and appears for $S$-modules in
Theorem X.2.9 in \cite{EKMM97}.

\begin{prop} \label{prop:SimpSpecSS}
Let ${X_*} $ be a simplicial spectrum, and assume
that $\sk_s({X_*})\rightarrow \sk_{s+1}({X_*})$ is a cofibration for all
$s\geq 0$. 
There is a strongly convergent spectral sequence
 $$E^2_{s,t}({X_*})= H_s(\pi_t({X_*}))\Rightarrow \pi_{s+t}({X_*}).$$

Let $R$ be a ring spectrum. If ${X_*}$ is also a simplicial $R$-algebra, then $E^*_{*,*}({X_*})$ is a
$\pi_*(R)$-algebra spectral sequence.

\end{prop}

Let $R$ be a commutative ring spectrum,
$M$ be a cofibrant right $R$-module, $N$ be a left $R$ module and let $B(M,R,N)$
be the bar
construction. In more details, $B(M,R,N)$ is the simplicial spectrum which in degree $q$
is equal to $M\wedge R^{\wedge q}\wedge N$, and where the face and
degeneracy maps are induced by the same formulas as in the algebra case using
the unit map and multiplication map. By Lemma 4.1.9 in
\cite{Shipley00} there is an equivalence $|B(M,R,N)|\simeq |M\wedge_RN|$. 

\begin{prop} \label{prop:barSS}
  Let $R$ be a bounded below ring spectrum, $M$ a right $R$-module and $N$ a left $R$-module.
  Then there is a strongly convergent spectral sequence 
  $$E^2_{s,t}= \tor_s^{\pi_*R}(\pi_*M,\pi_*N)_t\Rightarrow
\pi_{s+t}(M\wedge^L_RN).$$
\end{prop}
\begin{proof}
 For $S$-modules this is a corollary of Theorem X.2.9 in \cite{EKMM97}.
\end{proof}

\begin{remark} \label{rem:isCoalg}
 Let $X_*$ be a cofibrant simplicial $R$-module. If $\pi_*(X_*)$ is flat as a $\pi_*(R)$-module, this proposition yields an isomorphism
 $\pi_*(X_*\wedge^L_RX_*)\cong \pi_*(X_*)\otimes_{\pi_*(R)}\pi_*(X_*) $.

If
${X_*}$ is a simplicial
$R$-coalgebra, i.e., there is a coproduct map $\comult:{X_*}\rightarrow  {X_*}\wedge_{R}{X_*}$ with
a counit map ${X_*}\rightarrow R$ making the obvious diagrams commute up to
homotopy, and $\pi_*(X_*)$ is flat as a $\pi_*(R)$-module, then $\pi_*(X_*)$ is a $\pi_*(R)$-coalgebra
with coproduct induced by $\comult$ followed by the isomorphism $\pi_*(X_*\wedge^L_RX_*)\cong \pi_*(X_*)\otimes_{\pi_*(R)}\pi_*(X_*) $.
\end{remark}

\begin{prop}  \label{prop:barSSisCoalg}
Assume that ${X_*}$ is a cofibrant simplicial
$R$-coalgebra, and assume
that the map $\sk_s({X_*})\rightarrow \sk_{s+1}({X_*})$ is a cofibration for all
$s\geq 0$.
If each term $E^r({X_*})$ for $r\geq 1$ is flat over $\pi_*(R)$ then $E^*_{*,*}({X_*})$ is a 
$\pi_*(R)$-coalgebra spectral sequence.
If in addition, $\pi_*(X_*)$ is flat as a $\pi_*(R)$-module, then the 
spectral sequence converges to $\pi_*(X_*)$ as a $\pi_*(R)$-coalgebra.
\end{prop}
\begin{proof}
A similar statement is proven for $X=\loday_{S^1}R$ and mod $p$ homology in Theorem 4.5 in \cite{AngeltveitRognes05}. 
This proof also works for $T^n$ and $S^n$, since we can make the pinch maps simplicial, so that they descend to maps of chain complexes. 
\end{proof}

In particular, for  $B(R,\loday_{X}R,R)\simeq \loday_{S^1\wedge X}R$ we have
the following proposition.

\begin{prop} \label{prop:sigma_in_SS}
 Let $R$ be a commutative ring spectrum and let $X$ be a simplicial set.
The derivation
$$\sOperSm:\pi_*(\loday_{X}R)\rightarrow \pi_*(\loday_{S^1\wedge
X}R)$$
 defined in, \ref{def:circleInLodayMap}, takes $z\in \pi_n(\loday_{X}R)$ to the class of $[z]$ in 
$$E^2_{s,t}=\tor_s^{\pi_*(\loday_{X}R)}(\pi_*( R),\pi_*(R))_t\Rightarrow \pi_{s+t}(\loday_{S^1\wedge
X}R),$$
where $[z]$ in the reduced bar complex $B(\pi_*( R),\pi_*(\loday_{X}R),\pi_*(R))$, is represented by $z$ in $E^1_{1,n}\cong B_1(\pi_*( R),\pi_*(\loday_{X}R),\pi_*(R))_n \cong \pi_n(\loday_{X}R)$ .

\end{prop}

\begin{proof}
Using the minimal simplicial model for $S^1$ we get a simplicial spectrum 
$S^1_+\wedge \loday_{X}R$ which in simplicial degree $q$ is equal to 
$(S^1_q)_+\wedge \loday_{X}R\cong (\loday_{X}R)^{\vee q}$, the $q$-fold wedge of $\loday_{X}R$.
In the $E^2$-term of the spectral sequence in Proposition \ref{prop:SimpSpecSS} associated with this simplicial spectrum, 
the element $[S^1]\otimes z$ is represented by $1\oplus z$ in 
$E^1_{1,*}\cong \pi_*(\loday_{X}R\vee\loday_{X}R)\cong \pi_*(\loday_{X}R)\oplus \pi_*(\loday_{X}R)$, where the second summand corresponds to
the non-degenerate simplex in $S^1_1$. 

Similarly, there is a simplicial model for the spectrum $\loday_{S^1\wedge
X}R$, which in simplicial degree $q$ is equal to 
$\loday_{S^1_q\wedge X}R\cong \loday_{\bigvee_q X}R \cong (\loday_{X}R)^{\wedge_R q-1}$, the $(q-1)$-fold smash product over $R$.
The map $\sSmashInc:S^1_+\wedge \loday_{X}R\rightarrow \loday_{S^1\times X}R$, defined in Definition \ref{def:circleInc}, is given on these simplicial models in degree $q$ by the natural map
$$(\loday_{X}R)^{\vee q}\rightarrow (\loday_{X}R)^{\wedge q}\rightarrow (\loday_{X}R)^{\wedge_R q-1}$$
where the first map is induced by the inclusion into the various smash factors using the unit maps, 
and the second map is induced by the map 
$\loday_{X}R \rightarrow \loday_{\point}R$ on the factor indexed by the simplex which is the image of the $0$-th simplex under a $q$-th fold composition of degeneracy maps.
The element $\sOperSm(z)$ in the spectral sequence from Proposition \ref{prop:SimpSpecSS} associated with this simplicial spectrum, is thus 
represented by the element $z$ in $E^1_{1,*}\cong \pi_*(\loday_{X}R)$. 

Now we have to compare this last spectral sequence, with the spectral sequence coming from the bar complex $B(R,\loday_{X}R,R)$.
In simplicial degree $q$, $B(R,\loday_{X}R,R)$ is equal to $R\wedge \loday_{X}R^{\wedge q-1} \wedge R\cong \loday_{S^0\amalg (\coprod_qX)}R$.
The equivalence between $B(R,\loday_{X}R,R)$ and the model above is induced by the map $S^0\amalg \coprod_qX\rightarrow \bigvee_qX$ identifying $S^0$ and the basepoints in $X$ to the base point in $\bigvee_qX$. 
The element $\sOperSm(z)$ is thus represented by the class of $[z]$ in 
$$E^2_{s,t}=\tor^{\pi_*(\loday_{X}R)}(\pi_*( R),\pi_*(R))\Rightarrow \pi_{s+t}(\loday_{S^1\wedge
X}R),$$
where $[z]$ is in the reduced bar complex $B(\pi_*( R),\pi_*(\loday_{X}R),\pi_*(R))$.
\end{proof}

\begin{prop} \label{prop:shortestDiff}
 Let $R$ be a commutative ring and let $E^*$ be a first quadrant
connected $R$-Hopf algebra spectral sequence. The shortest non-zero
differentials in $E^*$ of lowest total degree, if there are any, are from an indecomposable element in $E^*$ to
a primitive
element in $E^*$. 
\end{prop}
\begin{proof}
See Proposition 4.8 in \cite{AngeltveitRognes05}
\end{proof}

The next proposition shows that in certain circumstances the coalgebra
structure of the abutment in a spectral sequence  is determined by the algebra structure of the dual spectral sequence. We
will use it to calculate the $\bb{F}_p$-Hopf algebra structure of
$\pi_*(\loday_{S^n}\HF)$.

 Let $R$ be a field with characteristic different than $2$, and let  $$\xymatrix{
0\ar[r] & A_{0} \ar[r]^i \ar[d]^= &
A_{1} \ar[d]^j \ar[r]^i &
A_{2} \ar[d]^j \ar[r] & \ldots \\ 
& E^1_{0} & E^1_{1} \ar[ul]_k &E^1_2 \ar[ul]_k
} $$ be an unrolled exact couple of connected cocommutative
$R$-coalgebras which are of finite type.
The unrolled exact couple gives rise to a
spectral sequence $E^*$ converging strongly to $\colim_s A_{s}$ by Theorem 6.1 in \cite{Boardman99}

\begin{prop} \label{prop:coextension_SS}
 Assume that in each degree
$t$ the map $A_{s,t}\rightarrow A_{s+1,t}$ eventually stabilizes, i.e., is an isomorphism for all $s\geq u$ for some $u$ depending on $t$.
Assume that $E^\infty$-term of the spectral
sequence is isomorphic, as an $R$-coalgebra, to a tensor product of exterior
algebras and divided power algebras.
Then there are no coproduct coextensions in the abutment. 
Hence, $\colim_s A_S\cong E^\infty$
as an $R$-coalgebra.
\end{prop}
\begin{proof}
The colimit  $\colim_s A_S$ of $R$-coalgebras is constructed in the underlying
category of $R$-modules.  
 Applying $D(-)=\hom_{R}(-,R)$ to the isomorphism $\colim_s A_S\cong E^\infty$, yields an isomorphism
 $D(\colim_s A_S)\cong D(E^\infty)$ where $D(E^\infty)$ is a free graded commutative algebra. Since $D(E^\infty)$ is the associated graded algebra of the filtered commutative algebra $
D(\colim_s A_S)$ this implies that $D(\colim_s A_S)$ is a free graded commutative algebra. 
Since the maps $A_s\rightarrow A_{s+1}$ eventually
stabilizes, $D(\lim_sA_s)\cong \colim_sA_s$, so  we can dualize again, and get
that there is an 
$R$-coalgebra isomorphism $\colim_sA_s\cong E^\infty$. 
\end{proof}

\section{Primitive Elements}
\label{sec:primElm}
In this section we prove several technical statements about the
degrees of certain
admissible words and primitive elements in multifold Hopf algebras. The first two lemmas can obviously be generalized to
all $n$, but we only need them for $n\leq p$, so we keep
their formulations as simple as possible.

\begin{lemma} \label{lemma:p-adic_generators}
 Let $n\leq 2p-2$, and let $x$ be an admissible word in $B_{n}$ of even degree.
 Let $l$ be the
number of occurrences of the letter $\sBar$ in the word $x$.
The sum of the coefficients in the $p$-adic expansion of the number
$\frac{|x|}2$ is equal to $n-l$.
\end{lemma}

\begin{proof}
The proof is by induction on $n$.
 It is true for $n=1$ since $l=0$ and $|\mu|=2$.
 Assume it is true for all $1\leq m\leq n-1$. An admissible word $x$ in
$B_{n}$
of even degree is, by part~\ref{lemma_part:oddAdmWord} in Lemma
\ref{lemma:AdmissWord},  either equal to $\varphi^ky$ or $\sBar^k\sBar z$ for
some $k\geq 0$, where $y$
and $z$ are admissible words in $B_{n-1}$ and $B_{n-2}$, respectively. 

First,
$\frac{|\varphi^ky|}2=p^k(1+p\frac{|y|}2)$, so if the sum of the
coefficients in the $p$-adic expansion of $\frac{|y|}2$ is $n-1-l$, where $l$
is the number of occurrences of $\sBar$ in $y$, the
sum of the coefficients in the $p$-adic expansion of $\frac{|\varphi^ky|}2$ is
$n-l$

Second, $\frac{|\sBar^k\sBar z|}2 =p^k(1+\frac{|z|}2)$, so if the sum of the
coefficients in the $p$-adic expansion of $\frac{|z|}2$ is $n-2-(l-1)=n-1-l$,
where $l-1$ is the number of occurrences of $\sBar$ in $z$, then  the
sum of the coefficients in the $p$-adic expansion of $\frac{|\sBar^k\sBar z|}2$
is $n-l$, unless there was carrying involved in the addition
$1+\frac{|z|}2$.  

There is only carrying involved if the degree of $z$ is equal to $-2$ modulo
$2p$, and by part~\ref{lemma_part:modpAdmWord} in Lemma~\ref{lemma:AdmissWord}
this implies that $z$ is equal to $(\sBar^0\sBar)^{p-2}\mu$, or
starts with $(\sBar^0\sBar)^{p-1}$ or $(\sBar^0\sBar)^{p-2}\varphi^0$. In
these cases $\sBar^0\sBar z$ has length at least $2p-1$, so there is no carrying
involved when $n\leq 2p-2$.
\end{proof}

\begin{lemma} \label{lemma:indecWord}
Let $Q(B_n)$ be the module of indecomposable
elements in $B_n$. 
 If $2\leq n\leq 2p$, then $Q(B_n)_{2pi-1}=0$ for all $i$ and
$\bigoplus_{i\geq 1}Q(B_n)_{2pi}$ is equal to the module generated by all
non-monic admissible words of length $n$.
\end{lemma}
\begin{proof}
 The module of indecomposable elements is generated by all admissible words of
length $n$. All non-monic words are in degree
$0$ modulo $2p$. All monic words
are primitive, so by~\ref{lemma:B_n_primitive} they are not in degree $-1$ or
$0$ modulo $2p$ when $2\leq n \leq 2p$.  
\end{proof}

\begin{lemma} \label{lemma:p-adic_product_mu}
The sum of the coefficients in the $p$-adic expansion of $\frac{|\mu_1^{p^{j_1}}\mu_2^{p^{j_2}}\ldots\mu_n^{p^{j_n}}|}2$, where
$j_i\geq 0$ and $|\mu_i|=2$ for $1\leq i \leq n$, is equal  to 
$n$ when $0<n<p$ and $n$ or $n-p+1$ when $p\leq n <2p$.
\end{lemma}
\begin{proof}
If less than $p$ of the numbers $j_i$ are equal, we get the case $n$, and if
at least $p$ of the numbers $j_i$ are equal, we get the case $n-p+1$. 
\end{proof}
\begin{corr} \label{corr:degOfComult}
Let $x$ be an admissible word in $B_n$ of even
degree. 

If $1\leq n\leq p$, then the degree of $x$ is not equal to the degree
of 
$\mu_1^{p^{j_1}}\mu_2^{p^{j_2}}\ldots\mu_n^{p^{j_n}}$, where $j_i\geq 0$ for
$1\leq i \leq n$. 

If $p\geq 5$,  $1\leq n\leq p$ and $1\leq s \leq n$, then the
degree of $x$ is not equal to the degree
of 
$(\mu_1^{p^{j_1}}\mu_2^{p^{j_2}}\ldots\mu_n^{p^{j_n}})\mu_s^{p^{j_{n+1}}}$,
where $j_i\geq 0$ for $1\leq i \leq n+1$.  
\end{corr}
\begin{proof}
 By Lemma~\ref{lemma:p-adic_generators} the sum of the coefficients in the
$p$-adic expansion of $\frac{|x|}2$ is equal to $n-l$ where $l$ is the number of
occurrences of the letter $\sBar$ in $x$. Part~\ref{lemma_part:sBarAdmWord}
in Lemma~\ref{lemma:AdmissWord} says that $1\leq l \leq \frac{n-1}2$, so
 $\frac{n+1}2\leq n-l\leq n-1$.
By Lemma~\ref{lemma:p-adic_product_mu} the sum of the coefficients in the
$p$-adic expansion of
$\frac{|\mu_1^{p^{j_1}}\mu_2^{p^{j_2}}\ldots\mu_n^{p^{j_n}}|}2$ is equal
to $n$ when $0<n<p$ and $n$ or $1$ when $n=p$. 
Now, $n-l\leq n-1<n<n+1$ and when $n=p$ then $1<\frac{n+1}2=\frac{p+1}2\leq
n-l$, proving the first claim. 

The sum of the
coefficients in the
$p$-adic expansion of
$\frac{|(\mu_1^{p^{j_1}}\mu_2^{p^{j_2}}\ldots\mu_n^{p^{j_n}})\mu_i^{p^{j_{n+1}}}
| } 2$ is equal to
$n+1$ when $0<n<p-1$ and $n+1$ or $n-p+2$ when $p-1\leq n\leq p$. When
$n=p-1\geq 4$ then
$1<\frac{n+1}2=\frac{p}2\leq n-l$ and  when $n=p\geq 5$ then $2<\frac{n+1}2 =
\frac{p+1}2\leq n-l$, proving the second claim.
\end{proof}

\begin{lemma} \label{lemma:dimOfn-primitive}
Let $n\leq p$, and let $P\subseteq \bigotimes_{U\subsetneq \bf{n}}B_U$ be the
$\bb{F}_p$-submodule generated by all products $z_{U_1}\cdots
z_{U_k}$,
where $U_1,\ldots, U_k$ is a partition of $\bf{n}$, and, 
$z_{U_i}$ is a primitive element in
$B_{U_i}$, for every $i$. Then $P_{2pi-1}=0$ for every $i\geq 2$, and the
module 
$\bigoplus_{i\geq 2}P_{2pi}$ is contained
in the module generated by all the elements
$\mu_1^{p^{j_1}}\mu_2^{p^{j_2}}\ldots\mu_n^{p^{j_n}}$, where $j_i\geq 0$ for
$1\leq i \leq n$.

\end{lemma}
\begin{proof}
 In a divided power algebra $\Gamma(x)$, the only primitive element are non-zero scalar multiples of
 $\gamma_1(x)$, and in a polynomial algebra $P(x)$ the primitive elements are generated by
$x^{p^j}$. By Proposition~3.12 in \cite{MilnorMoore65}, the
primitive elements in $B_{U_i}$ are thus linear combinations of monic words
 $w_i$ of length $|U_i|$ when $|U_i|>1$, and $\mu^{p^{j_i}}_{U_i}$ when
$|U_i|=1$.
Assume without loss of generality that $z$ is a product of such elements. 

Observe that the degree of a word starting with $\varphi^k$, $\sBar^k\sBar $
or $\mu^{p^k}$ is
$0$ modulo $2p$ when $k\geq1$. Thus multiplication with one of these words
will not
change the degree of the product modulo $2p$. 
The degree of $\varphi^0x$ and $\mu$ is $2$ modulo $2p$, and finally the degree
of $\sBar^0\sBar x$ is $2+|x|$ modulo $2p$. 

Except for the products $\mu_1^{p^{j_1}}\ldots\mu_n^{p^{j_n}}$, 
the smallest $n$ where the degree of $z$ is $0$ modulo $2p$ is thus $n=p+2$
where $z$ may be equal to 
$\mu_1\cdots
\mu_{p-2}\cdot\mu_{p-1}^{p^k}\cdot
\sBar_{p+2}^0\sBar_{p+1}\mu_{p}$.
Similarly, the smallest $n$ where the degree of $z$ can be $-1$ modulo $2p$
is $n=p+1$, where $z$ might be equal to 
$\mu_1\cdots \mu_{p-2}\cdot\mu_{p-1}^{p^k}\cdot
\sBar_{p+1}\mu_{p}$.
\end{proof}

This lemma is about which elements in $\Loday(T^\uord{n}_{n-1})$ are
simultaneously primitive in all $n$ Hopf algebra structures. For example
$\mu_1\mu_2^p\mu_3^{p^2}$ is simultaneously primitive in $\Loday(T^\uord{3})$
since it's a product of elements that are primitive in the different
circles. We only gain control over the degrees of the elements, but that is
sufficient for our needs. It's probably a very special case of a more general
statement about simultaneously primitive elements in an $S$-fold Hopf algebra,
but a more general statement has eluded us.

Given a finite subcategory $\simp\subseteq \catFinSet$ and a finite set $U$ we define 
$\simp|_U$, the \emph{restriction} of $\simp$ to $U$, to be the full
subcategory of $\simp$ with
objects $\{V\cap U|V\in \simp\}$.
The \emph{dimension} of $\simp$, is the the maximal cardinality of the
sets in $\simp$. 

\begin{lemma} \label{lemma:nPrimElmInTorus}
Assume part~3 of Theorem~\ref{thm:smash_over_torus} holds for $1\leq k\leq n-1$. Let $S$ 
be an object in $\catFinSet$ and let $\Delta$ be a saturated subcategory 
(see Definition~\ref{def:satSubCat}) of
$\catSpanSet{S}$ with dimension at
most $n-1$. 

Let $V\in \catFinSet$ and define $\bb{N}_V\subseteq
\bb{N}$ to be the set 
of degrees of monic words in $B_V$ when $|V|\geq 2$, and the set
$\{2p^i\}_{i\geq 0}$, the set of degrees of $\mu_v^{p^i}$, when $V=\{v\}$. Let
$\bb{N}_{\Delta}\subseteq \bb{N}$ be the set 
$$\bb{N}_{\Delta}=\Big\{\sum_{U_i\in \{U_1,\ldots, U_j\}}r_{U_i}\
\Big|
U_1,\ldots U_j ,\text{ is a partition of }S\text{ with } U_i\in \Delta
\text{ and }r_{U_i}\in
\bb{N}_{U_i}\Big\}.$$
If $z\in \Loday(T^\simp)$ is $S$-fold primitive, then $|z|\in
\bb{N}_{\Delta}$.
\end{lemma}
\begin{proof}
 We prove it by induction on the number of sets in $\Delta$. If
$S\setminus \big(\bigcup_{U\in \Delta}U\big)=W\not=\emptyset$, there
are no $S$-fold primitive elements in $\Loday(T^\simp)$, since 
$\Loday(T^{\simp|_{S\setminus j}})= \Loday(T^\simp)$ for any $j\in W$, so 
$\comult_S^j=\id\colon \Loday(T^\simp)\rightarrow \Loday(T^\simp)$. If
$S=\{s\}$ and $\Delta=S$ the lemma holds since $\Loday(T^\Delta)=B_{\{s\}} =
P(\mu_s)$, and the primitive elements are generated by $\mu_s^{p^i}$ for $i\geq
0$.

Let $V\in \Delta$ be a maximal set in $\Delta$, i.e., if $V\subseteq W\in
\Delta$ then $V=W$. Let $\widehat{\Delta}$ be the full subcategory
of
$\Delta$ not containing $V$.

Let $z^V_0,z^V_1,\ldots$ be an ordered monomial basis of
$B_V\subseteq \Loday(T^V)\cong \bigotimes_{U\subseteq V}B_U$ ordered so that
$|z^V_i|\leq|z^V_{i+1}|$ for all $i\geq 0$. Note that $z^V_0=1$.

When $z\not=0$ we can write $z$ uniquely as 
\begin{equation} \label{eq:proofSimPrim}
z=z^V_lx^{\widehat V}_l+z^V_{l-1}x^{\widehat V}_{l-1}+\ldots
+z_0^Vx_0^{\widehat
V},\end{equation}
where $x^{\widehat V}_i$ are elements in
$\Loday(T^{\widehat{\Delta}})\cong \bigotimes_{U\in \Delta, U\not=V}B_U$,
and $x^{\widehat V}_l\not=0$.
This is possible since $\Loday(T^\Delta)\cong
\Loday(T^{\widehat{\Delta}})\otimes B_V$.
If $l=0$, then $z\in
\Loday(T^{\widehat{\Delta}})$ and we are
done by the induction hypothesis.

Otherwise, given $j\in V$, assume $x_l^{\widehat{V}} \not\in
\Loday(T^{\Delta|_{S\setminus j}})\subseteq \Loday(T^\simp)$.
Then 
\begin{align*}
 \comult_S^j(z) &= \comult_S^j(z^V_l)\comult_S^j(x^{\widehat
V}_l)+\comult_S^j(z^V_{l-1})\comult_S^j(x^{\widehat V}_{l-1})+\ldots
+\comult_S^j(z_0^V)\comult_S^j(x_0^{\widehat
V}) \\
&=\big(1\otimes z_l^V+z_l^V\otimes 1 + \sum (z_l^V)'\otimes
(z_l^V)''\big)\big(1\otimes x_l^{\widehat{V}} + x_l^{\widehat{V}}\otimes 1 +
\sum (x_l^{\widehat{V}})'\otimes (x_l^{\widehat{V}})''\big)\\
&\phantom{=}\; +\ldots  \\
&=1\otimes z_l^Vx_l^{\widehat{V}}+ z_l^Vx_l^{\widehat{V}}\otimes 1 +
z_l^V\otimes x_l^{\widehat{V}}+  x_l^{\widehat{V}}\otimes z_l^V +\ldots
\end{align*}
Now, $\comult_S^j:\Loday(T^{\widehat{\Delta}})\rightarrow
\Loday(T^{\widehat{\Delta}})\otimes_{\Loday(T^{\widehat{\Delta}|_{S\setminus j}})}
\Loday(T^{\widehat{\Delta}})$, so the expression on the last line can not be equal to
$z\otimes 1 + 1\otimes z$ due to the summands $z_l^V\otimes x_l^{\widehat{V}}$
and $x_l^{\widehat{V}}\otimes z_l^V$ and the fact that $z_l^V,\ldots,z_0^V$ is
part of a basis and $z_l^V$ is of highest degree. Hence we get a contradiction
and  $x_l^{\widehat{V}} \in
\Loday(T^{\Delta|_{S\setminus j}})\subseteq \Loday(T^\simp)$. Doing this for all $j$
gives us that $x_l^{\widehat{V}} \in \Loday(T^{\Delta|_{S\setminus V}})\subseteq
\Loday(T^\simp)$.

For $U\in \Delta$, the projection maps 
$\pr^U_{U\setminus V}:T^U\rightarrow T^{U\setminus V}$ combine
into a map 
$$\pr:T^{\widehat{\Delta}}\rightarrow T^{\Delta|_{S\setminus V}}.$$
Since this map collapses $T^V_{|V|-1}$ to a point, the map
$\sphereproj^V:T^V\rightarrow S^{|V|}$ together
with $\pr$ induces a map  
$$\pr:T^{\Delta}\rightarrow S^{|V|}\vee T^{\Delta|_{S\setminus V}}.$$

For $j\in V$ the pinch map $\comult^j$ on the $j$-th circle induces a
commutative diagram
\begin{equation*}
\xymatrix{
 T^\Delta \ar[dd]^{\comult^j} \ar[r]^-\pr &
  S^V\vee T^{\Delta|_{S\setminus V}}
  \ar[d]^{\pinch \vee \id} \\
  & 
    S^V\vee  S^V\vee T^{\Delta|_{S\setminus V}}\\
  T^\Delta\amalg_{T^{\Delta|_{S\setminus j}}}T^\Delta 
  \ar[r]^-{\pr\amalg \pr} &
  (S^V\vee T^{\Delta|_{S\setminus V}})
  \amalg_{T^{\Delta|_{S\setminus V}}} 
  (S^V\vee T^{\Delta|_{S\setminus V}}) \ar[u]_\cong.}
\end{equation*}
Similarly, for $j\in S\setminus V$ the pinch map $\comult^j$ on the $j$-th
circle induces a commutative diagram 
\begin{equation*}
\xymatrix{
 T^\Delta \ar[dd]^{\comult^j} \ar[r]^-\pr &
  S^V\vee T^{\Delta|_{S\setminus V}}
  \ar[d]^{\id\vee \comult^j} \\
  & 
    S^V\vee (T^{\Delta|_{S\setminus V}}\amalg_{T^{\Delta|_{(S\setminus
V)\setminus j}}}T^{\Delta|_{S\setminus V}} )\\
  T^\Delta\amalg_{T^{\Delta|_{S\setminus j}}}T^\Delta 
  \ar[r]^-{\pr\amalg \pr} &
  (S^V\vee T^{\Delta|_{S\setminus V}})
  \amalg_{S^V\vee T^{\Delta|_{(S\setminus V)\setminus j}}} 
  (S^V\vee T^{\Delta|_{S\setminus V}}) \ar[u]_\cong.}
\end{equation*}

Applying the functor $\Loday(-)$ to these two diagrams yields for $j\in V$ 
a commutative diagram
\begin{equation} \label{diag:comultB_V}
\xymatrix{
 \Loday (T^\Delta) \ar[d]^{\comult^j_S} \ar[r]^-\pr &
  B_V\otimes \Loday(T^{\Delta|_{S\setminus V}})
  \ar[d]^{\comult_{B_V}\otimes \id} \\
  \Loday(T^\Delta)\otimes_{\Loday(T^{\Delta|_{S\setminus j}})}\Loday(T^\Delta) 
  \ar[r]^-{\pr\otimes \pr} &
  B_V\otimes B_V\otimes \Loday(T^{\Delta|_{S\setminus V}}),}\end{equation}
and for $j\in S\setminus V$ a commutative diagram
\begin{equation} \label{diag:comultNotB_V}
\xymatrix{
  \Loday(T^\Delta) \ar[d]^-{\comult_S^j} \ar[r]^-\pr &
  B_V\otimes \Loday(T^{\Delta|_{S\setminus V}})\ar[d]^-{\id\otimes
\comult_{S\setminus V}^j} \\
  \Loday(T^\Delta)\otimes_{\Loday(T^{\Delta|_{S\setminus j}})}\Loday(T^\Delta) 
  \ar[r]^-{\pr\otimes \pr} &
  B_V\otimes (\Loday(T^{\Delta|_{S\setminus V}})
\otimes_{\Loday(T^{\Delta|_{(S\setminus V)\setminus j}})}
\Loday(T^{\Delta|_{S\setminus V}})).}\end{equation}

We have proved that  $x^{\widehat
V}_l\in \Loday(T^{\Delta|_{S\setminus V}})$, so
\begin{equation*}
\pr(z)=z^V_lx^{\widehat V}_l+z^V_{l-1}\pr(x^{\widehat
V}_{l-1})+\ldots
+z_0^V\pr(x_0^{\widehat V}),\end{equation*}
is non-zero since $z^V_l,\ldots,z^V_0$ is part of a basis.
From Diagram~\ref{diag:comultB_V} we know that $\pr(z)$ must be primitive in the
$\Loday(T^{\Delta|_{S\setminus V}})$-Hopf algebra $B_V\otimes
\Loday(T^{\Delta|_{S\setminus V}})$, where the Hopf algebra structure is
induced by the $\bb{F}_p$-Hopf algebra structure on $B_V\cong B_{|V|}\cong \Loday(S^V).$ By the graded version of Proposition
3.12 in \cite{MilnorMoore65}, this implies that
if $\pr(x^{\widehat
V}_i)$ is non-zero then $z^V_i$ is a $V$-labeled monic
word when $|V|\geq 2$ or an element $\mu_v^{p^m}$ for some $m$ when $V=\{v\}$. 
It follows from Diagram~\ref{diag:comultNotB_V} that when $\pr(x^{\widehat
V}_i)\not=0$ it is $S\setminus V$-fold primitive. By induction the
Lemma holds for
$\pr(x^{\widehat V}_i)$, finishing the proof. 
\end{proof}

\begin{corr} \label{corr:dimOfnFoldPrim}
Given $n\leq p$, assume part~3 og Theorem~\ref{thm:smash_over_torus} holds for $1\leq k <
n$. Let $y$ be an $\uord{n}$-fold primitive element in $\Loday(T^\uord{n}_{n-1})$.
 If $x$ is an admissible word of length $n$ and degree $0$ modulo $2p$, 
, then $|x|-1\not=|y|$. 
If $z$ is an admissible word of length $n$ and of even degree, then
$|z^p|\not=|y|$. 
\end{corr}
\begin{proof}
 When $2\leq n\leq p$, Lemma~\ref{lemma:indecWord} says the admissible words of
length $n$ and degree $0$ modulo $2p$
are those that start with $\varphi^i$ or $\sBar^i$ for $i\geq 1$. Hence
$x=\sBar^ix'$ or $x=\varphi^ix'$ for $x'$ some admissible word of length
$n-1$. 
The element $\sBar^ix'$ is in degrees greater than or equal to $4p$. By Lemma
\ref{lemma:nPrimElmInTorus} and~\ref{lemma:dimOfn-primitive}, there are no
$\uord{n}$-fold primitive elements in $\Loday(T^\uord{n}_{n-1})$ in dimension
$2pm-1$ for $m\geq 2$, and hence $|x|-1\not=|y|$.

If $z$ is an admissible word of length $n$ and even degree, then by 
Lemma~\ref{lemma:indecWord} we have 
$z=\sBar^iz'$ or $z=\varphi^iz'$ for some admissible word $z'$ of length
$n-1$. So $|z^p|=|\sBar^{i+1}z'|$ or $|z^p|=|\varphi^{i+1}z'|$.
By Lemma~\ref{lemma:nPrimElmInTorus} and~\ref{lemma:dimOfn-primitive} the
degrees of the $\uord{n}$-fold primitive elements in $\Loday(T^\uord{n}_{n-1})$
are equal to the degrees of the products $\mu_1^{p^{j_1}}\ldots\mu_n^{p^{j_n}}$. By
Corollary~\ref{corr:degOfComult} neither $\sBar^{i+1}z'$ nor $\varphi^{i+1}z'$
is in the same degree as one of the products
$\mu_1^{p^{j_1}}\ldots\mu_n^{p^{j_n}}$, and hence $|z^p|\not=|y|$. 
\end{proof}

\section{B\"okstedt spectral sequence argument}
\label{sec:bokstedtlemma}
To prove that there are no non-zero $d^2$ differential when computing $\Loday(T^n)$ we look at the B\"okstedt spectral sequence

\begin{lemma} \label{lemma:bokstedttorus}
Given $n>2$ and a prime $p>2$. Assume that $$
\Loday(T^\uord{n-1})\cong \bigotimes_{U\subseteq \uord{n-1} } B_U,$$ where $B_U$ is
described in Definition~\ref{def:labelB_U}. Then
\begin{enumerate}
 \item 
The B\"okstedt spectral sequence
calculating $H_*(\loday_{T^\uord{n}}\HF)$ has $E^2$-page 
$$\overline{E}^2(T^\uord{n})\cong A_*\otimes \Loday(T^{\bf n-1})\otimes
\bigotimes_{\emptyset\not=U\subseteq \uord{n-1}}B_{U\cup \{n\}}\otimes
E(\sBok_n\oxi_1,\sBok_n\oxi_2,\ldots)\otimes
\Gamma(\sBok_n\otau_0,\sBok_n\otau_2,\ldots),$$
\item \label{lemma_part:noDiff} There are no differentials $d^r$ when $r<p-1$,
so $\overline{E}^2(T^\uord{n})=
\overline{E}^{p-1}(T^\uord{n})$.
\item \label{lemma_part:bokDiff}
 There are differentials 
$$d^{p-1}(\gamma_{p^l}(\sBok_n\otau_i))=\sBok_n\oxi_{i+1}\cdot
\gamma_{p^l-l}(\sBok_n\otau_i).$$
\setcounter{enumi_saved}{\value{enumi}}
\end{enumerate}

If, in addition, given $m\geq 0$ the homomorphism $\attach^\uord{n}\colon
\Loday(S^{n-1})\rightarrow
\Loday(T^\uord{n}_{n-1})$ factors through $\bb{F}_p$ in degrees less than or
equal to $2pm-1$ and the spectral
sequence $E^*(T^\uord{n})$ collapses in total
degrees
less than or equal to $2pm-1$ (that is $E^2(T^\uord{n})=E^\infty(T^\uord{n})$ in
these degrees) then:

\begin{enumerate}
\setcounter{enumi}{\value{enumi_saved}}
\item \label{lemma_part:possDiff} The only
other possible non-zero differentials in
$\overline{E}^{p-1}(T^\uord{n})$ starting in total degrees less than or equal to
$2p(m+1)-1$, are 
$$d^{p-1}(\gamma_{p^l}(\sBok_n
x))=\gamma_{p^l-p}(\sBok_nx)\sum_ir_{x,i}d^{p-1}(\gamma_p(\sBok_n\otau_i)),$$
where $x$ is a generator in $\Loday(T^\uord{n-1})$ of odd
degree and $r_{x,i}\in \Loday(T^\uord{n-1})\subset
\overline{E}^{p-1}_{0,*}(T^\uord{n})$.
\item \label{lemma_part:abutment}
Let $B'_U\subsetneq \overline{E}^2(T^\uord{n})$ be the algebra, isomorphic to
$B_U$,
that has the same generators as $B_U$, except that we exchange the generators 
$\gamma_{p^l}(\sBok_nx)$ in degrees less than $2p(m+1)$ with the infinite
cycles
$$\gamma_{p^l}((\sBok_nx)')=\sum_{j=0}^{p^{l-1}}\big(
(-1)^j\gamma_{p^l-pj}((\sBok_nx)')
\sum_{\alpha\in \bb{N}^\bb{N}, |\alpha|=j}\prod_{i\in
\bb{N}}r_i^{\alpha_i}\gamma_{p\alpha_i}(\sBok_n\otau_i)\big),$$
where $|\alpha|=\sum_{i\in \bb{N}}\alpha_i$, and the convention is that $0^0=1$,
$\gamma_0(x)=1$, and $\gamma_{i}(x)=0$ when $i<0$.

When $s+t\leq 2p(m+1)-2$ we get an isomorphism
$$\overline{E}^\infty_{s,t}(T^\uord{n}) \cong A_*\otimes \Loday(T^{\bf
n-1})\otimes
\bigotimes_{\emptyset\not=U\subseteq \uord{n-1}}B'_{U\cup \{n\}}\otimes
P_p(\sBok_n\otau_0,\sBok_n\otau_1,\ldots).
$$
\end{enumerate}
\end{lemma}

\begin{proof}
By Proposition 2.1 in \cite{McClureStaffeldt93} and the K\"unneth isomorphism there are
isomorphisms of $H_*(\loday_{T^\uord{n-1}}\HF)\cong A_*\otimes
\Loday(T^\uord{n-1})$-Hopf algebras
\begin{align*}\overline{E}^2(T^\uord{n})&=
HH_*(H_*(\loday_{T^\uord{n-1}}\HF))\cong 
  H_*(\loday_{T^\uord{n-1}}\HF)\otimes
\tor^{A_*\otimes \Loday(T^\uord{n-1})}(\bb{F}_p,\bb{F}_p)\\ 
&\cong
A_*\otimes \Loday(T^\uord{n-1})\otimes \bigotimes_{U\subseteq
\uord{n-1}}\tor^{B_U}(\bb{F}_p,\bb{F}_p)\otimes
\tor^{A_*}(\bb{F}_p,\bb{F}_p) \\
&\cong  
A_*\otimes \Loday(T^{\bf n-1})\otimes
\bigotimes_{\emptyset\not=U\subseteq \uord{n-1}}B_{U\cup \{n\}}\otimes
E(\sBok_n\oxi_1,\sBok_n\oxi_2,\ldots)\otimes
\Gamma(\sBok_n\otau_0,\sBok_n\otau_2,\ldots),
\end{align*}
where the empty set is left out in the tensor product in the last line, since
$\tor^{B_\emptyset}(\bb{F}_p,\bb{F}_p)$ is isomorphic to $\bb{F}_p$. See \cite{AngeltveitRognes05} for more details on the Hopf algebra structure.

\emph{Proof of~\ref{lemma_part:noDiff}:}
 The B\"okstedt spectral sequence $\overline{E}^2(T^\uord{n})$ is an
$A_*\otimes \Loday(T^\uord{n-1})$-Hopf-algebra spectral sequence.
 By Proposition~\ref{prop:shortestDiff}, the shortest
differential is therefore from an
indecomposable element to a primitive element. By the graded version of Proposition
3.12 in \cite{MilnorMoore65} the primitive elements are linear
combinations of the monic words in $\bigotimes_{\emptyset\not=U\subseteq
\uord{n-1}}B_{U\cup
\{n\}}$, and the elements $\sBok_n\oxi_{i+1}$ and $\gamma_1(\sBok_n\otau_i)$ for
$i\geq 0$. The primitive elements are thus in
filtration $1$ and
$2$. The indecomposable elements are linear combinations of the
$\bb{F}_p$-algebra
generators in $\bigotimes_{\emptyset\not=U\subseteq \uord{n-1}}B_{U\cup
\{n\}}\otimes
E(\sBok_n\oxi_1,\sBok_n\oxi_2,\ldots)\otimes
\Gamma(\sBok_n\otau_0,\sBok_n\otau_1,\ldots)$, given by the  admissible words
in $\bigotimes_{\emptyset\not=U\subseteq \uord{n-1}}B_{U\cup
\{n\}}$ together with the elements $\sBok_n\oxi_j$ and
$\gamma_{p^k}(\sBok_n\otau_j)$, and they are in
filtration $1,2$ and $\pi$ for $i>0$. The indecomposable elements in
filtration $p$ are generated by 
 $\gamma_p(\sBok_nx)$ for a generator $x$ in $A_*\otimes \Loday(T^{\bf n-1})$ of
odd
degree. By Theorem 1 in \cite{Hunter96}, these elements survive to
$\overline{E}^{p-1}(T^\uord{n})$, so 
$\overline{E}^2(T^\uord{n})=\overline{E}^{p-1}(T^\uord{n})$.

\emph{Proof of~\ref{lemma_part:bokDiff}:}
 Theorem 1 in \cite{Hunter96} also gives us the differentials 
\begin{equation}\label{eq:bokstedDiffInProof}
d^{p-1}(\gamma_{p+k}(\sBok_n\otau_i))=u_i\sBok_n\oxi_{i+1}\cdot
\gamma_k(\sBok_n\otau_i),
\end{equation}
where $u_i$ are units in $\bb{F}_p$.

\emph{Proof of~\ref{lemma_part:possDiff}:}
When $m=0$, there is nothing to prove, since all elements in filtration
$p$ and higher are in degrees at least $2p$. 
Since $\loday_{T^\uord{n}}\HF$ is an $\HF$-module it is a generalized Eilenberg
Mac~Lane spectrum, so the
Hurewicz homomorphism induces an isomorphism between the 
$\bb{F}_p$-modules  $A_*\otimes
\Loday(T^\uord{n})$ and $H_*(\loday_{T^\uord{n}}\HF)$.

 From the assumption that $\attach^\uord{n}$ factors through $\bb{F}_p$ in
degrees less than or equal to $2pm-1$ and that  
$E^2(T^\uord{n})_{<2pm}\cong E^\infty(T^\uord{n})_{<2pm}$, we know the
dimension of $H_*(\loday_{T^\uord{n}}\HF)$ as an $\bb{F}_p$-module in
degrees
less than $2pm$. We will
show that if there are
other differentials in the spectral sequence $\overline{E}^2(T^\uord{n})$ 
starting in degrees less than or equal to $2p(m+1)-1$, than those in part~\ref{lemma_part:bokDiff} and~\ref{lemma_part:possDiff} of the
lemma, the dimension of $\overline{E}^\infty(T^\uord{n})$ is smaller than the
abutment of the spectral
sequence, which is equal to $H_*(\loday_{T^\uord{n}}\HF)$, thus giving us a
contradiction.

Assume the only $d^{p-1}$-differentials in the B\"okstedt spectral sequence
$\overline{E}^2(T^\uord{n})$ are those generated by 
\ref{eq:bokstedDiffInProof}. 
Lemma~\ref{lemma:pTermSS} yields an isomorphism 
$$\overline{E}^p(T^\uord{n})\cong A_*\otimes \Loday(T^{\bf n-1})\otimes
\bigotimes_{\emptyset\not=U\subseteq \uord{n-1}}B_{U\cup \{n\}}\otimes
P_p(\sBok_n\otau_0,\sBok_n\otau_1,\ldots).$$
Proposition~\ref{prop:B_n} together with the assumption that
$\attach^\uord{n}$ factors through $\bb{F}_p$ in
degrees less than $2pm-1$ and that $E^2(T^\uord{n})_{<
2pm}=E^\infty(T^\uord{n})_{< 2pm}$,
 gives us an $\bb{F}_p$-module
isomorphism
\begin{align*}
E^\infty(T^\uord{n})_{< 2pm} &= 
E^2(T^\uord{n})_{<2pm} =(\Loday(T^\uord{n}_{n-1})\otimes B_\uord{n})_{<2pm} \\
 &\cong 
\Big(\bigotimes_{U\subseteq \uord{n}} B_U\Big)_{<2pm}\cong
\Big(\bigotimes_{U\subseteq \uord{n-1}} B_U\otimes
\bigotimes_{\emptyset\not=U\subseteq\uord{n-1}} B_{U\cup \{n\}}
\otimes B_{\{n\}}\Big)_{<2pm} \\
&\cong\Big( \Loday(T^\uord{n-1})\otimes 
\bigotimes_{\emptyset\not=U\subseteq\uord{n-1}} B_{U\cup \{n\}}
\otimes B_{\{n\}}\Big)_{<2pm}. 
\end{align*}

By Proposition \ref{prop:B_n}, there is an $\bb{F}_p$-module isomorphism from
$P_p(\sBok_n\otau_0,\sBok_n\otau_1,\ldots)$
to  $B_{\{n\}}$ given by mapping $\sBok_n\otau_i$ to $\mu_n^{p^i}$, and this
isomorphism yields an 
 $\bb{F}_p$-module isomorphism 
$\overline{E}^p(T^\uord{n})_{< 2pm}\cong(A_*\otimes
\Loday(T^\uord{n}))_{<2pm}\cong H_*(\loday_{T^\uord{n}}\HF)_{<2pm}$.

Assume there is a $d^{p-1}$-differential with image in
$\overline{E}^{p-1}(T^\uord{n})_{<
2pm}$, which doesn't have image in the ideal 
$(\sBok_n\oxi_1,\sBok_n\oxi_2,\ldots)\subseteq
\overline{E}^{p-1}(T^\uord{n})$,
which is the ideal generated by the images of all the differentials in
equation~\ref{eq:bokstedDiffInProof}. 
Then, in the degree
of the target of this differential, the
dimension of the $\bb{F}_p$-module
$\overline{E}^\infty(T^\uord{n})_{<2pm}$ would be smaller than the dimension of 
$H_*(\loday_{T^\uord{n}}\HF)_{< 2pm}\cong \overline{E}^p(T^\uord{n})_{< 2pm} $,
 giving us a contradiction. 

To find all possible $d^{p-1}$-differentials with target in
$(\sBok_n\oxi_1,\sBok_n\oxi_2,\ldots)$ it suffices to look at differentials
from indecomposable elements. 
 Possible non-zero $d^{p-1}$-differentials with image in
$\overline{E}^{p-1}_{< 2pm}$ are thus generated by 
$d^{p-1}(\gamma_{p^k}(\sBar^0_nx))$ and $d^{p-1}(\gamma_{p^k}(\varphi_nx))$ where
$x$ is an $U$-admissible word in $B_U\subseteq \Loday(T^{\bf n-1})$ for some
$\emptyset \not=U\subset \uord{n-1}$ of odd degree at most $2mp^{1-k} -1$ and
even degree at most 
$\frac{2mp^{1-k}-2}p$, respectively, and $k\geq 1$.  
From the calculation
$$\psi(d^{p-1}(\gamma_{p^k}(\varphi_nx)) )= d^{p-1}(\psi(\gamma_{p^k}(\varphi_nx))) =
d^{p-1}(\sum_{i+j=p^k}\gamma_i(\varphi_nx)\otimes \gamma_j(\varphi_nx)),$$
we see by induction on $k$, that $d^{p-1}(\gamma_{p^k}(\varphi_nx))$ must be
primitive.
Thus it is zero, since when $k\geq 1$, it is in filtration greater than or equal
to $p+1$, while
the primitive elements are in filtration $1$ and $2$.

For the elements $\gamma_{p^k}(\sBar^0_nx)$, Theorem 1 in \cite{Hunter96} yields the
formula 
$$d^{p-1}(\gamma_{p+k}(\sBar^0_nx)) = (\sBok\beta Q^{\frac{|x|+1}2}x)\cdot
\gamma_k(\sBar^0_nx),$$
so $\gamma_{p+k}(\sBar^0_nx)$ is a cycle if and only if $\gamma_{p}(\sBar^0_nx)$
is a cycle. 

In $\overline{E}^{p-1}_{1,*}(T^\uord{n})$, the ideal generated by the elements 
$\sBok_n\oxi_1,\sBok_n\oxi_2,\ldots$ is equal to $A_*\otimes
\Loday(T^\uord{n-1})\{\sBok_n\oxi_1,\sBok_n\oxi_2,\ldots\}$. 
Thus, if $d^{p-1}(\gamma_{p}(\sBar^0_nx))$ is non-zero, $\sBok_n\beta
Q^{\frac{|x|+1}2}x$ must be an element in $A_*\otimes
\Loday(T^\uord{n-1})\{\sBar_n\oxi_1,\sBar_n\oxi_2,\ldots\}$.
Since differentials from a $A_*$-comodule primitive has target an $A_*$-comodule
primitive, $\sBok_n\beta
Q^{\frac{|x|+1}2}x$ must actually be an element in
$\Loday(T^\uord{n-1})\{\sBok_n\oxi_1,\sBok_n\oxi_2,\ldots\}$. Hence, 
$$\sBok_n\beta Q^{\frac{|x|+1}2}x =
\sum_ir_{x,i}d^{p-1}(\gamma_p(\sBok_n\otau_i)),$$
where $r_{x,i}$ are elements in $\Loday(T^\uord{n-1})$. 
  
\emph{Proof of part~\ref{lemma_part:abutment}:}
By Lemma~\ref{lemma:changeBasisInSS}, the elements 
$\gamma_{p^l}((\sBok_nx)')$ in part
\ref{lemma_part:abutment} are cycles, and $\overline{E}^{p-1}$ is isomorphic as
an algebra to
$$\overline{E}^{p-1}(T^\uord{n})\cong A_*\otimes \Loday(T^{\bf n-1})\otimes
\bigotimes_{\emptyset\not=U\subseteq \uord{n-1}}B'_U\otimes
E(\sBok_n\oxi_1,\sBok_n\oxi_2,\ldots)\otimes
\Gamma(\sBok_n\otau_0,\sBok_n\otau_2,\ldots).$$

In total degrees less
than or equal to $2p(m+1)-1$, all elements in
$\bigotimes_{\emptyset\not=U\subseteq \uord{n-1}}B'_U$ are cycles. 
Thus,  when $s+t\leq 2p(m+1)-2$, the only differentials are those in part
\ref{lemma_part:bokDiff}, so by Lemma~\ref{lemma:pTermSS} there is
an isomorphism
$$\overline{E}^p(T^\uord{n})\cong A_*\otimes \Loday(T^{\bf n-1})\otimes
\bigotimes_{\emptyset\not=U\subseteq \uord{n-1}}B'_U\otimes
P_p(\sBok_n\otau_0,\sBok_n\otau_1,\ldots),$$
in total degrees less than $2p(m+1)-2$. 

All the algebra generators in filtration greater than $2$ are in total degrees
zero modulo $2p$. All generators in total degrees less than or equal to $2pm$
must be
cycles, because  otherwise, in the degrees of the target of this non-zero
differential, the dimension of the $\bb{F}_p$-module
$\overline{E}^\infty(T^\uord{n})_{<2pm}$ will be smaller than the dimension of
$H_*(\loday_{T^\uord{n}}\HF)_{< 2pm}\cong \overline{E}^p(T^\uord{n})_{< 2pm} $.
Thus there are no more differentials with source in total degrees
less than or equal to $2p(m+1)$, so $\overline{E}^p(T^\uord{n})_{\leq
2p(m+1)-2}\cong
\overline{E}^\infty(T^\uord{n})_{\leq 2p(m+1)-2}$.
\end{proof}

The final two lemmas are one standard homological calculation, and one easy homological calculation that was used in the previous Lemma.

\begin{lemma} \label{lemma:pTermSS}
Let $R$ be a field of characteristic $p$, and let $E^*$ be a connected $R$-algebra spectral sequence with 
$$E^{p-1}\cong A\otimes_R \Gamma_R(x_0,x_1,\ldots)\otimes
\ext_R(y_1,y_2,\ldots),$$
 where $x_i$ and $y_i$ are in filtration
$1$.
Assume there are non-zero
differentials 
$$d^{p-1}(\gamma_{p+k}(x_i))=\gamma_k(x_i)y_{i+1},$$
for all $i,k\geq 0$. Then
$$E^p\cong A\otimes  P_R(x_0,x_1,\ldots)/(x_0^p,x_1^p,\ldots).$$
\end{lemma}
\begin{proof}
Consider the $R$-algebra 
$\Gamma_R(x_i)\otimes E_R(y_{i+1})$ with differentials given by the equations 
$d^{p-1}(\gamma_{p+k}(x_i))=\gamma_k(x_i)y_{i+1}$. The cycles are
$\gamma_k(x_i)$ for $k\leq p-1$ and $\gamma_k(x_i)y_{i+1}$ for all $k$, but
this last family are also boundaries, so the homology is $P_R(x_i)/(x_i^p)$. 
 Since $R$ is a field, the lemma now follows from the K\"unneth isomorphism, 
\end{proof}

\begin{lemma} \label{lemma:changeBasisInSS}
Let $R$ be a field of characteristic $p$, and and let $E^*$ be a connected $R$-algebra spectral sequence with 
$$E^{p-1}\cong A\otimes \Gamma_R(x_0,x_1,\ldots)\otimes
\ext_R(y_1,y_2,\ldots)\otimes
\Gamma_R(z).$$
Assume there are differentials 
\begin{align*}
d^{p-1}(\gamma_{p+k}(x_i))&=\gamma_k(x_i)y_{i+1} \\
d^{p-1}(\gamma_{p+k}(z))&=\gamma_k(z)\cdot\sum_{l\in
\nat}r_ly_{l+1},
\end{align*}
where $r_l$ are elements in $R$.

For $k>\geq 0$ define $\gamma_{p^k}(z')$ by the formula
\begin{equation} \label{eq:changeCycles}
\gamma_{p^k}(z')=\sum_{j=0}^{p^{k-1}}\big(
(-1)^j\gamma_{p^k-pj}(z)
\sum_{\alpha\in \nat^\nat, |\alpha|=j}\prod_{i\in
\nat}r_i^{\alpha_i}\gamma_{p\alpha_i}(x_i)\big),\end{equation}
where $|\alpha|=\sum_{k\in \nat}\alpha_i$, and the convention is that $0^0=1$,
$\gamma_0(x)=1$, and $\gamma_{j}(x)=0$ when $j<0$.

There is an $R$-algebra isomorphism
$$A\otimes \Gamma_R(x_0,x_1\ldots)\otimes \ext_R(y_1,y_2,\ldots)\otimes
\Gamma_R(z')\cong 
A\otimes \Gamma_R(x_0,x_1\ldots)\otimes \ext_R(y_1,y_2,\ldots)\otimes
\Gamma_R(z),$$
induced by the equations in \ref{eq:changeCycles}. 
Furthermore, the elements $\gamma_{p^k}(z')$ are $d^{p-1}$ cycles in $E^{p-1}$.
\end{lemma}
\begin{proof}
First we show that the elements $\gamma_{p^k}(z')$ are cycles. Applying the Leibniz rule several times gives
\begin{multline} \label{eq:HopSSdiff}
d^{p-1}(\gamma_{p^k}(z'))= \sum_{j=0}^{p^{k-1}}\Big(
(-1)^j\gamma_{p^k-p(j+1)}(z)\big(\sum_{l\in
\nat}r_ly_{l+1}\big)
\sum_{\alpha\in \nat^\nat, |\alpha|=j}\prod_{i\in
\nat}r_i^{\alpha_i}\gamma_{p\alpha_i}(x_i)\Big) \\
   + \sum_{j=0}^{p^{k-1}}\big( (-1)^j\gamma_{p^k-pj}(z)
\sum_{\alpha\in \nat^\nat, |\alpha|=j}\sum_{l\in
\nat}r_l^{\alpha_l}\gamma_{p(\alpha_l-1)}(x_l)y_{l+1}\prod_{
l\not=i\in \nat}r_i^{\alpha_i}\gamma_{p\alpha_i}(x_i)\big),
\end{multline}
and there are no extra signs, since all the factors in the
expression of $\gamma_{p^k}(z)$ are in even
degrees.

In the first sum in equation \ref{eq:HopSSdiff} observe that
\begin{align*}
\big(\sum_{l\in \nat}r_ld^{p-1}(\gamma_p(x_l))\big)
&\sum_{\alpha\in \nat^\nat, |\alpha|=j}\prod_{i\in
\nat}r_i^{\alpha_i}\gamma_{p\alpha_i}(x_i)  \\
&=
\sum_{l\in \nat}
\sum_{\alpha\in \nat^\nat,
|\alpha|=j+1}r_l^{\alpha_l}y_{l+1}\gamma_{p(\alpha_l-1)}
(x_l)\prod_{l\not=i\in \nat}r_i^{\alpha_i}\gamma_{p\alpha_i}(x_i)
\end{align*}

Substituting this expression into equation \ref{eq:HopSSdiff} and 
increasing the summation index in the first sum 
with one, the differential is given by
\begin{multline*}
d^{p-1}(\gamma_{p^k}(z'))= \\  \sum_{j=1}^{p^{k-1}+1}\big(
(-1)^{j-1}\gamma_{p^k-pj}(z)
\sum_{\alpha\in \nat^\nat, |\alpha|=j}\sum_{l\in
\nat}r_l^{\alpha_l}\gamma_{p(\alpha_l-1)}(x_l)y_{l+1}\prod_{
l\not=i\in \nat}r_i^{\alpha_i}\gamma_{p\alpha_i}(x_i)\big)
 \\
  + \sum_{j=0}^{p^{k-1}}\big( (-1)^j\gamma_{p^k-pj}(z)
\sum_{\alpha\in \nat^\nat, |\alpha|=j}\sum_{l\in
\nat}r_l^{\alpha_l}\gamma_{p(\alpha_l-1)}(x_l)y_{l+1}\prod_{
l\not=i\in \nat}r_i^{\alpha_i}\gamma_{p\alpha_i}(x_i)\big).
\end{multline*}
The $j=p^{k-1}+1$ summand in the first sum is zero because
$\gamma_{p^k-(p^{k-1}+1)p}(z)=\gamma_{-p}(z)=0$. 
Similarly, the $j=0$ summand in the last sum is zero because $0=j=|\alpha|$
implies that $\alpha_l=0$ for all $l$, and hence
$\gamma_{p(\alpha_l-1)}(x_l)=\gamma_{-p}(x_l)=0$.

The rest of the summands cancel pairwise, due to the factors $(-1)^{j-1}$ and
$(-1)^j$. 
Thus $d^{p-1}(\gamma_{p^k}(z'))=0$.

That $(\gamma_{p^k}(z'))^p=0$ is clear by the Frobenius formula, since every
summand in the expression for $\gamma_{p^k}(z')$ contains a factor in a
divided power algebra. 

The composite 
$$\xymatrix@C=6pc{
\Gamma_R(z)\ar[r]^-{\gamma_{p^k}(z)\mapsto \gamma_{p^k}(z')} &
\Gamma_R(x_0,x_1,\ldots)\otimes \ext_R(y_1,y_2,\ldots)\otimes \Gamma_R(z)
\ar[r]^-{\pr_{\Gamma_R(z)}} &\Gamma_R(z)}$$
equals the identity. 
Hence, the map induced by equation \ref{eq:changeCycles} induces an $R$-algebra
isomorphism 
$$A\otimes \Gamma_R(x_0,x_1\ldots)\otimes \ext_R(y_1,y_2,\ldots)\otimes
\Gamma_R(z')\cong A\otimes 
\Gamma_R(x_0,x_1\ldots)\otimes \ext_R(y_1,y_2,\ldots)\otimes \Gamma_R(z).$$
\end{proof}

\bibliographystyle{alphaInit}

\bibliography{bibliografi}

\begin{thebibliography}{EKMM97}

\bibitem[AHL10]{AngeltveitHillTyler}
V.Angeltveit, M.~A.Hill, and T.Lawson.
\newblock Topological {H}ochschild homology of {$\ell$} and {$ko$}.
\newblock {\em Amer. J. Math.}, 132(2):297--330, 2010.

\bibitem[AR02]{AusoniRognes02}
C.Ausoni and J.Rognes.
\newblock Algebraic {$K$}-theory of topological {$K$}-theory.
\newblock {\em Acta Math.}, 188(1):1--39, 2002.

\bibitem[AR05]{AngeltveitRognes05}
V.Angeltveit and J.Rognes.
\newblock Hopf algebra structure on topological {H}ochschild homology.
\newblock {\em Algebr. Geom. Topol.}, 5:1223--1290 (electronic), 2005.

\bibitem[AR08]{AusoniRognes08}
C.Ausoni and J.Rognes.
\newblock The chromatic red-shift in algebraic {$K$}-theory.
\newblock {\em Monogr. Enseign. Math.}, 40:13--15, 2008.

\bibitem[Aus05]{Ausoni05}
C.Ausoni.
\newblock Topological {H}ochschild homology of connective complex {$K$}-theory.
\newblock {\em Amer. J. Math.}, 127(6):1261--1313, 2005.

\bibitem[BCD10]{BrunCarlssonDundas10}
M.Brun, G.Carlsson, and B.~I.Dundas.
\newblock Covering homology.
\newblock {\em Adv. Math.}, 225(6):3166--3213, 2010.

\bibitem[BDR04]{BaasDundasRognes04}
N.~A.Baas, B.~I.Dundas, and J.Rognes.
\newblock Two-vector bundles and forms of elliptic cohomology.
\newblock In {\em Topology, geometry and quantum field theory}, volume 308 of
  {\em London Math. Soc. Lecture Note Ser.}, pages 18--45. Cambridge Univ.
  Press, Cambridge, 2004.

\bibitem[BDS16]{BrunDundasStolz16}
M.Brun, B.Dundas, and M.Stolz.
\newblock Equivariant structure on smash powers, 2016.

\bibitem[BHM93]{BokstedtHsiangMadsen93}
M.B{\"o}kstedt, W.~C.Hsiang, and I.Madsen.
\newblock The cyclotomic trace and algebraic {$K$}-theory of spaces.
\newblock {\em Invent. Math.}, 111(3):465--539, 1993.

\bibitem[Boa99]{Boardman99}
J.~M.Boardman.
\newblock Conditionally convergent spectral sequences.
\newblock 239:49--84, 1999.

\bibitem[B{\"{o}}k86]{Bokstedt86}
M.B{\"{o}}kstedt.
\newblock Topological hochschild homology of $\mathbb{F}_p$ and $\mathbb{Z}$.
\newblock preprint, 1986.

\bibitem[BR05]{BrunerRognes05}
R.~R.Bruner and J.Rognes.
\newblock Differentials in the homological homotopy fixed point spectral
  sequence.
\newblock {\em Algebr. Geom. Topol.}, 5:653--690 (electronic), 2005.

\bibitem[Car55]{Cartan54}
H.Cartan.
\newblock {\em Alg\`ebres d'{E}ilenberg-{M}ac{L}ane at homotopie}.
\newblock S\'eminaire Henri Cartan, 1954-55.

\bibitem[CDD11]{CarlssonDouglasDundas11}
G.Carlsson, C.~L.Douglas, and B.~I.Dundas.
\newblock Higher topological cyclic homology and the {S}egal conjecture for
  tori.
\newblock {\em Adv. Math.}, 226(2):1823--1874, 2011.

\bibitem[DGM13]{DundasGoodwillieMcCarthy13}
B.~I.Dundas, T.~G.Goodwillie, and R.McCarthy.
\newblock {\em The local structure of algebraic {K}-theory}, volume~18 of {\em
  Algebra and Applications}.
\newblock Springer-Verlag London Ltd., London, 2013.

\bibitem[EKMM97]{EKMM97}
A.~D.Elmendorf, I.Kriz, M.~A.Mandell, and J.~P.May.
\newblock {\em Rings, modules, and algebras in stable homotopy theory},
  volume~47 of {\em Mathematical Surveys and Monographs}.
\newblock American Mathematical Society, Providence, RI, 1997.
\newblock With an appendix by M. Cole.

\bibitem[Hes96]{Hesselholt96}
L.Hesselholt.
\newblock On the {$p$}-typical curves in {Q}uillen's {$K$}-theory.
\newblock {\em Acta Math.}, 177(1):1--53, 1996.

\bibitem[HHR16]{HillHopkinsRavenel2016}
M.~A.Hill, M.~J.Hopkins, and D.~C.Ravenel.
\newblock On the nonexistence of elements of kervaire invariant one.
\newblock {\em Ann. of Math. (2)}, 184(1):1--262, 2016.

\bibitem[HM97]{HesselholtMadsen97}
L.Hesselholt and I.Madsen.
\newblock On the {$K$}-theory of finite algebras over {W}itt vectors of perfect
  fields.
\newblock {\em Topology}, 36(1):29--101, 1997.

\bibitem[Hun96]{Hunter96}
T.~J.Hunter.
\newblock On the homology spectral sequence for topological {H}ochschild
  homology.
\newblock {\em Trans. Amer. Math. Soc.}, 348(10):3941--3953, 1996.

\bibitem[JW75]{JohnsonWilson75}
D.~C.Johnson and W.~S.Wilson.
\newblock {$BP$} operations and {M}orava's extraordinary {$K$}-theories.
\newblock {\em Math. Z.}, 144(1):55--75, 1975.

\bibitem[McC01]{McCleary01}
J.McCleary.
\newblock {\em A user's guide to spectral sequences}, volume~58 of {\em
  Cambridge Studies in Advanced Mathematics}.
\newblock Cambridge University Press, Cambridge, second edition, 2001.

\bibitem[Mil58]{Milnor58}
J.Milnor.
\newblock The {S}teenrod algebra and its dual.
\newblock {\em Ann. of Math. (2)}, 67:150--171, 1958.

\bibitem[MM65]{MilnorMoore65}
J.~W.Milnor and J.~C.Moore.
\newblock On the structure of {H}opf algebras.
\newblock {\em Ann. of Math. (2)}, 81:211--264, 1965.

\bibitem[MM02]{MandellMay02}
M.~A.Mandell and J.~P.May.
\newblock {\em Equivariant orthogonal spectra and {$S$}-modules}, volume 159.
\newblock 2002.

\bibitem[MMSS01]{MandellMaySchwedeShipley01}
M.~A.Mandell, J.~P.May, S.Schwede, and B.Shipley.
\newblock Model categories of diagram spectra.
\newblock {\em Proc. London Math. Soc. (3)}, 82(2):441--512, 2001.

\bibitem[MS93]{McClureStaffeldt93}
J.~E.McClure and R.~E.Staffeldt.
\newblock On the topological {H}ochschild homology of {$b{\rm u}$}. {I}.
\newblock {\em Amer. J. Math.}, 115(1):1--45, 1993.

\bibitem[MSV97]{McClureetAl97}
J.McClure, R.Schw{\"a}nzl, and R.Vogt.
\newblock {$THH(R)\cong R\otimes S^1$} for {$E_\infty$} ring spectra.
\newblock {\em J. Pure Appl. Algebra}, 121(2):137--159, 1997.

\bibitem[Rog14]{Rognes14}
J.Rognes.
\newblock Algebraic {$K$}-theory of strict ring spectra, 2014.

\bibitem[Shi00]{Shipley00}
B.Shipley.
\newblock Symmetric spectra and topological hochschild homology.
\newblock {\em K-theory}, 19(2):155--183, 2000.

\bibitem[Sto11]{Stolz11}
M.Stolz.
\newblock {\em Equivariant Structure on Smash Powers of Commutative Ring
  Spectra}.
\newblock PhD thesis, University of Bergen, 2011.

\end{thebibliography}

\end{document}